\documentclass[a4paper, 10pt, american]{amsart}
\usepackage{times,latexsym,amssymb}
\usepackage{amsmath,amsthm,bm}
\usepackage{color}
\usepackage[colorlinks,pdfpagelabels,pdfstartview = FitH,bookmarksopen
= true,bookmarksnumbered = true,linkcolor = blue,plainpages =
false,hypertexnames = false,citecolor = red,pagebackref=false]{hyperref}
\usepackage{mathrsfs}
\usepackage{yfonts}
\usepackage{braket}
\usepackage{bbm}
\hypersetup{linkcolor=blue,linktoc=page}
\usepackage{titletoc}
\usepackage{paralist}
\usepackage{pifont}
\usepackage{enumerate}
\usepackage{cancel}
\usepackage{microtype} 
\usepackage{amsbsy}
\usepackage{amstext}
\usepackage{amssymb}
\usepackage{esint}
\usepackage{stmaryrd}
\usepackage[pdftex]{graphicx}
\usepackage{floatflt}
\usepackage{comment}
\usepackage{schemata}
\usepackage{tikz}
\usetikzlibrary{decorations.pathreplacing,calc,tikzmark}
\usetikzlibrary{matrix,decorations.pathreplacing}
\usepackage{blindtext}
\usepackage[utf8]{inputenc}
\usepackage{amsthm}
\usepackage{graphicx}
\usepackage{mathtools}
\usepackage{stmaryrd}
\usepackage{bm}
\SetSymbolFont{stmry}{bold}{U}{stmry}{m}{n}
\newtheorem{mytheorem}{Theorem}
\newtheorem{mylem}[mytheorem]{Lemma}
\newtheorem{mydef}[mytheorem]{Definition}
\newtheorem{myproposition}[mytheorem]{Proposition}

\newtheorem{remark}[mytheorem]{Remark}
\numberwithin{mytheorem}{section}
\numberwithin{equation}{section}
\def\Yint#1{\mathchoice
    {\YYint\displaystyle\textstyle{#1}}%
    {\YYint\textstyle\scriptstyle{#1}}%
    {\YYint\scriptstyle\scriptscriptstyle{#1}}%
    {\YYint\scriptscriptstyle\scriptscriptstyle{#1}}%
      \!\iint}
\def\YYint#1#2#3{{\setbox0=\hbox{$#1{#2#3}{\iint}$}
    \vcenter{\hbox{$#2#3$}}\kern-.51\wd0}}
\def\longdash{{-}\mkern-3.5mu{-}} 
\def\fiint{\Yint\longdash}
\def\Xint#1{\mathchoice
{\XXint\displaystyle\textstyle{#1}}%
{\XXint\textstyle\scriptstyle{#1}}%
{\XXint\scriptstyle\scriptscriptstyle{#1}}%
{\XXint\scriptscriptstyle\scriptscriptstyle{#1}}%
\!\int}
\def\XXint#1#2#3{{\setbox0=\hbox{$#1{#2#3}{\int}$ }
\vcenter{\hbox{$#2#3$ }}\kern-0.555\wd0}}

\def\fint{\Xint-}
\makeatletter
\DeclareRobustCommand*{\bfseries}{%
  \not@math@alphabet\bfseries\mathbf
  \fontseries\bfdefault\selectfont
  \boldmath
}
\makeatother
\allowdisplaybreaks
\DeclareMathOperator*{\esssup}{ess\,sup}

\DeclareMathOperator*{\essosc}{ess\,osc}

\DeclareMathOperator\supp{spt}
\DeclareMathOperator\dist{dist}
\DeclareMathOperator\divv{div}
\DeclareMathOperator\loc{loc}
\DeclareMathOperator\p{p}

\newcommand{\foo}[1]{\mathbf{#1}}
\newcommand{\vertiii}[1]{{\left\vert\kern-0.25ex\left\vert\kern-0.25ex\left\vert #1 
    \right\vert\kern-0.25ex\right\vert\kern-0.25ex\right\vert}}
\newcommand{\bigchi}{\scalebox{1.3}{$\chi$}}

\renewcommand{\d}{\mathrm{d}}
\newcommand{\dx}{\mathrm{d}x}
\newcommand{\dy}{\mathrm{d}y}

\newcommand{\dt}{\mathrm{d}t}
\newcommand{\ds}{\mathrm{d}s}
\newcommand{\dtau}{\mathrm{d}\tau}
\newcommand{\R}{\mathbb{R}}
\newcommand{\N}{\mathbb{N}}
\renewcommand{\epsilon}{\varepsilon}
\newcommand{\critical}{\frac{2n}{n+2}}
\newcommand{\uproman}[1]{\uppercase\expandafter{\romannumeral#1}}

\usepackage{tikz}
\usetikzlibrary{patterns,patterns.meta}

\usepackage{tikz}
\usetikzlibrary{patterns,patterns.meta}
\usepackage{pgfplots}
\pgfplotsset{compat=1.11}
\usepgfplotslibrary{fillbetween}
\usetikzlibrary{intersections}

\usetikzlibrary{patterns}

\usepackage{graphicx}
\usepackage{dsfont}
\usepackage{esint}
\usepackage{xcolor}
\usepackage{caption}
\usepackage{subcaption}

\subjclass[2020]{35B65, 35B45, 35K55, 35K65}
\keywords{Doubly nonlinear parabolic PDE, Schauder estimates, Gradient regularity}

\begin{document}
\title[Gradient regularity for parabolic equations]{Gradient regularity for a class of doubly nonlinear parabolic partial differential equations}
\date{\today}

\author[M. Strunk]{Michael Strunk}
\address{Michael Strunk\\
Fachbereich Mathematik, Universit\"at Salzburg\\
Hellbrunner Str. 34, 5020 Salzburg, Austria}
\email{michael.strunk@plus.ac.at}

\begin{abstract} 
In this paper, we study the local gradient regularity of non-negative weak solutions to doubly nonlinear parabolic partial differential equations of the type
\begin{align*}
    \partial_t u^q - \divv A(x,t,Du)=0
    \qquad\mbox{in~$\Omega_T$},
\end{align*}
with~$q>0$,~$\Omega_T\coloneqq\Omega\times(0,T)\subset\R^{n+1}$ a space-time cylinder, and~$A=A(x,t,\xi)$ a vector field satisfying standard~$p$-growth conditions. Our main result establishes the local Hölder continuity of the spatial gradient of non-negative weak solutions in the super-critical fast diffusion regime 
$$0<p-1<q<\frac{n(p-1)}{(n-p)_+}.$$
This result is achieved by utilizing a time-insensitive Harnack inequality and Schauder estimates that are developed for equations of parabolic~$p$-Laplacian type. Additionally, we establish a local $L^\infty$-bound for the spatial gradient.
\end{abstract}
\maketitle
\tableofcontents
\section{Introduction and main results} \label{sec:introduction}
The parabolic partial differential equation
\begin{align} \label{pde}
    \partial_t u^q - \divv{A(x,t,D u)}=0
    \qquad\mbox{in~$\Omega_T$}, 
\end{align}
 with an arbitrary exponent~$q>0$, represents a generalized version of the prototype equation which involves the~$p$-Laplacian as the diffusion part and that is given by
\begin{align} \label{pdemitplaplace}
    \partial_t u^q - \divv{(|Du|^{p-2}Du)} = 0\qquad\text{in $\Omega_T$},
\end{align}
where~$p>1$. Here and in the following,~$\Omega_T=\Omega\times(0,T)$ denotes a space-time cylinder taken over a bounded domain~$\Omega\subset\R^n$ with~$n\geq 2$ and a finite time~$T>0$. Treating the general equation~\eqref{pde}, we assume the vector field~$A$ to satisfy the standard~$p$-growth structure conditions~\eqref{voraussetzungen} introduced in Section~\ref{subsec:assumptions} below. Based on a nonlinear behaviour in both, the term involving the time derivative as well well as the term that incorporates the spatial derivative, equations of type~\eqref{pde} are referred to as doubly nonlinear. For the prototype equation~\eqref{pdemitplaplace} some well known special cases arise. Only with the specific choice~$q=1$ and~$p=2$, it is linear and yields the heat equation. If~$q=p-1$, it is homogeneous with respect to multiplication and is sometimes called Trudinger's equation. In the case~$p=2$, we obtain the porous medium equation, whereas the case~$q=1$ yields the parabolic~$p$-Laplace equation.\,\\
The aim of this paper is to extend the regularity results obtained for the prototype equation~\eqref{pdemitplaplace} to the broader scope of the doubly nonlinear parabolic equation~\eqref{pde}, where the vector field~$A$ is assumed to satisfy standard~$p$-growth structure conditions. 
Indeed, this generalization turns out to hold true, resulting in a Lipschitz estimate for weak solutions within the space-time domain~$\Omega_T$ based on a pointwise value of the solution. Additionally, we obtain a Hölder estimate for the spatial gradient of weak solutions in~$\Omega_T$, where the Hölder exponent~$\alpha$, which lies within the range of~$(0,1)$, is dependent on the provided data. However, it is important to note that we are only able to provide statements regarding the regularity of the spatial gradient of solutions to~\eqref{pde}, as our notion of solution according to Definition~\ref{definitionglobal} does not encompass any weak time derivatives. 
\subsection{Literature overview} \label{subsec:literatur} 
In order to offer a comprehensive summary of existing findings regarding the regularity of solutions and their spatial derivative to~\eqref{pde} and~\eqref{pdemitplaplace}, we aim to provide a concise overview. Extensive research has been conducted on the local boundedness and local Hölder continuity of weak solutions to the prototype doubly nonlinear equation~\eqref{pdemitplaplace}. The investigation of local boundedness of solutions has first been started by Ivanov in his work~\cite{ivanov1997maximumloesungen}. In relation to Hölder regularity of weak solutions, the well-studied cases include the homogeneous case where~$q-1=p$, as well as the doubly degenerate range when~$q<p-1$ and~$p>2$. Additionally, the doubly singular range is examined when~$q>p-1$ and~$1<p<2$. For further details on known results in these cases, we recommend referring to~\cite{bogelein2021holder,bogelein2022holderpartii,ivanov1989holder,ivanov1995holder,ivanov1997holder,kuusi2012local,naianleahpartiii,porzio1993holder,urbano1930method,vespri2022extensive}. However, the state of affairs regarding gradient properties of weak solutions is currently quite fragmented, with limited knowledge available at present. In the realm of this subject, the parabolic~$p$-Laplace equation is the special case that has been thoroughly comprehended, as efforts to examine the gradient regularity of weak solutions to this equation were initiated by DiBenedetto and Friedman in their celebrated works~\cite{dibenedetto1985holder,dibenedetto1985addendum}. Indeed, solutions have been proven to possess local~$C^{1,\alpha}$ regularity. Despite extensive research, the porous medium equation remains considerably less understood in comparison. Regarding findings on the regularity of the gradient of solutions to the latter, we refer the reader to~\cite{aronson1969regularity,aronson1986optimal,benilan1983strong,caffarelli1987lipschitz,dibenedetto1979,dibenedetto1985holder,dibenedetto1985addendum,dibenedetto1991local,gianazza2023local,jin2022regularity}. For a more comprehensive overview of this subject, we also suggest referring to the introductory section in~\cite{gradientholder}, where further explanation on this topic is provided. \,\\

At this point in time, only few attempts concerning gradient regularity of weak solutions to equation~\eqref{pdemitplaplace}, and in particular to~\eqref{pde}, have been made in a wider parameter range. More recently though, new findings on the gradient regularity for weak solutions to~\eqref{pdemitplaplace} have been achieved in~\cite{gradientholder}. Initiated by Ivanov and Mkrtychyan, investigations into the local gradient boundedness have been made in~\cite{ivanov1992weighted}. A few years later, Ivanov also examined the local boundedness of the spatial gradient in~\cite[Theorem~6.1]{ivanov1999gradient} and the correction demonstrated in~\cite[\S5]{ivanov2000regularity}, where he treated non-negative weak solutions to the prototype equation~\eqref{pdemitplaplace} for a certain range of the parameters~$p$ and~$q$. In their work~\cite{savare1994asymptotic}, Savar\'{e} and Vespri introduced a gradient estimate that is similar but slightly weaker than the one presented in~\cite[Theorem~$1.1$]{gradientholder}, which we also derive in form of our Theorem~\ref{hauptresultat}.
However, the notions of a weak solution used in the aforementioned articles~\cite{ivanov1999gradient,ivanov2000regularity,ivanov1992weighted,savare1994asymptotic} all differ from our definition presented in~\eqref{definitionglobal}, thereby making it difficult to directly compare their investigations with ours. The most recent advancements in the study of gradient regularity of solutions to doubly nonlinear parabolic equations can be found in the paper by~\cite{gradientholder}. This article serves as the basis for the novel findings presented in our current work. Notably, the authors derived time-insensitive Harnack inequalities in the~\textit{super-critical fast diffusion regime}, where the parameters satisfy the relation~$0 < p-1 < q < \frac{n(p-1)}{(n-p)_+}$. This range of parameters aligns with the one previously examined by Savar\'{e} and Vespri. We recall that weak solutions to~\eqref{pdemitplaplace} display an infinite speed of propagation in the general fast diffusion regime~$q>p-1$, reminiscent of the behavior exhibited by the heat equation. These findings, stated in \cite[Theorem 1.10 \& Theorem~1.11]{gradientholder}, are particularly noteworthy because they do not comprise the typical waiting time phenomenon that is commonly observed in the parabolic setting.
Therefore, both Harnack inequalities possess an elliptic nature, which is a unique characteristic in this specific range of parameters. Additionally, by utilizing the aforementioned Harnack inequality in conjunction with appropriate Schauder estimates, the authors were able to establish a local~$L^{\infty}$-gradient bound. Moreover, they obtained the local Hölder continuity of the gradient of non-negative weak solutions to~\eqref{pdemitplaplace}. Consequently, it is natural to inquire whether their main regularity theorem is solely applicable to weak solutions of the prototype equation~\eqref{pdemitplaplace} or if it continues to hold true in a more general context of equation~\eqref{pde}.

\subsection{Structure conditions}\label{subsec:assumptions} 
Throughout this paper, we consistently assume the following set of structure conditions to hold true. We consider a vector field $$ A\colon \Omega_T\times\mathbb{R}^n \to \mathbb{R}^n,$$
such that the mapping~$\xi\mapsto A(x,t,\xi)$ is differentiable. Moreover, the maps~$A(x,t,\xi)$ as well as~$\partial_\xi A(x,t,\xi)$ are assumed to be Carath\'eodory functions, i.e.
\begin{align*}
    (x,t) &\mapsto A(x,t,\xi),\,\,(x,t)\mapsto\partial_{\xi}A(x,t,\xi)\,\, \mbox{are measurable}
\end{align*}
for every $\xi\in\R^n$, and
\begin{align*}
    \xi &\mapsto A(x,t,\xi),\,\,\xi\mapsto\partial_{\xi}A(x,t,\xi)\,\, \mbox{are continuous}
\end{align*}
for a.e.~$(x,t)\in \Omega_T$. It is notable that we do not impose any higher regularity than measurability for the mapping~$t\mapsto A(x,t,\xi)$. Furthermore, the vector field~$A(x,t,\xi)$ is assumed to satisfy the following standard~$p$-growth and ellipticity conditions
\begin{align}
    \left\{
    \begin{array}{l}
    | A(x,t,\xi)| + (\mu^2 + |\xi|^2 )^{\frac{1}{2}}|\partial_\xi A(x,t,\xi)| \leq C_1 (\mu^2 + |\xi |^{2})^{\frac{p-1}{2}} \\[3pt]
    \langle \partial_{\xi}A(x,t,\xi)\eta, \eta \rangle \geq C_2 (\mu^2 + |\xi|^2 )^{\frac{p-2}{2}}|\eta|^2 \\[3pt]
    |A_i (x,t,\xi)-A_i(y,t,\xi)| \leq C_3|x-y|^\alpha ( \mu^2 + |\xi|^2)^{\frac{p-1}{2}}
    \end{array}
    \right. \label{voraussetzungen}
\end{align}
for a.e.~$(x,t) \in \Omega_T$,~$i\in\{1,...,n\}$, any~$\eta, \xi \in\R^n$, and an Hölder exponent~$\alpha\in(0,1)$. Here, we let~$p>1$,~$\mu \in [0,1]$, while~$C_1, C_2, C_3$ represent positive constants. The parameter~$\mu\in[0,1]$ serves as a regularizing quantity that distinguishes between the degenerate resp.~singular case if~$\mu=0$ and the non-degenerate resp. non-singular case when~$\mu \in (0,1]$. Note that for the standard example we have in mind, which is the~$p$-Laplacian, conditions~$\eqref{voraussetzungen}_1$ and ~$\eqref{voraussetzungen}_2$ are satisfied in the case where~$\mu=0$ with constants~$C_1=C_1(p)=1+|p-2|$,~$C_2=C_2(p)= 1-(2-p)_+$, while condition~$\eqref{voraussetzungen}_3$ is satisfied regardless with~$C_3=0$. Moreover, the vector field incorporated in the~$p$-Laplacian may also contain some bounded coefficients~$a\in L^\infty(\Omega_T,\R)$ satisfying $\gamma^{-1}_1\leq a(x,t)\leq \gamma_1$ and~$a(x,t)-a(y,t)|\leq \gamma_2 |x-y|^\alpha$ a.e. in~$\Omega_T$, where~$\gamma_1,\gamma_2$ denote some positive constants and~$\alpha\in(0,1)$ is the Hölder exponent of~$a$ with respect to the spatial variable~$x$, i.e.~$A(x,t,\xi)=a(x,t)|\xi|^{p-2}\xi$. 

\subsection{Main result} 
Imposing the standard~$p$-growth structure conditions~\eqref{voraussetzungen} on the vector field~$A$, our main theorem establishes a local gradient bound as well as a local gradient Hölder estimate in a generic compact subset~$\mathcal{K} \Subset \Omega_T$. Both estimates are stated quantitatively in form of the following theorem, representing the main result of this article. 

\begin{mytheorem}\label{hölderstetigkeitohnegrößer0}
 Let ~$0<p-1<q<\frac{n(p-1)}{(n-p)_+}$ and~$u$ be a non-negative weak solution to~\eqref{pde} under assumptions~\eqref{voraussetzungen} with~$\mu=0$. Then, there exist~$\alpha_1=\alpha_1(n,p,q,C_1,C_2,\alpha)\in(0,1)$ and~$C=C(n,p,q,C_1,C_2,C_3,\alpha)>1$, such that: for any compact subset~$\mathcal{K}_1\Subset \Omega_T$, let
 \begin{align*}
     \rho_1 \coloneqq \inf_{\substack{(x,t)\in \mathcal{K}_1\\
    (y,s) \in\partial (\Omega_T)}}\big( |x-y| + |t-s|^{\frac{1}{p}}\big)>0.
 \end{align*}
 Moreover, we set
 \begin{align*}
     K_1 \coloneqq \max\Big\{1,\esssup\limits_{\mathcal{K}_1} u\Big\}, \quad
     K_2 \coloneqq \max\Big\{1, \esssup\limits_{\mathcal{K}_2} u\Big\},
 \end{align*}
 where~$\mathcal{K}_2\Supset\mathcal{K}_1$ is given by 
 \begin{align*}
     \mathcal{K}_2\coloneqq\Big\{(x,t)\in\Omega_T:~\inf\limits_{(y,s)\in\partial(\Omega_T)}  \big(|x-y|+|t-s|^{\frac{1}{p}}\big)\geq \frac{\rho_1}{2}\Big\}\Subset\Omega_T.
 \end{align*}
    Then, there hold the gradient bound
    \begin{align}\label{korollargradientbound}
        \esssup\limits_{\mathcal{K}_1}|Du| \leq C \frac{K_1 K_2^{\frac{q+1-p}{p}}}{\rho_1}
    \end{align}
    as well as the gradient Hölder estimate
        \begin{align}\label{korollarhölder}
 |Du(z_1)-Du(z_2)| \leq C\frac{ K_1 K_2 ^{\frac{q+1-p}{p}}}{\rho_1}\Bigg[K_2^{\frac{q+1-p}{p}}\frac{|x_1-x_2|}{\rho_1}+ \Big(\frac{K_2}{K_1}\Big)^{\frac{q+1-p}{2}} \sqrt{\frac{|t_1-t_2|}{\rho_1 ^p}}\,\Bigg]^{\alpha_1}
    \end{align}
    for any~$z_1=(x_1,t_1)$,~$z_2 = (x_2,t_2)\in \mathcal{K}_1$.
\end{mytheorem}


\begin{remark} \upshape
Our central result, Theorem~\ref{hölderstetigkeitohnegrößer0}, will be established in Section~\ref{sec:höldercontinuityofthegradient} as a consequence of Proposition~\ref{hauptresultat}. It is important to note that Proposition~\ref{hauptresultat} solely holds true within the super-critical fast diffusion regime, which is characterized by the condition~$0 < p - 1 < q < \frac{n(p-1)}{(n-p)_+}$. Consequently, our main regularity result, Theorem~\ref{hölderstetigkeitohnegrößer0}, is also valid specifically within this range of parameters. The Figure~\ref{fig:figure1} serves to highlight the region of parameters where Theorem~\ref{hölderstetigkeitohnegrößer0} is applicable.
\end{remark}


\begin{remark} \upshape
    According to~\cite[Corollary~1.7]{gradientholder}, weak solutions to equation~\eqref{pde} under our structure conditions~\eqref{voraussetzungen} are locally bounded in~$\Omega_T$, yielding that both quantities~$K_1$ and~$K_2$ in the previous Theorem~\ref{hölderstetigkeitohnegrößer0} are indeed finite.
\end{remark}


\begin{remark} \upshape
  We emphasize that the regularity result in Theorem~\ref{hölderstetigkeitohnegrößer0} is of local nature. Future research should focus on establishing the global regularity of weak solutions to~\eqref{pde} in~$\overline{\Omega}_T$, that is, regularity up to the lateral boundary~$\partial\Omega\times(0,T)$ of the domain. Currently, the only available boundary regularity result for solutions to~\eqref{pde} in a general doubly nonlinear setting is due to Gianazza and Jesus~\cite{gianazzaboundary}. They considered non-negative weak solutions to the model equation~\eqref{pdemitplaplace} and proved Hölder continuity up to the lateral boundary of the domain in the super-critical fast diffusion regime~$0 < p - 1 < q < \frac{n(p-1)}{(n-p)_+}$, assuming a smooth domain. Furthermore, they obtained a Carleson estimate and a boundary Harnack inequality for non-negative weak solutions to~\eqref{pdemitplaplace}. However, global gradient regularity remains elusive, and the situation for general doubly nonlinear equations of the form~\eqref{pde} is entirely unknown. We also note the global regularity result by Jin and Xiong~\cite{jin2022regularity}, who demonstrated global smoothness of non-negative weak solutions to the porous medium equation in the fast diffusion regime, given by the choice of parameters as~$p=2,q>1$.
\end{remark}


\begin{center}
\begin{figure}
\begin{tikzpicture}[scale=1.25]
\draw [<->,thick] (0,5.1) 
|- (6,0);

\draw[black,dotted,thick] (3,0.05)--(3,2.98);
\draw[black,dotted,thick] (3,4.82)--(3,5.1);

\fill [gray] (0,0).. controls (1,1.6) and (1.5,2.4)  
.. (2.2,4.8) -- (3,4.8);
        .. (3,4.8) -- (3,3);
        .. (3,3) -- (0,0);
\fill [gray] (3.00001,4.8)--(3.00001,3)--(2.2,3)--(2.2,4.8);
\fill [gray, opacity=0.3] (2.99999,2.99999)--(4.800002,4.800002)--(2.99999,4.800002)--(2.99999,2.99999);
\fill [gray] (0,0)--(3,3)--(3,4.8)--(1,1.67)--(0,0);
\draw (-0.05,0.3) node [left] {\footnotesize $q=0$};
\draw (0.5,-0.05) node [below] {\footnotesize $p=1$};
\draw (3,-0.05) node [below] {\footnotesize $p=n$};
\draw (4.75,4.2) node [below] {\footnotesize $q=p-1$};
\draw (2.0,4.3) node [left] {\footnotesize $q=\frac{n(p-1)}{(n-p)}$};
  \draw (5.9,-0.05) node [below] {$p$};
\draw (-0.05,5.0) node [left] {$q$};
 \draw[gray,opacity=0.0001,thin] (0,0)--(4.8,4.8);
\end{tikzpicture}
\caption{\label{fig:figure1}}
\end{figure}
\end{center}
\subsection{Strategy of the proof} 
Regarding the overall course of this paper, a few words are in order. The main focus is to establish the main regularity Theorem~\ref{hölderstetigkeitohnegrößer0}. One essential tool that we exploit on our path towards the main regularity result is the Harnack inequality of elliptic nature mentioned above that solely holds true in the super-critical fast diffusion regime~$0<p-1<q<\frac{n(p-1)}{(n-p)_+}$. This inequality allows us to properly handle the nonlinearity present in the evolution term of~\eqref{pde}. To specify, by employing the transformation~$v\coloneqq u^q$ in~\eqref{pde}, there holds~$Du = \frac{1}{q} v^{\frac{1-q}{q}} Dv$ a.e. in~$\Omega_T$. Now, due to the Harnack inequality mentioned above, there exists a universal constant~$\gamma>1$ that depends on~$n,p,q,C_1,C_2$, such that on an intrinsic cylinder~$\mathcal{Q}\subset\Omega_T$ not further specified at this point,~$v$ is bounded below by~$\gamma^{-1}$ and bounded above by~$\gamma$. Therefore, the coefficients~$a(x,t)\coloneqq \textstyle{\frac{1}{q}}v^{\frac{1-q}{q}}$ satisfy the bounds
$$C^{-1}\leq a(x,t)\leq C \quad\mbox{for a.e.~$(x,t) \in \mathcal{Q}$}$$
with~$C=C(n,p,q,C_1,C_2)$. As a result, we achieve that~$v$ is a weak solution to a parabolic equation of~$p$-Laplacian type and we are able to break down the original equation into the following version
\begin{equation} \label{transformpde}
\partial_t v - \divv \Tilde{A}(x,t,Dv)= 0 \quad\mbox{in~$\mathcal{Q}$},
\end{equation} 
where the vector field~$\Tilde{A}$ is given by
\begin{equation*}
    \Tilde{A}(x,t,\xi)\coloneqq A(x,t,a(x,t)\xi)
\end{equation*}
for a.e.~$(x,t)\in\mathcal{Q}$,~$\xi\in\R^n$. Later on in Section~\ref{sec:höldercontinuityofthegradient}, it will turn out that the vector field~$\Tilde{A}$ again satisfies the~$p$-growth conditions~\eqref{voraussetzungen} with~$\mu=0$, and with positive constants~$\Tilde{C}_1,\Tilde{C}_2,\Tilde{C}_3$ and an Hölder exponent~$\Tilde{\alpha}\in(0,1)$, where~$\Tilde{C}_1,\Tilde{C}_2$ depend on~$n,p,q,C_1,C_2$, while~$\Tilde{C}_3$ exhibits the same dependence but additionally also depends on~$C_3$. The Hölder exponent~$\Tilde{\alpha}\in(0,1)$ depends on~$n,p,q,C_1,C_2,\alpha$ but not on the structure constant~$C_3$. In particular, the evolutionary term no longer exhibits nonlinearity, which leads us to focus on developing Schauder estimates for weak solutions to parabolic~$p$-Laplacian type equations. As an initial step, our aim in Section~\ref{sec:gradientbound} is to demonstrate that any weak solution to equations of the form~\eqref{transformedpde} exhibits a spatial derivative that is locally bounded in~$\Omega_T$, i.e. satisfies~$|Du|\in L^\infty_{\loc}(\Omega_T)$. Throughout Section~\ref{sec:gradientbound}, structure assumption~$\eqref{voraussetzungen}_3$ is replaced by assuming that~$A$ is differentiable with respect to the spatial variable~$x$, allowing us to differentiate the equation. In the case where~$p>\critical$, previous works have established this result for the super-quadratic case~$p\geq 2$ in~\cite[Theorem 1.2]{bogeleinduzaarmarcellini} and for the sub-quadratic-super-critical case~$\critical<p<2$ in~\cite[Theorem 1.3]{singer2015parabolic}. However, at this point in time, the sub-critical range~$p\leq \critical$ remains unaddressed to our knowledge. Our results yield the very same quantitative estimates as those in~\cite[Proposition 9.1 \& Proposition 9.2]{gradientholder}, differing from the ones in~\cite[Theorem 1.2]{bogeleinduzaarmarcellini} and~\cite[Theorem 1.3]{singer2015parabolic}. Subsequently, we aim to establish an~\textit{a priori} gradient estimate for weak solutions of a more regular version of equation~\eqref{transformedpde}, where the vector field~$A$ is independent of the spatial variable~$x$, i.e.~$A=A(t,\xi)$. It should be noted that the authors in~\cite{gradientholder} also utilized a similar \textit{a priori} gradient regularity result, which they derived from~\cite[Theorem~1.3]{degeneratesystems}. However, unfortunately we are unable to employ the mentioned estimate from~\cite[Theorem 1.3]{degeneratesystems} in our setting as it applies to systems of Uhlenbeck structure. On the contrary, we consider equations without Uhlenbeck structure in general, and rather utilize the complete machinery involving both the non-degenerate and degenerate regime in order to ultimately obtain a similar~\textit{a priori} gradient estimate. This process necessitates a careful treatment and constitutes one of the major contributions of this article. However, it should be highlighted that in~\cite[Theorem~3.3]{kuusi2013gradient} in the sub-quadratic case and in~\cite[Theorem~3.2]{kuusi2014wolff} in the super-quadratic case, comparable~\textit{a priori} gradient regularity results were obtained using a different approach compared to our method for handling the degenerate regime. Instead of employing a De Giorgi-type expansion of positivity as in Section~\ref{subsubsec:degenerateregime}, the authors utilized a Harnack inequality for super-solutions to linear parabolic equations in the sub-quadratic case in~\cite{kuusi2013gradient}, whereas in the super-quadratic case they employed arguments based on logarithmic type arguments in~\cite{kuusi2014wolff}, leading to~\cite[Proposition~3.11]{kuusi2013gradient} resp.~\cite[Proposition~3.4]{kuusi2014wolff}. Furthermore, these approaches both necessitate an additional argument regarding the sign of~$D_i u$, where~$i=1,...,n$. In contrast, we were able to effectively treat the quantities~$|D_i u|$ in a cohesive manner by exploiting the fact that the function
\begin{equation*}
    \Big(|D_i u|^2 - \frac{\lambda^2_\mu}{4} \Big)^2_+
\end{equation*}
is a weak sub-solution to a linear parabolic equation and, up to a re-scaling of the function and the equation, consequently belongs to a parabolic De Giorgi class. Once the~\textit{a priori} gradient estimate result has been established, the gradient regularity result is then transferred to weak solutions of more general structure later on, leveraging the Hölder continuity of~$A$ with respect to the spatial variable~$x$. In particular, we subsequently consider vector fields~$A$ that potentially depend on the spatial variable~$x$. To specify, this process of transferring regularity is demonstrated in Section~\ref{subsec:schauderplaplacetype} through the estimates provided in equations~\eqref{comparisonabscheins} -- \eqref{schrankecomparisonfunction}. In order to conclude the Schauder estimates for parabolic equations of~$p$-Laplacian type, an approximation procedure is required to address the case when~$\mu=0$. Furthermore, the vector field~$A$ is mollified to ensure differentiability with respect to the spatial variable~$x$. \,\\
Throughout the paper, energy estimates involving second order spatial derivatives, which are commonly referred to as Caccioppoli type inequalities, are crucial. It should be noted that the majority of our energy estimates are only valid when~$\mu>0$, where~$\mu\in[0,1]$ represents the regularizing parameter from the set of assumptions~\eqref{voraussetzungen}. In the case where~$\mu=0$, the existence of second order spatial derivatives cannot be ensured in general, as the ellipticity of the equation breaks down at points~$\xi\in\R^n$ equal to zero. Due to the lack of structure in the vector field~$A$ in our setting, we are unable to exploit known results obtained for equations with Uhlenbeck structure, as established in~\cite[Proposition~4.1]{degeneratesystems}. Instead, we deviate from~\cite[Proposition~4.1]{degeneratesystems} and utilize the energy estimates derived in~\cite{bogeleinduzaarmarcellini,singer2015parabolic}. This approach naturally arises, as energy estimates in both of these previous articles provide counterparts to~\cite[Proposition~4.1]{degeneratesystems} in the context of parabolic~$p$-Laplacian type equations, where the Uhlenbeck structure is replaced by the standard~$p$-growth assumptions~$\eqref{voraussetzungen}_1$ --~$\eqref{voraussetzungen}_2$ and the assumed differentiability with respect to the spatial variable~$x$ of the vector field~$A$, stated in form of structure condition~$\eqref{voraussetzungendiffbarkeit}_3$. As a result, we obtain the same inequality as~\cite[Proposition~9.4]{gradientholder}, as well as energy estimates comparable to those in~\cite[Proposition~3.2]{degeneratesystems}.

\subsection{Plan of the paper} 
The paper is organized as follows: in Section~\ref{sec:prelim}, we will commence by presenting the notation and framework, which includes our notion of weak solutions to~\eqref{pde} and additional supplementary material required later on. As a preliminary step towards Theorem~\ref{hölderstetigkeitohnegrößer0}, we first treat weak solutions to the parabolic~$p$-Laplacian equation~\eqref{pdeqgleich1}, which does not exhibit a nonlinearity in the evolution term. Additionally, we assume that the vector field~$A$ is differentiable with respect to the spatial variable~$x$, in contrast to the general structural conditions~\eqref{voraussetzungen}. Under these assumptions~\eqref{voraussetzungendiffbarkeit} on the vector field~$A$, we obtain that any weak solution~$u$ to~\eqref{pdeqgleich1} admits a locally bounded spatial derivative in~$\Omega_T$, i.e. there holds~$|Du|\in L^\infty_{\loc}(\Omega_T)$. Furthermore, we are able to acquire local quantitative estimates for the essential supremum of the spatial derivative. The majority of this article is hereinafter dedicated to Schauder estimates for weak solutions to the parabolic~$p$-Laplacian type equation~\eqref{pdeqgleich1} treated in Section~\ref{sec:schauderestimates}. Thereby, we obtain an \textit{a priori} gradient Hölder estimate for weak solutions to~\eqref{pdeqgleich1}. Subsequently, employing comparison arguments together by taking the~\textit{a priori} estimate into account, we obtain appropriate Schauder estimates for weak solutions to~\eqref{pdeqgleich1}. The final Section~\ref{sec:höldercontinuityofthegradient} is then devoted to the proof of Proposition~\ref{hauptresultat} and subsequently also to the main regularity Theorem~\ref{hölderstetigkeitohnegrößer0}, where the time-insensitive Harnack inequality mentioned above is applied, followed by an exploitation of the Schauder estimates from the previous Section~\ref{sec:schauderestimates}. \,\\

\textbf{Acknowledgements.}
The author would like to express sincere gratitude to Professor Verena Bögelein for her guidance throughout the development of this article, and also to Professor Frank Duzaar for many valuable recommendations. In addition, the author would like to thank Naian Liao for his contributions in finalizing this article.  \\
This research was funded in whole or in part by the Austrian Science Fund (FWF) [10.55776/P31956]. For open access purposes, the author has applied a CC BY public copyright license to any author accepted manuscript version arising from this submission.

\section{Preliminaries}\label{sec:prelim}

\subsection{Notation} \label{subsec:notation} 
Throughout this article,~$\Omega_T = \Omega\times (0,T)$ denotes a space-time cylinder where~$\Omega\subset\mathbb{R}^n$ is a bounded domain and~$(0,T)$ represents a time interval for a certain time~$T>0$. The parabolic boundary of~$\Omega_T$ will be denoted by
$$
    \partial_p \Omega_T 
    = 
    \big(\overline\Omega\times\{0\}\big) \cup \big(\partial\Omega\times(0,T)\big).
$$
Given a function~$f\in L^1(\Omega_T)\cong L^1(0,T;L^1(\Omega))$, we shall also sometimes write~$f(t)$ instead of~$f(\cdot,t)$ when it is convenient. The expressions~\textit{gradient}, \textit{spatial gradient}, and \textit{spatial derivative} will all be used interchangeably to refer to the weak derivative~$Du$ with respect to the spatial variable~$x$, without distinguishing between them. According to this notion, the weak partial derivative of any weakly differentiable function~$u\colon\Omega_T\to\R$ with respect to~$i\in\{1,...,n\}$ will be denoted by~$D_i u$. To distinguish between spatial and time derivatives, the latter will be denoted as~$\partial_t u$, given that~$u\colon\Omega_T\to\R$ admits a weak derivative with respect to the time variable~$t$. It is worth noting that no distinction between classical and weak derivatives will be made, with the context clarifying the intended meaning.
Moreover, we will regard the Euclidean norm~$|\cdot|$ on $\mathbb{R}^n$ for~$n\geq 2$ and the absolute value~$|\cdot|$ on~$\R$ as interchangeable. In this paper, both will be referred to as~$|\cdot|$ with the specific context providing clarity regarding its denotation. At certain stages we will also employ the maximum norm~$|x|_{\infty}\coloneqq \max\limits_{i=1,...,n}|x_i|$ on~$\R^n$ and recall the equivalence of norms
$$|x|_\infty \leq |x|\leq \sqrt{n}|x|_\infty$$
for any~$x\in\R^n$. The standard scalar product on~$\R^n$ will be denoted by~$\displaystyle\langle \cdot\rangle$ and the dyadic product of two vectors~$\xi,\eta \in\R^n$ by~$\xi \otimes \eta$. The positive part of a real quantity~$a\in\R$ is denoted as~$a_+ = \max\{a,0\}$, while the negative part is denoted as~$a_- = \max\{-a,0\}$. Constants are consistently represented as~$C$ or~$C(\cdot)$ and their dependence is described solely on their variables, not their specific values. However, it is possible for constants to vary from one line to another without further clarification.
\,\\

We define the open ball in~$\mathbb{R}^n$ with radius~$\rho>0$ and center~$x_0\in\mathbb{R}^n$ as 
$$B_\rho(x_0) \coloneqq\{x\in\mathbb{R}^n\colon~|x-x_0|<\rho\}.$$
By a point~$z_0\in\R^{n+1}$, we always refer to~$z_0 = (x_0,t_0)$ where~$x_0\in\R^n$ and~$t_0\in\R_{\geq 0}$. If~$x_0=0$, it will be common to simplify notation by writing~$B_\rho$ instead of~$B_\rho(x_0)$.  
Additionally, we may omit the vertex of the cylinder for convenience. The \textit{general (backward) parabolic cylinder} with vertex~$z_0\in\R^{n+1}$ is defined as 
$$Q_{R,S}(z_0) \coloneqq B_R(x_0)\times (t_0-S,t_0],$$
 with the standard parabolic cylinder given by
$$Q_\rho(z_0)\coloneqq B_\rho(x_0)\times (t_0-\rho^2,t_0].$$
The following~\textit{(backward) intrinsic parabolic cylinders} lead to homogeneous estimates that incorporate the spatial gradient~$Du$. Given~$\lambda>0$, we set
$$Q^{(\lambda)}_\rho(z_0) \coloneqq B_\rho(x_0) \times \Lambda^{(\lambda)}_\rho (t_0) $$
with
$$\Lambda^{(\lambda)}_\rho (t_0) \coloneqq (t_0-\lambda^{2-p}\rho^2,t_0].$$
Next, we will introduce different concepts of distances. Consider two points~$z_1=(x_1,t_1)$,~$z_2 = (x_2,t_2) \in \Omega_T$. The \textit{parabolic distance} between~$z_1$ and~$z_2$ is defined as
$$d_{p}(z_1,z_2) \coloneqq |x_1-x_2| + \sqrt{|t_1-t_2|},$$
while the \textit{intrinsic parabolic distance} is denoted by
$$d^{(\lambda)}_{p}(z_1,z_2) \coloneqq |x_1-x_2| + \sqrt{\lambda^{p-2}|t_1-t_2|}.$$
Finally, for a subset~$E\subset \Omega_T$, the parabolic distance from~$E$ to the parabolic boundary~$\partial_p \Omega_T$ is defined as 
$$\dist_{p}(E,\partial_p \Omega_T) \coloneqq \inf_{\substack{z_1\in E\\ z_2 \in\partial_p \Omega_T}}d_{\p}(z_1,z_2).$$
Since our notion of solution according to Definition~\eqref{definitionglobal} does not involve spatial weak derivatives of second order, it becomes useful to utilize the customary difference quotient technique. In the case of $\mu\in(0,1]$, this approach will yield the existence and~$p$-integrability of second order spatial derivatives. Let~$f\in L^1_{\loc}(\Omega_T,\R^k)$,~$i\in\{1,...,n\}$,~$k\in\N$, and~$h\neq 0$. We define the difference quotient of~$f$ for a.e.~$(x,t)\in\Omega_T$ in the spatial direction~$e_i$ by
\begin{align}\Delta_h^{(i)}f(x,t) \coloneqq \frac{f(x+he_i,t)-f(x,t)}{h},\label{differenzenquotient}
\end{align}
where~$|h|>0$ is chosen sufficiently small, such that~$(x+he_i)\in\Omega_T$ holds true. Let~$\mu\in[0,1]$. The following functions will turn out to be expedient in the course of this paper: 
\begin{align} \label{expfunctions}
h_\mu(s) &\coloneqq (\mu^2+s^2)^{\frac{p-2}{2}},\\ 
\mathcal{V}_\mu^{(p)}(\xi) &\coloneqq(\mu^2+|\xi|^2)^{\frac{p-2}{4}}\xi = h_{\mu}(|\xi|)^{\frac{1}{2}}\xi, \nonumber \\
\mathcal{D}_{\mu}^{(i)}(h)(x,t) &\coloneqq \mu^2 + |Du(x,t) |^2+|Du(x+he_i,t) |^2 \nonumber
\end{align}
for~$\xi\in\R^n$,~$s\in\R_{\geq 0}$, $(x,t)\in\Omega_T$, and~$|h|>0$ small enough. Further, for~$\xi,\eta,\zeta\in\R^n$ the following bilinear form on~$\R^n$ will be employed
\begin{equation} \label{bilineardef}
    \foo{\mathcal{B}}(\cdot,\xi)(\eta,\zeta)\coloneqq \langle \partial_\xi A(\cdot,\xi)\eta,\zeta\rangle.
\end{equation}
For any~$f\in L^1(\Omega_T,\R^k)$, we use the common abbreviation for the average integral of~$f$. Let~$A\subset\Omega$ such that~$\mathcal{L}^n(A)>0$. Then, we define the slice-wise mean~$(f)_A\colon (0,T)\to\R^n$ as follows
\begin{align*} (f)_A(t) \coloneqq \displaystyle\fint_A f(\cdot,t),\dx,\quad\mbox{for a.e.~$t\in(0,T)$}.
\end{align*}
Since most of the functions in this paper are continuous with respect to the time variable, that is,~$f\in C(0,T;L^1(\Omega))$, the slice-wise mean value is well-defined for all~$t\in(0,T)$. In a similar fashion, we denote the mean value of~$f$ on some measurable set~$E\subset \Omega_T$, such that~$\mathcal{L}^{n+1}(E)>0$, by
\begin{align*}
     (f)_E \coloneqq \displaystyle\fiint_E f(x,t)\,\dx\dt.
\end{align*}
If~$E$ is an intrinsic parabolic cylinder of the form~$Q^{(\lambda)}_{\rho}(z_0)$, then we abbreviate the mean value of~$f$ on~$Q^{(\lambda)}_{\rho}(z_0)$ by~$(f)^{(\lambda)}_{z_0,\rho}$, meaning
\begin{align*}
     (f)^{(\lambda)}_{z_0,\rho} \coloneqq \displaystyle\fiint_{Q^{(\lambda)}_{\rho}(z_0)} f(x,t)\,\dx\dt.
\end{align*}
If it is evident from the context, the vertex~$z_0$ may sometimes be omitted to simplify notation.

\subsection{Definition of weak solution}\label{subsec:def-weak} 
While it is standard, we subsequently state the employed definition of a weak solution in our paper.
\begin{mydef}[Weak solution]\label{definitionglobal}
Let~$q>0$,~$p>1$. A non-negative measurable function~$u\colon\Omega_T\to \R_{\ge 0}$ in the class  
$$
    u\in C\big([0,T];L^{q+1}(\Omega)\big) \cap L^{p}\big(0,T;W^{1,p}(\Omega)\big)
$$ 
under assumptions~\eqref{voraussetzungen} is a non-negative weak sub(super)-solution of \eqref{pde}, if 
\begin{align}\label{lösungglobal}
    \iint_{\Omega_T}\big[-u^q\partial_t\phi + \langle A(x,t,Du), D\phi\rangle \big]\,\dx\dt  
    \leq (\geq)\, 0 
\end{align}
for any non-negative function
$$
    \phi \in W^{1,q+1}_{0}\big(0,T;L^{q+1}(\Omega)\big) \cap L^p\big(0,T;W^{1,p}_0(\Omega)\big).
$$
A non-negative function~$u$ is a non-negative weak solution of \eqref{pde}, if it is both, a weak sub-solution and a weak super-solution. 
\end{mydef}
This ensures the convergence of all the integrals in~\eqref{lösungglobal}. According to our Definition~\ref{definitionglobal}, weak sub/super-solutions belong to the space
$$C\big([0,T];L^{q+1}(\Omega)\big) \cap L^{p}\big(0,T;W^{1,p}(\Omega)\big),$$
meaning they are continuous functions with respect to the time variable for a.e.~$x\in\Omega$. However, this property is not restrictive, as pointed out in~\cite[Proposition~4.9]{gradientholder}. We will now state the notion of a weak solution to the following Cauchy-Dirichlet problem of equation~\eqref{pde}:
\begin{align} \label{cauchydirichletproblem}
    \begin{cases}
        \partial_t u^q - \divv A(x,t,Du) = 0 & \mbox{in $\Omega_T$},\\
        u=g & \mbox{on $\partial\Omega\times(0,T)$},\\
        u(\cdot,0) = u_0 & \mbox{in $\Omega$},
    \end{cases}
\end{align}
where~$g\in L^p\big(0,T;W^{1,p}(\Omega)\big)$ with~$\partial_t g \in L^{p'}\big(0,T;W^{-1,p'}(\Omega)\big)$ and~$u_0\in L^{q+1}(\Omega)$. Here,~$p' \coloneqq \frac{p}{p-1}$ in the usual manner denotes the Hölder conjugate of~$p$.
\begin{mydef}[Weak solution to the Cauchy-Dirichlet problem]\label{definitioncauchydirichletproblem}
Let~$q>0$,~$p>1$. A non-negative measurable function~$u\colon\Omega_T\to \R_{\ge 0}$ in the class  
$$
    u\in C\big([0,T];L^{q+1}(\Omega) \big) \cap\big(g+ L^{p}\big(0,T;W^{1,p}_0(\Omega)\big)\big)
$$ 
under assumptions~\eqref{voraussetzungen} is a non-negative weak solution to the Cauchy-Dirichlet problem~\eqref{cauchydirichletproblem}, if 
\begin{align}\label{lösungcauchydirichletproblem}
    \iint_{\Omega_T}\big[(u_0^q-u^q)\partial_t\phi + \langle A(x,t,Du), D\phi\rangle \big]\,\dx\dt  
    = 0 
\end{align}
for any non-negative function
$$
    \phi \in W^{1,q+1}\big(0,T;L^{q+1}(\Omega)\big) \cap L^p\big(0,T;W^{1,p}_0(\Omega)\big),
$$
such that~$\phi(\cdot,T)\equiv 0$.
\end{mydef}
According to this definition, if~$u$ is a weak solution to the Cauchy-Dirichlet problem~\eqref{cauchydirichletproblem}, then it can be inferred from~\cite[Proposition~4.9]{gradientholder} that~$u(\cdot,0)=u_0$ a.e. in~$\Omega$. The existence of weak solutions to the Cauchy-Dirichlet problem~\eqref{cauchydirichletproblem} has been established in~\cite{alt1983quasilinear}.
\subsection{Mollification in time} \label{subsec:mollificationintime} 
According to our definition, weak solutions may not necessarily exhibit weak differentiability with respect to the time variable. To address this challenge, we commonly resort to specific regularization techniques. In this regard, we will utilize Steklov-averages, referring to~\cite{steklovproperties} for further insights into their properties. For a function~$f \in L^1(\Omega_T)$ and~$0<h<T$, we define its {\it Steklov-average}~$[f]_{ h}$ by 
\begin{equation}\label{def-stek-right}
	[f]_{h}(x,t) 
	\coloneqq
	\left\{
	\begin{array}{cl}
		\displaystyle{\frac{1}{h} \int_{t}^{t+h} f(x,\tau) \,\d\tau ,} 
		& t\in (0,T-h) , \\[9pt]
		0, 
		& t\in [T-h,T)
\,.
	\end{array}
	\right.
\end{equation}
Rewriting inequality~\eqref{lösungglobal} in terms of Steklov-means~$[u]_h$ of~$u$, yields 
\begin{align}
    \int_{\Omega\times\{t\}}\big[\partial_t [u^q]_h \phi + \langle [A(x,t,Du)]_h, D\phi\rangle\big]\,\dx
    \leq (\geq)\, 0 \label{lösung-steklov}
\end{align}
for any non-negative function~$\phi \in W^{1,p}_0(\Omega)$ and any~$t\in(0,T)$. 

\subsection{Auxiliary material}\label{subsec:lemmas} 
 The aim of this section is to present several useful preliminary results. As the majority of the subsequent material is commonly known, we will largely omit providing proofs. 
 \subsubsection{Difference quotients} \label{subsubsec:diffquot}\,\\
 The following initial lemmas recall standard estimates for the difference quotients, as introduced in~\eqref{differenzenquotient}. To state a reference, we refer the reader to~\cite[Chapter~5.8,~Theorem~3]{evans2022partial}.
 \begin{mylem} \label{differenzenquotientlemmaeins}
      Let~$1<p<\infty$,~$i\in\{1,...,n\}$,~$0<r<\rho$, and~$D_iu\in L^p(B_\rho)$. Then, there holds
      $$\lim\limits_{\rho-r>|h|\downarrow 0}\big\|\Delta_h^{(i)}u - D_iu \big\|_{L^p(B_r)} = 0.$$
 \end{mylem}
  \begin{mylem} \label{differenzenquotientlemmazwei}
     Let~$1<p<\infty$,~$i\in\{1,...,n\}$,~$0<r<\rho$, and~$u\in L^p(B_\rho)$. Furthermore, suppose there exists a constant~$K\geq 0$, such that
     $$\displaystyle\int_{B_r}\big|\Delta_h^{(i)}u \big|^p\,\dx \leq K,\quad \mbox{for any $|h|<\rho-r$}.$$
     Then, there holds~$D_i u \in L^p(B_\rho)$ including the quantitative estimate
     $$\displaystyle\int_{B_\rho}|D_iu|^p\,\dx \leq K.$$
 \end{mylem}
 \begin{mylem} \label{differenzenquotientlemmadrei}
     Let~$1<p<\infty$,~$i\in\{1,...,n\}$,~$0<r<\rho$,~$|h|<\rho-r$, and~$u\in L^p(B_\rho)$. Then, there holds
     $$\displaystyle\int_{B_r}|u(x+he_i)|^p\,\dx \leq \displaystyle\int_{B_\rho}|u|^p\,\dx.$$
     Furthermore, if~$D_iu\in L^p(B_\rho)$, then
     $$\displaystyle\int_{B_r}\big|\Delta_h^{(i)}u \big|^p\,\dx \leq \displaystyle\int_{B_\rho}|D_i u|^p\,\dx.$$
 \end{mylem}
  
The subsequent lemma presents a sort of discrete integration by parts.
\begin{mylem}\label{diskretePI}
    Let~$k\in\N$. Further, let~$F\in L^1(\Omega,\R^k)$,~$\phi \in C^1_0(\Omega,\R^k)$, and~$0<|h|<\dist(\supp\phi,\partial\Omega)$. For any~$i\in\{1,...,n\}$, there holds
    \begin{align*}
        \displaystyle\int_\Omega \big\langle \Delta_h^{(i)} F, \phi\big\rangle\,\dx = - \displaystyle\int_\Omega\int_0^1 \langle F(x+hs e_i), D_i\phi(x)\rangle \,\ds\dx.
    \end{align*}
\end{mylem}

\begin{proof}
First, we consider~$F\in C^1(\Omega,\R^k)$. For~$x\in\supp \phi$, we obtain
\begin{align*}
    \Delta^{(i)}_h F(x) &= \frac{1}{h}\displaystyle\int_{0}^{1}\frac{\d}{\ds} F(x+hse_i)\,\ds = \displaystyle\int_{0}^{1} D_i F(x+hse_i)\,\ds.
\end{align*}
This implies
\begin{align*}
    \displaystyle\int_{\Omega} \big\langle \Delta^{(i)}_h F(x),\phi(x)\big\rangle\,dx &= \displaystyle\int_{\Omega}\int_{0}^{1} \langle D_i F(x+hse_i), \phi(x)\rangle\,\ds\dx \\
    &= \displaystyle\int_{0}^{1}\int_{\Omega}\langle D_i F(x+hse_i), \phi(x)\rangle\,\dx\ds \\
    &= -\displaystyle\int_{0}^{1}\int_{\Omega}\langle F(x+hse_i), D_i\phi(x)\rangle\,\dx\ds \\
    &= -\displaystyle\int_{\Omega}\int_{0}^{1}\langle F(x+hse_i), D_i\phi(x)\rangle\,\ds\dx.
\end{align*}
Given~$F\in L^1(\Omega,\R^k)$, we consider a sequence~$(F_k)_{k\in\N}\subset C^1(\Omega,\R^k)$ with
$$\lim\limits_{k\to\infty}\displaystyle\int_{\Omega}|F-F_k|\,\dx = 0.$$
Then, there hold
\begin{align*}
    \lim\limits_{k\to\infty}\bigg|\displaystyle\int_{\Omega} \big\langle \Delta^{(i)}_h F, \xi\big\rangle\,dx - \displaystyle\int_{\Omega} \big\langle \Delta^{(i)}_h F_k, \xi\big\rangle\,dx\bigg| \leq \lim\limits_{k\to\infty} \frac{2\left\|\phi\right\|_\infty}{|h|} \displaystyle\int_{\Omega} |F-F_k|\,\dx = 0
\end{align*}
as well as
\begin{align*}
    \lim\limits_{k\to\infty}&\bigg|\displaystyle\int_{\Omega} \left\langle F(x+hse_i), D_i \right\rangle\xi\,dx - \displaystyle\int_{\Omega} \langle F_k(x+hse_i), D_i\xi \rangle\,dx\bigg| \\
    &\leq \lim\limits_{k\to\infty} \left\|D\phi\right\|_\infty \displaystyle\int_{\Omega} |F-F_k |\,\dx = 0.
\end{align*}
This finishes the proof.
\end{proof}

\subsubsection{Structure estimates} \label{subsubsec:mu} \,\\
In this section, we will present several commonly used estimates that primarily incorporate the regularizing parameter~$\mu\in[0,1]$. For the next lemma, we refer to~\cite[Lemma~2.1]{giaquinta1986partial} in the case~$\alpha\geq 0$, and to~\cite[Lemma~2.1]{acerbi1989regularity} in the case~$-\frac{1}{2}<\alpha<0$.
\begin{mylem} \label{integralungleichung}
    Let~$-\frac{1}{2}<\alpha$ and~$\mu\in [0,1]$. Then, there holds
    \begin{align*}
        \frac{1}{C}(\mu^2 + |\xi|^2 + |\eta|^2 )^\alpha \leq \displaystyle\int_{0}^{1}(\mu^2 + |\xi + s(\eta-\xi) |^2 )^\alpha\,\ds \leq C(\mu^2 + |\xi|^2 + |\eta|^2 )^\alpha
    \end{align*}
    for any~$\xi,\eta\in\R^n$ with~$C=C(\alpha)>0$.
\end{mylem}

As an immediate consequence of the preceding Lemma~\ref{integralungleichung}, we obtain the following structure properties for the vector field~$A$.
\begin{mylem} \label{AbschätzungenfürAeins}
Let~$p>1$,~$\mu\in(0,1]$, and~$A\colon\Omega_T\times\R^n\to\R^n$ satisfy~$\eqref{voraussetzungen}_1 - \eqref{voraussetzungen}_2$. For any~$(x,t)\in\Omega_T$ and~$\xi,\eta\in\R^n$, there hold
\begin{align*}
    | A(x,t,\xi)-A(x,t,\eta)| \leq C (\mu^2 + |\xi|^2 + |\eta|^2 )^{\frac{p-2}{2}}|\xi-\eta|
\end{align*}
with~$C=C(p,C_1)$, and
\begin{align*}
    \langle A(x,t,\xi)-A(x,t,\eta), \xi-\eta\rangle \geq \frac{1}{C} (\mu^2 + |\xi|^2 + |\eta|^2 )^{\frac{p-2}{2}}|\xi-\eta|^2
\end{align*}
with~$C=C(p,C_2)$. The estimates remain valid, if either~$\xi$ or~$\eta$ is non-zero, even when~$\mu$ equals zero and~$1<p<2$.
\end{mylem}
\begin{proof}
    If~$\xi=\eta$, there is nothing to prove. Hence, we assume~$\xi\neq \eta$. We start with the first estimate. By utilizing~$\eqref{voraussetzungen}_{1}$ and Lemma~\ref{integralungleichung}, we obtain
    \begin{align*}
        | A(x,t,\xi)-A(x,t,\eta)| &\leq \displaystyle\int_{0}^{1}\Big|\frac{\d}{\ds} A(x,t,\eta+s(\xi-\eta))\Big|\,\ds \\ 
        & = \displaystyle\int_{0}^{1}|\partial_{\xi} A(x,t,\eta+s(\xi-\eta))||\xi-\eta|\,\ds \\
                &\leq C_1 \displaystyle\int_{0}^{1}(\mu^2 + |\eta+s(\xi-\eta)|^2)^{\frac{p-2}{2}}\,\ds \, |\xi-\eta|\\
        &\leq C (\mu^2 + |\xi|^2 + |\eta|^2 )^{\frac{p-2}{2}}|\xi-\eta|.
    \end{align*}
The second inequality follows from~$\eqref{voraussetzungen}_{2}$ and again Lemma~\ref{integralungleichung}
     \begin{align*}
        \langle A(x,t,\xi)-A(x,t,\eta), \xi-\eta \rangle &= \displaystyle\int_{0}^{1}\Big\langle\frac{\d}{\ds} A(x,t,\eta+s(\xi-\eta)),\xi-\eta\Big\rangle\,\ds \\
        & = \displaystyle\int_{0}^{1}\langle\partial_\xi A(x,t,\eta+s(\xi-\eta))(\xi-\eta),\xi-\eta\rangle\,\ds  \\
        &\geq C_2\displaystyle\int_{0}^{1}(\mu^2 + |\eta+s(\xi-\eta)|^2)^{\frac{p-2}{2}}\,\ds\, |\xi-\eta|^2 \\
        &\geq \frac{1}{C} (\mu^2 + |\xi|^2 + |\eta|^2 )^{\frac{p-2}{2}}|\xi-\eta|^2.
    \end{align*}
\end{proof}
Exploiting condition~$\eqref{voraussetzungen}_{3}$, we derive another useful lemma.

\begin{mylem} \label{AbschätzungenfürAzwei}
Let~$p>1$,~$\mu\in[0,1]$, and~$A\colon\Omega_T\times\R^n\to\R^n$ satisfy
$$|\partial_x A(x,t,\xi)|\leq C (\mu^2+|\xi|^2)^{\frac{p-1}{2}}$$
for a.e.~$(x,t)\in\Omega_T$,~$\xi\in\R^n$, with~$C\geq 0$. For any~$t\in(0,T)$,~$x,y\in\Omega$ with~$[x,y]\subset\Omega$, where
\begin{equation*}
    [x,y]\coloneqq\{v\in\Omega\colon~v=sx + (1-s)y,\,s\in[0,1] \},
\end{equation*}
and any~$\xi,\eta\in\R^n$, there holds
\begin{align*}
    |\langle A(x,t,\xi)-A(y,t,\xi),\eta \rangle| \leq C  (\mu^2 + |\xi|^2 )^{\frac{p-1}{2}}|x-y| |\eta|
\end{align*}
with~$C=C(n,C_3)$.
\end{mylem}
\begin{proof}
    An application of the mean value theorem yields the existence of~$z\in[x,y]$ with
    $$A(x,t,\xi)-A(y,t,\xi)=\langle \partial_x A(z,t,\xi),x-y\rangle$$
    for a.e.~$t\in(0,T)$,~$\xi\in\R^n$. Employing~$\eqref{voraussetzungen}_{3}$ and the Cauchy-Schwarz inequality, we thus immediately obtain the claimed estimate
    \begin{align*}
        | \langle A(x,t,\xi)-A(y,t,\xi),\eta\rangle| &\leq |A(x,t,\xi)-A(y,t,\xi)| |\eta| \\
        &\leq | \partial_x A(z,t,\xi)||x-y||\eta| \\
        &\leq n C_3 (\mu^2 + |\xi|^2 )^{\frac{p-1}{2}} |x-y||\eta|.
    \end{align*}
\end{proof}
A further beneficial lemma is the following, which can be found in~\cite[Lemma~2.4]{singer2015parabolic}.
\begin{mylem} \label{lemmafurenergyestimate}
    Let~$p>1$ and~$\mu\in[0,1]$. For any~$\xi,\eta\in\R^n$, there holds
    \begin{align*}
       (\mu^2+|\xi|^2)^{\frac{p}{2}} \leq C (\mu^2+|\eta|^2)^{\frac{p}{2}} + C(\mu^2 + |\xi|^2 + |\eta|^2 )^{\frac{p-2}{2}}|\xi-\eta|^2
    \end{align*}
    with~$C=C(n,p)$.
\end{mylem}

The subsequent lemma originates from~\cite[Lemma~2.2]{acerbi1989regularity}.
\begin{mylem}\label{Vabschätzung}
 Let~$p\in(1,2)$ and~$\mu\in(0,1]$. There holds
\begin{align*}
    \frac{1}{C} \frac{|\xi-\eta|}{(\mu^2+|\xi|^2+|\eta|^2)^{\frac{2-p}{4}}} \leq \big|\mathcal{V}_{\mu}^{(p)}(\xi)-\mathcal{V}_{\mu}^{(p)}(\eta) \big| \leq C \frac{|\xi-\eta|}{(\mu^2+|\xi|^2+|\eta|^2)^{\frac{2-p}{4}}}
\end{align*}
  for any~$\xi,\eta\in\R^n$, and~$C=C(n,p)$. The estimate remains valid, if either~$\xi\neq 0$ or~$\eta\neq 0$, even when~$\mu=0$.
\end{mylem}
The following lemma presents a version of Kato's inequality.
\begin{mylem} \label{lem-kato}
    Let~$k\in \N$. For any~$u\in W^{2,1}_{\loc}(\Omega,\R^k)$ there holds
    \begin{equation*}
        |D|Du||\leq |D^2 u|\quad\mbox{a.e. in $\Omega$}.
    \end{equation*}
\end{mylem}
We refer to~\cite{degeneratesystems} for the subsequent lemma.
\begin{mylem} \label{Vfunktionableitung}
    Let~$p>1$ and~$\mu\in(0,1]$. For any~$u\in W^{2,2}_{\loc}(B_R(z_0),R^k)$ and any~$B_R(z_0)\subset \R^n$, there holds
    \begin{align*}
        C_1(p) h_\mu(|Du|)|D^2u|^2 \leq \big| D\mathcal{V}^{(p)}_{\mu}(Du)\big|^2 \leq C_2(p) h_\mu(|Du|)|D^2u|^2\quad\mbox{a.e. in $B_R(z_0)$}.
    \end{align*}
\end{mylem}
\begin{proof}
The second inequality follows from~\cite[Lemma~2.3]{degeneratesystems}. Hence, we only establish the first inequality. For any~$i\in\{1,...,n \}$ we compute
\begin{align*}
    |D_i\mathcal{V}^{(p)}_\mu(Du)|^2 &= \bigg| D_i Du \sqrt{h_\mu(|Du|)} + D_i|Du| Du \frac{h'_\mu(|Du|)}{2\sqrt{h_\mu(|Du|)}} \bigg|^2 \\
    &= |D_i Du|^2 h_\mu(|Du|) + |D_i|Du||^2\bigg[h'_\mu(|Du|)|Du|+\frac{h'_\mu(|Du|)^2|Du|^2}{4h_\mu(|Du|)} \bigg]
\end{align*}
a.e.~in $B_R(z_0)$. After summing over~$i=1,...,n$, we obtain the following
\begin{equation*}
    |D\mathcal{V}^{(p)}_\mu(Du)|^2 = h_\mu(|Du|)|D^2u|^2 + |D|Du||^2\bigg[h'_\mu(|Du|)|Du|+\frac{h'_\mu(|Du|)^2|Du|^2}{4h_\mu(|Du|)} \bigg].
\end{equation*}
Now, if~$p\geq 2$, the quantity in brackets is positive and we have established the desired inequality. In the case~$p<2$, the term in brackets may be negative. If the term in brackets in negative, by utilizing Kato's inequality from Lemma~\ref{lem-kato}, we estimate
\begin{equation*}
    |D\mathcal{V}^{(p)}_\mu(Du)|^2 \geq h_\mu(|Du|)|D^2u|^2 \bigg[1+ \frac{h'_\mu(|Du|)^2|Du|}{h_\mu(|Du|)}+\frac{h'_\mu(|Du|)^2|Du|^2}{4h_\mu(|Du|)^2}\bigg].
\end{equation*}
Further, the term in brackets in the preceding estimate may be bounded below by exploiting the fact that~$1<p<2$, which yields
\begin{align*}
    & 1+\frac{h'_\mu(t)t}{h_\mu(t)} + \frac{h'_\mu(t)^2t^2}{4h_\mu(t)^2} = \frac{4p\mu^2 t^2 + 4\mu^4 + p^2 t^4}{4(\mu^2+t^2)^2} \\
    &\quad \geq 1- \frac{\mu^2 t^2 +\frac{3}{4}t^4}{(\mu^2+t^2)^2} 
     = \frac{\mu^4 + \mu^2 t^2 + \frac{1}{4}t^4}{(\mu^2+t^2)^2} \geq \frac{1}{4}
\end{align*}
for any~$t\geq 0$. Thus, the inequality also holds true in this case.
\end{proof}

The lemma hereinafter is an immediate consequence of the definition of the bilinear form~$\foo{\mathcal{B}}$ defined in Section~\ref{subsec:notation}, taking into account assumptions~$\eqref{voraussetzungen}_1$ and~$\eqref{voraussetzungen}_2$.
\begin{mylem} \label{bilinearformeins}
    Let~$p>1$ and~$\mu\in(0,1]$. There holds
    \begin{align}\label{bilinearformeinsestimate}
        \begin{cases}
            |\foo{\mathcal{B}}(\cdot,\xi)(\eta,\zeta)| \leq C_1 h_\mu(|\xi|)|\eta||\zeta| &\\
            \foo{\mathcal{B}}(\cdot,\xi)(\eta,\eta) \geq C_2 h_\mu(|\xi|)|\eta|^2
        \end{cases}
    \end{align}
    for any~$\xi,\eta,\zeta\in \R^n$.
\end{mylem}
\subsubsection{Additional material} 
\,\\
The technique of reabsorbing certain quantities in estimates will prove highly beneficial on numerous occasions. This is achieved by utilizing the subsequent iteration lemma, as showcased in~\cite[Lemma~6.1]{giusti2003direct}.
  \begin{mylem} \label{iterationlemma}
      Let~$\phi(\rho)$ be a bounded, non-negative function on~$0\leq R_0\leq \rho\leq R_1$ and assume that for any~$R_0\leq \rho < r \leq R_1$ there holds
      $$\phi(\rho) \leq \eta \phi(r) + \frac{A}{(r-\rho)^{\alpha}} + \frac{B}{(r-\rho)^{\beta}} + C$$
      for some constants~$A,B,C,\alpha\geq \beta\geq 0$, and~$\eta\in(0,1)$. Then, there exists a constant~$C=C(\eta,\alpha)$, such that for all~$R_0\leq \rho_0<r_0\leq R_1$ there holds
      $$\phi(\rho_0) \leq C \bigg(\frac{A}{(r_0-\rho_0)^{\alpha}} + \frac{A}{(r_0-\rho_0)^{\alpha}} + C \bigg).$$
  \end{mylem}
The following version of the dominated convergence theorem will be useful and originates from~\cite[Chapter~1.3,~Theorem~4]{evans2018measure}.
\begin{mylem} \label{dominatedconvergence}
    Let~$g,(g_k)_{k\in\N}\in L^1(\Omega)$ and~$f,(f_k)_{k\in\N}$ be Lebesgue measurable. Suppose that~$|f_k|\leq g_k$ for any~$k\in\N$, both~$f_k\to f$ and~$g_k\to g$ a.e. in~$\Omega$ as~$k\to\infty$, and
    $$\lim\limits_{k\to\infty}\displaystyle\int_{\Omega}g_k\,\dx = \displaystyle\int_{\Omega}g\,\dx.$$
    Then, there holds
    $$\lim\limits_{k\to\infty}\displaystyle\int_{\Omega}|f_k-f|\,\dx = 0.$$
\end{mylem}

In order to deal with the nonlinearity of equation~\eqref{pde}, the lemma below will be convenient and may be found in \cite[Lemma 2.2]{giaquinta1986partial}. 
\begin{mylem}\label{lem:a-b}
Let~$k\in\mathbb N$. For any~$\alpha>1$, there exists a constant~$C=C(\alpha)$, such that
$$
    \frac{1}{C}\big||a|^{\alpha-1}a - |b|^{\alpha-1}b\big|
	\le
	\big(|a|^{\alpha-1} + |b|^{\alpha-1}\big)|a-b|
	\le
	C \big||a|^{\alpha-1}a - |b|^{\alpha-1}b\big|
$$
for all~$a,b\in\R^k$.
\end{mylem}

\section{Gradient bound for parabolic~\texorpdfstring{$p$}{}-Laplacian type equations}\label{sec:gradientbound}

In this section, we establish local boundedness of the spatial gradient of weak solutions to parabolic~$p$-Laplacian type equations. Moreover, we provide quantitative estimates for the whole parameter range~$p>1$. Specifically, we consider weak solutions to equations of the type
\begin{align}
        \partial_t u - \divv A(x,t,Du)=0
    \qquad\mbox{in $\Omega_T$,} \label{pdeqgleich1}
\end{align} 
where~$A$ is assumed to be differentiable with respect to the spatial variable~$x$ and satisfy the following structure conditions
\begin{align}
    \left\{
    \begin{array}{l}
    | A(x,t,\xi)| + (\mu^2 + |\xi|^2 )^{\frac{1}{2}}|\partial_\xi A(x,t,\xi)| \leq C_4 (\mu^2 + |\xi |^{2})^{\frac{p-1}{2}} \\[3pt]
    \langle \partial_{\xi}A(x,t,\xi)\eta, \eta \rangle \geq C_5 (\mu^2 + |\xi|^2 )^{\frac{p-2}{2}}|\eta|^2 \\[3pt]
     |\partial_x A_i(x,t,\xi)| \leq C_6 (\mu^2+|\xi|^2)^{\frac{p-1}{2}}  
    \end{array}
    \right. \label{voraussetzungendiffbarkeit}
\end{align}
in the case~$\mu\in(0,1]$ for a.e.~$(x,t) \in \Omega_T$,~$i\in\{1,...,n\}$, any~$\eta, \xi \in\R^n$, with positive constants~$C_4, C_5, C_6$. It should be noted that assumptions~$\eqref{voraussetzungendiffbarkeit}_1 - \eqref{voraussetzungendiffbarkeit}_2$ correspond to the first two general assumptions in~$\eqref{voraussetzungen}$, while condition~$\eqref{voraussetzungendiffbarkeit}_3$ deviates from its counterpart in~\eqref{voraussetzungen}. The differentiability of~$A$ with respect to~$x$ allows us to differentiate the weak form of~$\eqref{pdeqgleich1}$, making it a valuable tool in order to derive Caccioppoli type inequalities. In particular, equation~\eqref{pdeqgleich1} does not exhibit a nonlinearity in the evolution part anymore. It is notable that within this section, weak solutions are not required to be non-negative, but rather they may be signed. The local gradient boundedness of weak solutions to equation~\eqref{pdeqgleich1} under the assumptions~\eqref{voraussetzungendiffbarkeit} in the domain~$\Omega_T$ has already been established in the parameter range~$p>\critical$.~The case~$p\geq 2$ is covered by~\cite[Theorem~1.2]{bogeleinduzaarmarcellini}, while the range~$\critical<p<2$ is addressed in~\cite[Theorem~3.4]{singer2015parabolic}. In this article, we focus on the missing sub-critical case~$p\leq \critical$. In addition to establishing local boundedness of the gradient, we also obtain quantitative estimates that will be utilized in Section~\ref{sec:schauderestimates}. In order to prove gradient boundedness, energy estimates involving second order spatial derivatives are crucial. Since our notion of solution does not include higher order spatial derivatives, it is necessary to utilize the method of difference quotients. Generally, local existence of second order spatial derivatives can only be ensured when~$\mu\in(0,1]$. Additionally, in the sub-quadratic range~$1<p<2$, the local square-integrability of the gradient, i.e.~$|Du|\in L^2_{\loc}(\Omega_T)$, is a necessary condition. To establish the latter, the technique of difference quotients again proves to be beneficial. Once~$|Du|\in L^2_{\loc}(\Omega_T)$ has been established, the higher integrability result in the sub-critical range~$p\leq\critical$ stated in Lemma~\ref{higherintegrabilitylemma} and the local gradient boundedness, combined with the quantitative estimates from Propositions~\ref{gradientboundednesssubcritical} and~\ref{gradientboundednesssupercritical}, are immediate consequences of~\cite[Chapter~9]{gradientholder}. It should be noted that Proposition~\ref{gradientboundednesssubcritical} is formulated specifically for locally bounded solutions, as weak solutions in the sub-critical range~$p\leq\critical$ may be unbounded. However, in the super-critical case~$p>\critical$, it has been proven in~\cite[Chapter V, Theorem 3.1]{dibenedetto1993degenerate} that weak solutions are locally bounded. Indeed, it is straightforward to verify the assumptions~(B$_1$) - ~(B$_6$) of~\cite[Chapter V, Theorem 3.1]{dibenedetto1993degenerate} by exploiting growth condition~$\eqref{voraussetzungendiffbarkeit}_1$ and utilizing Lemma~\ref{AbschätzungenfürAeins}. Therefore, the additional boundedness assumption is not necessary in the later treatment of the case where~$p> 2$, established in Proposition~\ref{gradientboundednesssupercritical}.

\begin{myproposition}\label{gradientboundednesssubcritical}
     Let~$1<p\leq 2$,~$\mu\in(0,1]$, and~$u$ be a locally bounded weak solution to~\eqref{pdeqgleich1} under assumptions~\eqref{voraussetzungendiffbarkeit}. Then, there holds~$|Du|\in L^\infty_{\loc}(\Omega_T)$. Furthermore, for any~$\epsilon\in(0,1]$ and any cylinder~$Q^{(\lambda)}_{2\rho}(z_0)\Subset\Omega_T$, we have the quantitative estimate
     \begin{align}\label{gradientboundednesssubcriticalestimate}
         \esssup\limits_{Q^{(\lambda)}_{\frac{\rho}{2}}(z_0)}|Du| \leq C\epsilon\lambda &+ \frac{C\lambda^{\frac{1}{2}}}{\epsilon^ \theta}\bigg[\Big(\frac{\omega}{\rho \lambda} \Big)^{\frac{2}{p}} + \frac{\omega}{\rho \lambda} + \frac{\mu}{\lambda} \bigg]^{\frac{n(2-p)+2p}{4p}} \\
         &\cdot \bigg[\displaystyle\fiint_{Q^{(\lambda)}_{2\rho}(z_0)}(\mu^2 + |Du|^2)^{\frac{p}{2}}\,\dx\dt \bigg]^{\frac{1}{2p}} \nonumber
     \end{align}
     with~$C=C(n,p,C_4,C_5,C_6)$,~$\theta=\theta(n,p)$, and~$\omega\coloneqq \essosc\limits_{Q^{(\lambda)}_{2\rho}(z_0)}u$.
\end{myproposition}
\begin{myproposition}\label{gradientboundednesssupercritical}
         Let~$p>2$,~$\mu\in(0,1]$,~and $u$ be a weak solution to~\eqref{pdeqgleich1} under assumptions~\eqref{voraussetzungendiffbarkeit}. Then, there holds~$|Du|\in L^\infty_{\loc}(\Omega_T)$. Furthermore, for any~$\epsilon\in(0,1]$ and any cylinder~$Q^{(\lambda)}_{2\rho}(z_0)\Subset\Omega_T$, we have the quantitative estimate
\begin{equation}\label{gradientboundednesssupercriticalestimate}
         \esssup\limits_{Q^{(\lambda)}_{\frac{\rho}{2}}(z_0)}|Du| \leq C\epsilon\lambda + \frac{C}{\epsilon^\theta} \bigg[ 
\lambda^{2-p} \displaystyle\fiint_{Q^{(\lambda)}_{2\rho}(z_0)}(\mu^2 + |Du|^2)^{\frac{p}{2}}\,\dx\dt \bigg]^{\frac{1}{2}} 
     \end{equation}
     with~$C=C(n,p,C_4,C_5,C_6)$,~$\theta=\theta(n,p)$, and~$\omega\coloneqq \essosc\limits_{Q^{(\lambda)}_{2\rho}(z_0)}u$.
\end{myproposition}

\begin{remark}\upshape
   As demonstrated in~\cite[Chapter~9]{gradientholder}, both estimates~\eqref{gradientboundednesssubcriticalestimate} and~\eqref{gradientboundednesssupercriticalestimate} remain stable in the limit~$p\to 2$ and result in slightly different bounds for the spatial derivative. In the upcoming section on Schauder estimates, we will utilize the quantitative estimates~\eqref{gradientboundednesssubcriticalestimate} and~\eqref{gradientboundednesssupercriticalestimate} that have already been established for intrinsic parabolic cylinders. Corresponding gradient boundedness results for standard parabolic cylinders have also been derived: in the super-quadratic case~$p\geq 2$, we refer to~\cite[Theorem~1.2]{bogeleinduzaarmarcellini}, and in the sub-quadratic-super-critical case~$\critical<p<2$, we refer to~\cite[Theorem~3.4]{singer2015parabolic}.
\end{remark}

We will begin by treating the local square-integrability of the spatial derivative of weak solutions to equation~\eqref{pdeqgleich1}. Note that the property~$|Du| \in L^2_{\loc}(\Omega_T)$ is not included in Definition~\ref{definitionglobal} in the parameter range~$p\leq \critical$, where~$\critical<2$ if~$n\geq 3$. To prove~$|Du|\in L^2_{\loc}(\Omega_T)$, we need to establish the following energy estimate of Caccioppoli type. 

\begin{mylem}
\label{energyestimatefürduinl2}
 Let~$1<p\leq 2$,~$\mu\in(0,1]$,~and $u$ be a weak solution to~\eqref{pdeqgleich1} under assumptions~\eqref{voraussetzungendiffbarkeit}. Further, let~$Q_\rho\Subset\Omega_T$, such that~$\frac{\rho}{2}\leq\sigma<r\leq \rho-h_0$, where~$|h|<h_0<\frac{\rho}{2}$. Then, for any~$i\in\{1,...,n\}$, there holds the energy estimate
    \begin{align}\label{energyestimatefürduinl2inequality}
   & \esssup\limits_{t\in(t_0-\sigma^2,t_0]} \displaystyle\int_{B_\sigma(x_0)}\big|\Delta_h^{(i)}u \big|^2\,\dx + \displaystyle\iint_{Q_{\sigma}(z_0)}\big|\Delta_h^{(i)} \mathcal{V}_\mu^{(p)}(Du)\big|^2 \,\dx\dt \\ \nonumber
    & \leq C\Big(1+\frac{1}{(r-\sigma)^2}\Big) \displaystyle\iint_{Q_{\rho}(z_0)} (\mu^2+|Du|^2 )^{\frac{p}{2}} \,\dx\dt + \frac{C}{(r-\sigma)^2} \displaystyle\iint_{Q_{r}(z_0)} \big|\Delta_h^{(i)} u \big|^2 \,\dx\dt \nonumber
\end{align}
with~$C=C(n,p,C_4,C_5,C_6)$.
\end{mylem}
\begin{proof}
The starting point is inequality~\cite[(3.1)]{singer2015parabolic}, which yields the following energy estimate
\begin{align*}
   \frac{1}{2}\displaystyle\int_{B_{r}(x_0)} & \phi^2  \int_{0}^{|\Delta^{(i)}_h u(x,\tau)|^2} g(s)\,\ds\dx \\
   &+ 3 \displaystyle\iint_{B_{r}(x_0)\times(t_0-r^2,\tau]} \phi^2 \chi \big\langle \Delta^{(i)}_h A(\cdot,Du), D\big[\Delta^{(i)}_h(u) g(\big|\Delta^{(i)}_h u \big|^2)\big] \big\rangle\,\dx\dt \\ 
   &\leq -2 \displaystyle\iint_{B_{r}(x_0)\times(t_0-r^2,\tau]} \phi \chi \big\langle \Delta^{(i)}_h A(\cdot,Du), D \big[\phi g(\big|\Delta^{(i)}_h u \big|^2)\big] \big\rangle  \Delta^{(i)}_h u \,\dx\dt \\ 
   & \quad + \frac{1}{2}\displaystyle\iint_{B_{r}(x_0)\times(t_0-r^2,\tau]} \phi^2 \partial_t\chi \int_{0}^{|\Delta^{(i)}_h u|^2} g(s)\,\ds\dx\dt 
\end{align*}
for a.e.~$\tau\in(t_0-{r}^2,t_0]$. Here,~$g\in W^{1,\infty}(\R)$ denotes an arbitrary non-negative, bounded and increasing Lipschitz function,~$\phi\in C^1_0(B_{r}(x_0),[0,1])$ is a spatial cut-off function satisfying~$\phi\equiv 1$ on~$B_{\sigma}(x_0)$ and~$|D\phi| \leq \frac{2}{r-\sigma}$, whereas~$\chi\in W^{1,\infty}([t_0-r^2,t_0],[0,1])$ denotes an increasing cut-off function in time with~$\chi(t_0-r^2)=0$,~$\chi\equiv 1$ on~$(t_0-\sigma^2,t_0]$, and~$|\partial_t \chi| \leq \frac{2}{r^2-\sigma^2}$. We apply the preceding estimate with the admissible choice~$g\equiv 1$. In turn, this yields the inequality
\begin{align}\label{singerinequalitywithggleicheins}
   \frac{1}{2}\displaystyle\int_{B_{r}(x_0)} & \phi^2 \big| \Delta^{(i)}_h u(x,\tau)\big|^2\,\dx \\ \nonumber
   &+ 3 \displaystyle\iint_{B_{r}(x_0)\times(t_0-r^2,\tau]} \phi^2 \chi \big\langle \Delta^{(i)}_h A(\cdot,Du), \Delta^{(i)}_h(Du) \big\rangle\,\dx\dt \\ \nonumber
   &\leq -2 \displaystyle\iint_{B_{r}(x_0)\times(t_0-r^2,\tau]} \phi \chi \big\langle \Delta^{(i)}_h A(\cdot,Du), D \phi \big\rangle  \Delta^{(i)}_h u \,\dx\dt \\ \nonumber
   &  \quad+ \frac{C}{(r-\sigma)^2}\displaystyle\iint_{Q_{r}(z_0)}\big|\Delta^{(i)}_h u \big|^2\,\dx\dt \nonumber
\end{align}
for a.e.~$\tau\in(t_0-{r}^2,t_0]$, and $C=C(n,p)$. In the last term of the preceding inequality, we exploited the fact that~$r^2-\sigma^2 \geq (r-\sigma)^2$ and bounded the spatial cut-off function~$\phi$ by~$1$. In particular, the second term on the left-hand side and the first quantity on the right-hand side of the preceding inequality require a careful treatment. We will start by bounding the second term on the left-hand side from below. Due to~$|(x+he_i)-x|=|h|$, Lemma~\ref{AbschätzungenfürAeins} and Lemma~\ref{AbschätzungenfürAzwei} yield for a.e.~$(x,t)\in Q_{r}(z_0)$ the following estimate
\begin{align}\label{lefthandside}
    \big\langle  & \Delta^{(i)}_h A(\cdot,Du), \Delta^{(i)}_h(Du) \big\rangle(x,t) \\ \nonumber
    & = \frac{1}{h^2}\langle A(x+h e_i,t,Du(x+he_i,t)),Du(x+he_i,t)-Du(x,t)\rangle \\ \nonumber
    & \quad- \frac{1}{h^2} \langle A(x+he_i,t,Du(x,t)), Du(x+he_i,t)-Du(x,t)\rangle \\ \nonumber
    & \quad+ \frac{1}{h^2}\langle A(x+h e_i,t,Du(x,t))-A(x,t,Du(x,t)), Du(x+he_i,t)-Du(x,t)\rangle \\ \nonumber
    &\geq \frac{1}{\Tilde{C}} \mathcal{D}^{(i)}_\mu(h)^{\frac{p-2}{2}}\big|\Delta^{(i)}_h(Du) \big|^2 - C\mathcal{D}^{(i)}_\mu(h)^{\frac{p-1}{2}}\big|\Delta^{(i)}_h(Du)\big| \\ \nonumber
    &= \frac{1}{\Tilde{C}} \mathcal{D}^{(i)}_\mu(h)^{\frac{p-2}{2}}\big|\Delta^{(i)}_h(Du) \big|^2 - C\mathcal{D}^{(i)}_\mu(h)^{\frac{p}{4}} \mathcal{D}^{(i)}_\mu(h)^{\frac{p-2}{4}} \big|\Delta^{(i)}_h(Du)\big| \\ \nonumber
    &\geq \frac{1}{\Tilde{C}} \mathcal{D}^{(i)}_\mu(h)^{\frac{p-2}{2}}\big|\Delta^{(i)}_h(Du) \big|^2 - \frac{1}{2\Tilde{C}}\mathcal{D}^{(i)}_\mu(h)^{\frac{p-2}{2}}\big|\Delta^{(i)}_h(Du) \big|^2 - C\mathcal{D}^{(i)}_\mu(h)^{\frac{p}{2}} \\ \nonumber
    &= \frac{1}{\Tilde{C}} \mathcal{D}^{(i)}_\mu(h)^{\frac{p-2}{2}}\big|\Delta^{(i)}_h(Du) \big|^2 - C\mathcal{D}^{(i)}_\mu(h)^{\frac{p}{2}} \nonumber
\end{align}
with~$C=C(n,p,C_5,C_6)$, and~$\Tilde{C}=\Tilde{C}(n,p,C_5,C_6)$. In the last inequality of the preceding estimate, we utilized Young's inequality and a suitable choice of constant to reabsorb the first quantity. Next, we will treat the first term on the right-hand side of~\eqref{singerinequalitywithggleicheins} and abbreviate~$\Psi = 2 \chi \phi D\phi  \Delta^{(i)}_h u$. By applying Lemma~\ref{diskretePI}, utilizing Hölder's inequality, and considering the growth condition~$\eqref{voraussetzungendiffbarkeit}_1$, we deduce the estimate
\begin{align} \label{righthandside}
    -2 & \displaystyle\iint_{B_r(x_0)\times(t_0-r^2,\tau]} \phi \chi \big\langle \Delta^{(i)}_h A(\cdot,Du), D \phi \big\rangle \Delta^{(i)}_h u \,\dx\dt \\ \nonumber
    &= - \displaystyle\iint_{B_r(x_0)\times(t_0-r^2,\tau]} \big\langle \Delta^{(i)}_h A(\cdot,Du), \Psi \big\rangle\,\dx\dt \\ \nonumber
    &= \displaystyle\iint_{B_r(x_0)\times(t_0-r^2,\tau]}\int_{0}^{1} \langle A(x+she_i,t,Du(x+she_i,t)), D_i\Psi\rangle\,\ds\dx\dt \\ \nonumber
    &\leq \bigg[\displaystyle\iint_{Q_r(z_0)}\int_{0}^{1}|A(x+she_i,t,Du(x+she_i,t)) |^{\frac{p}{p-1}}\,\ds\dx\dt \bigg]^{\frac{p-1}{p}}  \\ \nonumber
    &  \quad\cdot \bigg[\displaystyle\iint_{B_r(x_0)\times(t_0-r^2,\tau]}|D_i\Psi|^p\,\dx\dt \bigg]^{\frac{1}{p}} \\\nonumber
    & \leq \bigg[\displaystyle\iint_{Q_{\rho}(z_0)}(\mu^2 + |Du|^2 )^{\frac{p}{2}}\,\dx\dt \bigg]^{\frac{p-1}{p}} \bigg[\displaystyle\iint_{B_r(x_0)\times(t_0-r^2,\tau]}|D_i\Psi|^p\,\dx\dt \bigg]^{\frac{1}{p}} \nonumber
\end{align}
for a.e.~$\tau\in(t_0-r^2,t_0]$. Next, we further investigate the quantity involving~$D_i \Psi$. The assumptions on~$\phi$ and~$\chi$ yield
\begin{align*}
    |D_i\Psi| & = 2 \chi \big|D_i\phi D\phi \Delta^{(i)}_h u + \phi D_iD\phi \Delta^{(i)}_h u + \phi \big\langle D\phi, \Delta^{(i)}_h (D_iu) \big\rangle \big| \\
    & \leq 2 \chi \big[(|D\phi|^2 + \phi|D^2\phi|)\big|\Delta^{(i)}_h u \big| + \phi |D\phi| \big| \Delta^{(i)}_h (Du)\big| \big] \\
    &\leq C \chi \bigg[ \bigchi_{B_r(x_0)} \frac{\big|\Delta^{(i)}_h u \big|}{(r-\sigma)^2} + \phi\frac{\big|\Delta^{(i)}_h (Du) \big|}{r-\sigma}  \bigg]
\end{align*}
with~$C>0$ depending on~$n,p,C_4,C_5$. Therefore, we obtain in~\eqref{righthandside} 
\begin{align*}
    \bigg[\displaystyle\iint&_{B_r(x_0)\times(t_0-r^2,\tau]}|D_i\Psi|^p\,\dx\dt \bigg]^{\frac{1}{p}} \\
    &\leq  \frac{C}{(r-\sigma)^2}\bigg[\displaystyle\iint_{Q_r(z_0)}\big|\Delta^{(i)}_h u\big|^p\,\dx\dt \bigg]^{\frac{1}{p}} \\ \nonumber
    &  \quad+ \frac{C}{r-\sigma} \bigg[\displaystyle\iint_{B_r(x_0)\times(t_0-r^2,\tau]}\chi^p\phi^p \big|\Delta^{(i)}_h (Du)\big|^p\,\dx\dt \bigg]^{\frac{1}{p}} \\ \nonumber
    &= \foo{I}+\foo{II}.
\end{align*}
Bounding~$\foo{I}$ further below with Lemma~\ref{differenzenquotientlemmadrei}, we achieve
\begin{align*}
    \foo{I} \leq \frac{C}{(r-\sigma)^2} \bigg[\displaystyle\iint_{Q_{\rho}(z_0)}(\mu^2 + |Du|^2)^\frac{p}{2}\,\dx\dt \bigg]^{\frac{1}{p}}.
\end{align*}
To estimate~$\foo{II}$, we utilize Hölder's inequality with exponents~$(\frac{2}{2-p},\frac{2}{p})$ and obtain
\begin{align*}
\foo{II} &= \frac{C}{r-\sigma} \\
& \quad\cdot  \bigg[\displaystyle\iint_{B_r(x_0)\times(t_0-r^2,\tau]}  \chi^{\frac{p}{2}} \mathcal{D}^{(i)}_\mu(h)^{\frac{(2-p)p}{4}} \big(\chi^{\frac{1}{2}}\phi \mathcal{D}^{(i)}_\mu(h)^{\frac{p-2}{4}} \big|\Delta^{(i)}_h (Du) \big|\big)^p\,\dx\dt \bigg]^{\frac{1}{p}} \\
&\leq \frac{C}{r-\sigma} \bigg[\displaystyle\iint_{Q_r(z_0)}\mathcal{D}^{(i)}_\mu(h)^{\frac{p}{2}}\,\dx\dt \bigg]^{\frac{2-p}{2p}} \\
& \quad\cdot \bigg[\displaystyle\iint_{B_r(x_0)\times(t_0-r^2,\tau]}\chi \phi^2 \mathcal{D}^{(i)}_\mu(h)^{\frac{p-2}{2}}\big|\Delta^{(i)}_h (Du) \big|^2  \,\dx\dt \bigg]^{\frac{1}{2}} \\
&\leq \frac{C}{r-\sigma} \bigg[\displaystyle\iint_{Q_{\rho}(z_0)}(\mu^2 + |Du|^2 )^{\frac{p}{2}}\,\dx\dt \bigg]^{\frac{2-p}{2p}} \\
& \quad\cdot \bigg[\displaystyle\iint_{B_r(x_0)\times(t_0-r^2,\tau]}\chi \phi^2 \mathcal{D}^{(i)}_\mu(h)^{\frac{p-2}{2}}\big|\Delta^{(i)}_h (Du) \big|^2  \,\dx\dt \bigg]^{\frac{1}{2}}.
\end{align*}
The preceding estimates and an application of Young's inequality with exponent~$2$ yield the following bound in inequality~\eqref{righthandside}
\begin{align*}
    -2 & \displaystyle\iint_{B_r(x_0)\times(t_0-r^2,\tau]} \phi \chi \big\langle \Delta^{(i)}_h A(\cdot,Du), D \phi \big\rangle \big|\Delta^{(i)}_h u \big|^2\,\dx\dt \\ \nonumber
    &\leq \frac{C}{(r-\sigma)^2} \displaystyle\iint_{Q_{\rho}(z_0)}(\mu^2 + |Du|^2 )^{\frac{p}{2}}\,\dx\dt  \\
    & \quad+ \frac{C}{r-\sigma} \bigg[\displaystyle\iint_{Q_{\rho}(z_0)}(\mu^2 + |Du|^2 )^{\frac{p}{2}}\,\dx\dt \bigg]^{\frac{1}{2}}\\
    & \quad\cdot \bigg[\displaystyle\iint_{B_r(x_0)\times(t_0-r^2,\tau]}\chi \phi^2 \mathcal{D}^{(i)}_\mu(h)^{\frac{p-2}{2}}\big|\Delta^{(i)}_h (Du) \big|^2  \,\dx\dt \bigg]^{\frac{1}{2}} \\
    &\leq \frac{C}{(r-\sigma)^2} \displaystyle\iint_{Q_{\rho}(z_0)}(\mu^2 + |Du|^2 )^{\frac{p}{2}}\,\dx\dt  \\
    & \quad+ \frac{1}{2\Tilde{C}}\displaystyle\iint_{B_r(x_0)\times(t_0-r^2,\tau]} \chi \phi^2 \mathcal{D}^{(i)}_\mu(h)^{\frac{p-2}{2}}\big|\Delta^{(i)}_h (Du) \big|^2  \,\dx\dt,
\end{align*}
where~$\Tilde{C}=\Tilde{C}(n,p,C_5,C_6)$ denotes the very same constant from inequality~\eqref{lefthandside} and~$C=C(n,p,C_4,C_5,C_6)$. Thus, after bounding the left-hand side of~\eqref{singerinequalitywithggleicheins} further below by exploiting~\eqref{lefthandside}, we are able to reabsorb the last term in the preceding estimate into the left-hand side. Overall, we end up with
\begin{align*}
   \frac{1}{2}&\displaystyle\int_{B_r(x_0)}  \phi^2 \big| \Delta^{(i)}_h u(x,\tau)\big|^2\,\dx \\ \nonumber
   &+  \displaystyle\iint_{B_r(x_0)\times(t_0-r^2,\tau]} \phi^2 \chi \mathcal{D}^{(i)}_\mu(h)^{\frac{p-2}{2}}\big|\Delta^{(i)}_h (Du) \big|^2\,\dx\dt \\ \nonumber
   &\leq C\displaystyle\iint_{Q_{\rho}(z_0)}(\mu^2+|Du|^2 )^{\frac{p}{2}}\,\dx\dt + \frac{C}{(r-\sigma)^2}    \displaystyle\iint_{Q_{\rho}(z_0)}(\mu^2 + |Du|^2 )^{\frac{p}{2}}\,\dx\dt 
   \\ \nonumber
   &  \quad+ \frac{C}{(r-\sigma)^2}\displaystyle\iint_{Q_r(z_0)}\big|\Delta^{(i)}_h u \big|^2\,\dx\dt \nonumber
\end{align*}
for a.e.~$\tau\in(t_0-r^2,t_0]$, and $C=C(n,p,C_4,C_5,C_6)$. Finally, bounding the second term on the left-hand side in the preceding inequality further below with Lemma~\ref{Vabschätzung}, exploiting the assumptions made on the cut-off functions~$\phi$ and~$\chi$, together with taking the essential supremum over~$\tau\in (t_0-\sigma^2,t_0]$ and passing to the limit~$\tau\uparrow t_0$, we obtain the claimed inequality
\begin{align*}
    &\esssup\limits_{t\in(t_0-\sigma^2,t_0]} \displaystyle\int_{B_\sigma(x_0)}\big|\Delta_h^{(i)}u \big|^2\,\dx + \displaystyle\iint_{Q_{\sigma}(z_0)}\big|\Delta_h^{(i)} \mathcal{V}_\mu^{(p)}(Du)\big|^2 \,\dx\dt \\
    & \leq C\Big(1+\frac{1}{(r-\sigma)^2}\Big) \displaystyle\iint_{Q_{\rho}(z_0)} (\mu^2+|Du|^2 )^{\frac{p}{2}} \,\dx\dt+ \frac{C}{(r-\sigma)^2} \displaystyle\iint_{Q_{r}(z_0)} \big|\Delta_h^{(i)} u \big|^2 \,\dx\dt \nonumber
\end{align*}
with~$C=C(n,p,C_4,C_5,C_6)$. 
\end{proof}
\begin{remark} \label{cabhängigvonmu} \upshape
The necessity of the Caccioppoli type estimate~\eqref{energyestimatefürduinl2inequality} becomes evident when comparing it to the energy estimates given in~\cite[(3.6)]{bogeleinduzaarmarcellini} and~\cite[(3.7)]{singer2015parabolic}. The main issue with these estimates is the presence of the term~$\mathcal{D}^{(i)}_\mu(h)^{\frac{p-2}{2}}$, which is challenging to handle when~$1<p\leq\critical$. Furthermore, it is crucial to avoid introducing an additional dependence on the parameter~$\mu$ in the constants. Therefore, we choose not to use~\cite[(3.6)]{bogeleinduzaarmarcellini} or~\cite[(3.7)]{singer2015parabolic}, but instead rely on our Caccioppoli type estimate~\eqref{energyestimatefürduinl2inequality}.
\end{remark}
We are now in position to prove the claimed local square-integrability of~$Du$ in~$\Omega_T$.
\begin{mylem}\label{DuinL2}
Let~$1<p\leq\critical$,~$\mu\in(0,1]$, and~$u$ be a locally bounded weak solution to~\eqref{pdeqgleich1} under assumptions~\eqref{voraussetzungendiffbarkeit}. Then, there holds
$$| Du|\in L^2_{\loc}(\Omega_T).$$
Moreover, for any cylinder~$Q_\rho(z_0) \Subset \Omega_T$ we have the quantitative estimate
\begin{align} \label{duinl2ungleichung}
   \displaystyle\fiint_{Q_{\frac{\rho}{2}}(z_0)} |Du|^2\,\dx\dt &\leq \frac{C L^2}{\rho^2} \bigg[\displaystyle\fiint_{Q_{\rho}(z_0)}(\mu^2 +|Du |^2)^{\frac{p}{2}}\,\dx\dt\bigg]^{\frac{2-p}{p}}  \\ \nonumber
    & \quad + CL\Big(1 + \frac{1}{\rho}\Big) \bigg[\displaystyle\fiint_{Q_{\rho}(z_0)}(\mu^2 + |Du |^2)^{\frac{p}{2}}\,\dx\dt\bigg]^{\frac{1}{p}} \nonumber
\end{align}
with~$L\coloneqq \| u\|_{L^\infty(Q_{\rho}(z_0))}$, and~$C=C(n,p,C_4,C_5,C_6)$.
\end{mylem}
\begin{proof}
    Let us fix~$\frac{\rho}{2}\leq \sigma<\hat{r}<r\leq\rho-h_0$, such that~$\hat{r}=\frac{1}{2}(\sigma+r)$, where~$|h|<h_0<\frac{\rho}{2}$. Consider a smooth spatial cut-off function~$\phi \in C^1_0(B_{\hat{r}}(x_0),[0,1])$, such that~$\phi\equiv 1$ on~$B_\sigma(x_0)$ and~$|D\phi| \leq \frac{2}{\hat{r}-\sigma}$. Rather than directly investigating~$|Du|\in L^2_{\loc}(\Omega_T)$, we initially consider the difference quotient of~$u$. This approach enables us to employ the discrete integration by parts technique introduced in Lemma~\ref{diskretePI}. We obtain
\begin{align*}
    \quad&\displaystyle\iint_{Q_\sigma(z_0)} \big|\Delta_h^{(i)} u\big|^2\,\dx\dt \leq \displaystyle\iint_{Q_{\hat{r}}(z_0)} \big|\Delta_h^{(i)} u\big|^2 \phi^2\,\dx\dt \\
    & = \displaystyle \int_{t_0-\hat{r}^2}^{t_0} \int_{B_{\hat{r}} (x_0)} \big\langle \Delta_h^{(i)} u, \phi^2 \Delta_h^{(i)} u\big\rangle \,\dx\dt \\
    & = \displaystyle -\int_{t_0-\hat{r}^2}^{t_0} \int_{B_{\hat{r}} (x_0)}\int_0^1 \big\langle u(x+hse_i,t), D_i\big(\phi^2 \Delta_h^{(i)} u \big)(x,t)\big\rangle\,\ds\dx\dt \\
    &\leq L \displaystyle\iint_{Q_{\hat{r}}(z_0)} \big|D_i\big(\phi^2 \Delta_h^{(i)} u \big) \big|\,\dx\dt \\
    &\leq L \displaystyle\iint_{Q_{\hat{r}}(z_0)} \big[2\phi |D_i\phi| \big|\Delta_h^{(i)} u \big| + \phi^2 \big|D_i\big(\Delta_h^{(i)} u \big)  \big| \big]\,\dx\dt \\
    &\leq \frac{4L}{\hat{r}-\sigma} \displaystyle\iint_{Q_{\hat{r}}(z_0)} \big|\Delta_h^{(i)} u \big|\,\dx\dt + L \displaystyle\iint_{Q_{\hat{r}}(z_0)} \big|\Delta_h^{(i)}Du \big|\,\dx\dt \\
    & = \foo{I}+\foo{II}.
\end{align*}
In the last step of the preceding estimate, we first utilized the commutativity of weak derivatives~$D_i$ with difference quotients~$\Delta_h^{(i)}$, and subsequently passed to the full derivative~$Du$ in the very last term. The initial quantity~$\foo{I}$ presents no difficulties, as it may be bounded above by utilizing Lemma~\ref{differenzenquotientlemmadrei} and the Cauchy-Schwarz inequality, which yield
\begin{align}\label{termeins}
    \foo{I} &\leq \frac{4L}{\hat{r}-\sigma} \displaystyle\iint_{Q_{r}(z_0)} |D u|\,\dx\dt \\ \nonumber
    &\leq \frac{C L}{r-\sigma}\bigg[\displaystyle \fiint_{Q_{\rho}(z_0)} (\mu^2+|Du |^2 )^{\frac{p}{2}}\,\dx\dt \bigg]^{\frac{1}{p}}|Q_{\rho}.(z_0)| \nonumber
\end{align}
However, the second term~$\foo{II}$ requires a more delicate treatment. Applying Hölder's inequality twice and subsequently Lemma~\ref{Vabschätzung}, we achieve
\begin{align} \label{mugleich1}
    &\quad\quad\displaystyle\iint_{Q_{\hat{r}}(z_0)} \big|\Delta_h^{(i)}(Du) \big|\,\dx\dt = \displaystyle\iint_{Q_{\hat{r}}(z_0)} \mathcal{D}_{\mu}^{(i)}(h)^{\frac{p-2}{4}} \big|\Delta_h^{(i)}(Du) \big|\mathcal{D}_{\mu}^{(i)}(h)^{\frac{2-p}{4}}\,\dx\dt \\ \nonumber
    &\leq \bigg[\displaystyle\iint_{Q_{\hat{r}}(z_0)}\mathcal{D}_{\mu}^{(i)}(h)^{\frac{p-2}{2}} \big|\Delta_h^{(i)}(Du) \big|^2\,\dx\dt \bigg]^{\frac{1}{2}} \bigg[\displaystyle\iint_{Q_{\hat{r}}(z_0)} \mathcal{D}_{\mu}^{(i)}(h)^{\frac{2-p}{2}}\,\dx\dt \bigg]^{\frac{1}{2}} \\ \nonumber
    &\leq C \bigg[\displaystyle\iint_{Q_{\hat{r}}(z_0)}\big|\Delta_h^{(i)} \mathcal{V}_{\mu}^{(p)}(Du)\big|^2 \,\dx\dt \bigg]^{\frac{1}{2}} \bigg[\displaystyle\iint_{Q_{\hat{r}}(z_0)} \mathcal{D}_{\mu}^{(i)}(h)^{\frac{2-p}{2}}\,\dx\dt \bigg]^{\frac{1}{2}} \\ \nonumber
    & \leq C\bigg[\displaystyle\iint_{Q_{\hat{r}}(z_0)}\big|\Delta_h^{(i)} \mathcal{V}_{\mu}^{(p)}(Du)\big|^2 \,\dx\dt \bigg]^{\frac{1}{2}} \bigg[\displaystyle\iint_{Q_{\hat{r}}(z_0)} \mathcal{D}_{\mu}^{(i)}(h)^{\frac{p}{2}}\,\dx\dt \bigg]^{\frac{2-p}{2p}}|Q_\rho(z_0)|^{\frac{p-1}{p}}.
\end{align}
In order to continue, we utilize the Caccioppoli type inequality from Lemma~\ref{energyestimatefürduinl2}. After replacing~$\sigma$ by~$\hat{r}$ in~\eqref{energyestimatefürduinl2inequality} and due to our choice of~$r,\hat{r},\sigma$, there holds
\begin{align}\label{singerinequality}
  \esssup\limits_{t\in(t_0-\hat{r}^2,t_0]} \displaystyle&\int_{B_\sigma(x_0)}\big|\Delta_h^{(i)}u \big|^2\,\dx + \displaystyle\iint_{Q_{\hat{r}}(z_0)}\big|\Delta_h^{(i)} \mathcal{V}_\mu^{(p)}(Du)\big|^2 \,\dx\dt \\ \nonumber
    & \leq C\Big(1+\frac{1}{(r-\hat{r})^2}\Big) \displaystyle\iint_{Q_{\rho}(z_0)} (\mu^2+|Du|^2)^{\frac{p}{2}} \,\dx\dt \\ \nonumber
    &  \quad+ \frac{C}{(r-\hat{r})^2} \displaystyle\iint_{Q_{r}(z_0)} \big|\Delta_h^{(i)} u \big|^2 \,\dx\dt \nonumber \\ \nonumber
    &\leq C\Big(1+\frac{1}{(r-\sigma)^2}\Big) \displaystyle\iint_{Q_{\rho}(z_0)} (\mu^2+|Du|^2 )^{\frac{p}{2}} \,\dx\dt \\ \nonumber
    &  \quad+ \frac{C}{(r-\sigma)^2} \displaystyle\iint_{Q_{r}(z_0)} \big|\Delta_h^{(i)} u \big|^2 \,\dx\dt. \nonumber
\end{align}
For the second quantity~$\foo{II}$, the inequalities above yield the following
\begin{align*}
    \foo{II} &\leq  \bigg[ CL \bigg[\Big(1+\frac{1}{(r-\sigma)^2}\Big)\displaystyle\iint_{Q_{\rho}(z_0)}(\mu^2 +|Du|^2)^{\frac{p}{2}}\,\dx\dt \\
    &  \quad+ \frac{1}{(r-\sigma)^2}\displaystyle\iint_{Q_r(z_0)}\big|\Delta_h^{(i)} u\big|^2 \,\dx\dt \bigg]^{\frac{1}{2}}+CL \bigg[\displaystyle\iint_{Q_{\rho}(z_0)}(\mu^2 +|Du |^2)^{\frac{p}{2}}\,\dx\dt\bigg]^{\frac{1}{2}} \bigg] \\
    &  \quad\cdot\bigg[\displaystyle\iint_{Q_{\hat{r}}(z_0)} \mathcal{D}_{\mu}^{(i)}(h)^{\frac{p}{2}}\,\dx\dt \bigg]^{\frac{2-p}{2p}}  |Q_\rho(z_0)|^{\frac{p-1}{p}} \\
    &\leq CL\Big(1+\frac{1}{(r-\sigma)^2}\Big)^{\frac{1}{2}}\bigg[\displaystyle\iint_{Q_{\rho}(z_0)}(\mu^2 + |Du|^2)^{\frac{p}{2}}\,\dx\dt\bigg]^{\frac{1}{p}}|Q_\rho(z_0)|^{\frac{p-1}{p}} \\
    &  \quad+\frac{C L |Q_\rho(z_0)|^{\frac{p-1}{p}}}{r-\sigma} \bigg[\displaystyle\iint_{Q_r(z_0)}\big|\Delta_h^{(i)} u\big|^2 \,\dx\dt \bigg]^{\frac{1}{2}} \bigg[\displaystyle\iint_{Q_{\rho}(z_0)}(\mu^2 + |Du|^2)^{\frac{p}{2}}\,\dx\dt\bigg]^{\frac{2-p}{2p}} \\
    &  \quad+ C L \big|Q_\rho(z_0)\big|^{\frac{p-1}{p}} \bigg[\displaystyle\iint_{Q_{\rho}(z_0)}(\mu^2 + |Du |^2)^{\frac{p}{2}}\,\dx\dt\bigg]^{\frac{1}{p}}.
\end{align*}
 An application of Young's inequality with exponent~$2$ further implies
\begin{align}\label{termzwei}
    \foo{II} &\leq \frac{1}{2} \displaystyle\iint_{Q_r(z_0)}\big|\Delta_h^{(i)} u\big|^2 \,\dx\dt + 
    C L |Q_\rho(z_0)|^{\frac{p-1}{p}}\bigg[\displaystyle\iint_{Q_{\rho}(z_0)}(\mu^2 + |Du|^2)^{\frac{p}{2}}\,\dx\dt\bigg]^{\frac{1}{p}} \\ \nonumber
    & \quad+   C L |Q_\rho(z_0)|^{\frac{p-1}{p}}\Big(1+\frac{1}{r-\sigma}\Big)\bigg[\displaystyle\iint_{Q_{\rho}(z_0)}(\mu^2 + |Du |^2)^{\frac{p}{2}}\,\dx\dt\bigg]^{\frac{1}{p}}  \\ \nonumber
    & \quad+ \frac{C L^2 |Q_\rho(z_0)|^{\frac{2(p-1)}{p}}}{(r-\sigma)^2} \bigg[\displaystyle\iint_{Q_{\rho}(z_0)}(\mu^2 + |Du |^2)^{\frac{p}{2}}\,\dx\dt\bigg]^{\frac{2-p}{p}}  \\ \nonumber
    &= \frac{1}{2} \displaystyle\iint_{Q_r(z_0)}\big|\Delta_h^{(i)} u\big|^2 \,\dx\dt + \frac{C L^2 |Q_\rho(z_0)|}{(r-\sigma)^2} \bigg[\displaystyle\fiint_{Q_{\rho}(z_0)}(\mu^2 +|Du |^2)^{\frac{p}{2}}\,\dx\dt\bigg]^{\frac{2-p}{p}} \\  \nonumber
    & \quad + C L \Big(1+\frac{1}{r-\sigma}\Big) |Q_\rho(z_0)|\bigg[\displaystyle\fiint_{Q_{\rho}(z_0)}(\mu^2 + |Du |^2)^{\frac{p}{2}}\,\dx\dt\bigg]^{\frac{1}{p}}. \nonumber
\end{align}
Combining both estimates~\eqref{termeins} and~\eqref{termzwei}, and further defining
$$\phi(s) \coloneqq \displaystyle\iint_{Q_s(z_0)} \big|\Delta_h^{(i)} u\big|^2\,\dx\dt,$$
yields the estimate
\begin{align*}
    \phi(\sigma) &\leq \frac{1}{2}\phi(r) + \frac{C L^2}{(r-\sigma)^2} \bigg[\displaystyle\fiint_{Q_{\rho}(z_0)}(\mu^2 + |Du |^2)^{\frac{p}{2}}\,\dx\dt\bigg]^{\frac{2-p}{p}} |Q_\rho(z_0)| \\
    & \quad + CL \Big(1+\frac{1}{r-\sigma}\Big) \bigg[\displaystyle\fiint_{Q_{\rho}(z_0)}(\mu^2 + |Du|^2)^{\frac{p}{2}}\,\dx\dt\bigg]^{\frac{1}{p}}|Q_\rho(z_0)|. \nonumber
\end{align*}
At this point, we are in position to apply the geometric decay Lemma~\ref{iterationlemma} with the choice~$\sigma=\frac{\rho}{2}$ to obtain
\begin{align} \label{differenzabschätzung}
\displaystyle\iint_{Q_{\frac{\rho}{2}}(z_0)} \big|\Delta_h^{(i)} u\big|^2\,\dx\dt &\leq \frac{C L^2}{\rho^2} \bigg[\displaystyle\fiint_{Q_{\rho}(z_0)}(\mu^2 +|Du |^2)^{\frac{p}{2}}\,\dx\dt\bigg]^{\frac{2-p}{p}} |Q_\rho(z_0)| \\ \nonumber
    &  \quad + CL\Big(1+\frac{1}{\rho}\Big) \bigg[\displaystyle\fiint_{Q_{\rho}(z_0)}(\mu^2 + |Du|^2)^{\frac{p}{2}}\,\dx\dt\bigg]^{\frac{1}{p}}|Q_\rho(z_0)| . \nonumber
\end{align}
Since the right-hand side of the above inequality is independent of the parameter~$h \neq 0$, we may apply Lemma~\ref{differenzenquotientlemmazwei} and divide both sides by~$|Q_\rho(z_0)|$.~This yields the claimed quantitative estimate
\begin{align*}
\displaystyle\fiint_{Q_{\frac{\rho}{2}}(z_0)}|D_i u|^2\,\dx\dt &\leq \frac{C L^2}{\rho^2} \bigg[\displaystyle\fiint_{Q_{\rho}(z_0)}(\mu^2 + |Du |^2)^{\frac{p}{2}}\,\dx\dt\bigg]^{\frac{2-p}{p}}  \\
    & \quad + CL\Big(1+\frac{1}{\rho}\Big) \bigg[\displaystyle\fiint_{Q_{\rho}(z_0)}(\mu^2 + |Du |^2)^{\frac{p}{2}}\,\dx\dt\bigg]^{\frac{1}{p}}
    \end{align*}
    with~$C=C(n,p,C_4,C_5,C_6)$.
\end{proof}

The next lemma presents the existence and~$p$-integrability of second-order spatial derivatives of weak solutions to equation~\eqref{pdeqgleich1} in the case where~$\mu\in(0,1]$. Respective results have already been established in a similar manner in the super-critical regime~$\critical<p$. For the super-quadratic range~$p\geq 2$, we refer to Lemma 3.1 in~\cite{bogeleinduzaarmarcellini}, and for the sub-quadratic range above the critical exponent~$\frac{2n}{n+2}<p<2$, we refer to Lemma~3.1 in~\cite{singer2015parabolic}. The energy estimate below is derived utilizing the Caccioppoli type inequality~\eqref{energyestimatefürduinl2} and the quantitative estimate~\eqref{differenzabschätzung} for difference quotients, which was obtained in the proof of Lemma~\ref{DuinL2}.
\begin{mylem}[Existence of second order spatial derivatives] \label{energyestimatezweiteableitung} Let~$1<p\leq\critical$,~$\mu \in (0,1]$, and~$u$ be a locally bounded weak solution to~\eqref{pdeqgleich1} under assumptions~\eqref{voraussetzungendiffbarkeit}. Then, there hold
\begin{equation*}
    |\mathcal{V}_\mu^{(p)}(Du)|\in L^2_{\loc}\big(0,T;W^{1,2}_{\loc}(\Omega) \big)
\end{equation*}
and
\begin{equation*}
    |Du|\in L^\infty_{\loc}\big(0,T;L^2_{\loc}(\Omega) \big)\cap L^p_{\loc}\big(0,T;W^{1,p}_{\loc}(\Omega)\big).
\end{equation*}
Moreover, for any~$i\in\{1,...,n\}$, and any cylinder~$Q_\rho(z_0) \Subset \Omega_T$, we have the quantitative estimate
\begin{align}\label{energyzweiteableitung}
    &\esssup\limits_{t\in(t_0-(\frac{\rho}{4})^2,t_0]}\displaystyle\,\fint_{B_{\frac{\rho}{4}}(x_0)}|D_i u(x,t) |^2\, \dx + \rho^2\displaystyle\fiint_{Q_{\frac{\rho}{4}}(z_0)}\big|D_i \mathcal{V}_{\mu}^{(p)}(Du) \big|^2 + |D_i Du|^p\,\dx\dt \\ \nonumber
    & \leq C(1+\rho^2) \displaystyle\fiint_{Q_{\frac{\rho}{2}}(z_0)} (\mu^2+|Du|^2 )^{\frac{p}{2}} \,\dx\dt + \frac{CL^2}{\rho^{2}} \bigg[\displaystyle\fiint_{Q_{\rho}(z_0)}(\mu^2 + |Du| ^2)^{\frac{p}{2}}\,\dx\dt\bigg]^{\frac{2-p}{p}} \\ \nonumber
    &  \quad+ CL \Big(1+\frac{1}{\rho} \Big) \bigg[\displaystyle\fiint_{Q_{\rho}(z_0)}(\mu^2 + |Du|^2)^{\frac{p}{2}}\,\dx\dt\bigg]^{\frac{1}{p}} \nonumber
\end{align}
with~$L\coloneqq \left\| u\right\|_{L^\infty(Q_{\rho}(z_0))}$, and~$C=C(n,p,C_4,C_5,C_6)$. 
\end{mylem}
\begin{proof}
    Similarly to the proof of the previous lemma, let us fix radii $\frac{\rho}{4}= \sigma < r = \frac{1}{2}\rho\leq \rho-h_0$, where $|h|<h_0<\frac{\rho}{2}$. Once again, we employ the Caccioppoli type energy estimate~\eqref{energyestimatefürduinl2inequality} from Lemma~\ref{energyestimatefürduinl2}, which yields
    \begin{align}\label{energyestimateaufkleineremzylinder}
   &\esssup\limits_{t\in(t_0-(\frac{\rho}{4})^2,t_0]} \displaystyle \int_{B_{\frac{\rho}{4}}(x_0)}\big|\Delta_h^{(i)}u \big|^2\,\dx + \displaystyle\iint_{Q_{\frac{\rho}{4}}(z_0)}\big|\Delta_h^{(i)} \mathcal{V}_\mu^{(p)}(Du)\big|^2 \,\dx\dt \\ \nonumber
    & \leq C\Big(1+\frac{1}{\rho^2}\Big) \displaystyle\iint_{Q_{\rho}(z_0)} (\mu^2+|Du|^2 )^{\frac{p}{2}} \,\dx\dt+ \frac{C}{\rho^2} \displaystyle\iint_{Q_{\frac{\rho}{2}}(z_0)} \big|\Delta_h^{(i)} u \big|^2 \,\dx\dt. \nonumber
\end{align}
 Next, we exploit inequality~\eqref{differenzabschätzung} from the proof of Lemma~\ref{DuinL2} to further estimate the second integral in the preceding estimate. This way, we achieve
\begin{align}\label{ersteungleichung}
  &\esssup\limits_{t\in(t_0-(\frac{\rho}{4})^2,t_0]} \displaystyle\int_{B_{\frac{\rho}{4}}(x_0)}\big|\Delta_h^{(i)}u \big|^2\,\dx + \displaystyle\iint_{Q_{\frac{\rho}{4}}(z_0)}\big|\Delta_h^{(i)} \mathcal{V}_\mu^{(p)}(Du)\big|^2 \,\dx\dt \\ \nonumber
    & \leq C\Big(1+\frac{1}{\rho^2}\Big) \displaystyle\iint_{Q_{\rho}(z_0)} (\mu^2+|Du|^2 )^{\frac{p}{2}} \,\dx\dt+ \frac{CL^2}{\rho^{2-n}} \bigg[\displaystyle\fiint_{Q_{\rho}(z_0)}(\mu^2 + |Du|^2)^{\frac{p}{2}}\,\dx\dt\bigg]^{\frac{2-p}{p}} \\ \nonumber
    &  \quad+ \frac{CL}{\rho^{-n}} \Big(1+\frac{1}{\rho} \Big) \bigg[\displaystyle\fiint_{Q_{\rho}(z_0)}(\mu^2 + |Du |^2)^{\frac{p}{2}}\,\dx\dt\bigg]^{\frac{1}{p}}. \nonumber
\end{align}
  To obtain the term involving second order spatial derivatives in~\eqref{energyzweiteableitung}, we apply Lemma~\ref{Vabschätzung} and take~$p$-th power, which implies
$$\big|\Delta_{h}^{(i)}(Du)\big|^p \leq C(n,p)\big|\Delta_h^{(i)}\mathcal{V}^{(p)}_{\mu}(Du)\big|^p \mathcal{D}_{\mu}^{(i)}(h)^{\frac{(2-p)p}{4}}.$$
Integrating and applying Young's inequality with exponents~$(\frac{2}{p},\frac{2}{2-p})$ then yields
\begin{align}\label{zweiteungleichung}
    \displaystyle\iint_{Q_{\frac{\rho}{4}}(z_0)}\big|\Delta_h^{(i)}(Du)\big|^p\,\dx\dt &\leq C \displaystyle\iint_{Q_{\frac{\rho}{4}}(z_0)}\big|\Delta_h^{(i)}\mathcal{V}_\mu^{(p)}(Du)\big|^2\,\dx\dt \\
    &  \quad+ C\displaystyle\iint_{Q_{\rho}(z_0)}(\mu^2+|Du|^2 )^{\frac{p}{2}}\,\dx\dt.\nonumber
\end{align}
By combining estimates~\eqref{ersteungleichung} and~\eqref{zweiteungleichung}, we obtain
\begin{align*}
    \esssup\limits_{t\in(t_0-(\frac{\rho}{4})^2,t_0]}\displaystyle&\int_{B_{\frac{\rho}{4}}(x_0)}\big|\Delta^{(i)}_h u \big|^2\, \dx + \displaystyle\iint_{Q_{\frac{\rho}{4}}(z_0)}\big|\Delta^{(i)}_h \mathcal{V}_{\mu}^{(p)}(Du) \big|^2 + \big|\Delta^{(i)}_h Du\big|^p\,\dx\dt \\
    & \leq C\Big(1+\frac{1}{\rho^2}\Big) \displaystyle\iint_{Q_{\frac{\rho}{2}}(z_0)} (\mu^2+|Du|^2)^{\frac{p}{2}} \,\dx\dt \\ \nonumber
    &  \quad+ \frac{CL^2}{\rho^{2-n}} \bigg[\displaystyle\fiint_{Q_{\rho}(z_0)}(\mu^2 + |Du|^2)^{\frac{p}{2}}\,\dx\dt\bigg]^{\frac{2-p}{p}} \\ \nonumber
    &  \quad+ \frac{CL}{\rho^{-n}} \Big(1+\frac{1}{\rho} \Big) \bigg[\displaystyle\fiint_{Q_{\rho}(z_0)}(\mu^2 + |Du|^2)^{\frac{p}{2}}\,\dx\dt\bigg]^{\frac{1}{p}} \nonumber
\end{align*}
with~$C=C(n,p,C_4,C_5,C_6)$. The claimed energy estimate~\eqref{energyzweiteableitung} now follows by an application of Lemma~\ref{differenzenquotientlemmazwei} and by taking mean values on both sides of the preceding inequality.
\end{proof}
\begin{remark} \upshape
    As mentioned in Remark~\ref{cabhängigvonmu}, it is possible to obtain inequality~\eqref{energyzweiteableitung} comprising a weaker upper bound by employing the energy estimate~\cite[(3.7)]{singer2015parabolic} instead of~\eqref{energyestimatefürduinl2inequality} along with the rough bound~$\mathcal{D}^{(i)}_\mu(h)^{\frac{p-2}{2}} \leq \mu^{2-p}$. However, in this case, the constant~$C$ also depends on the parameter~$\mu$, i.e.~$C = C(n,p,C_4,C_5,C_6,\mu)$. In particular, as~$\mu$ approaches zero,~$C$ and furthermore the right-hand side of~\eqref{energyzweiteableitung} diverges.
\end{remark}
In order to establish local boundedness of the gradient, we need to prove an intermediate result that shows the higher integrability of~$|Du|$ in the sub-quadratic case~$1<p< 2$. To achieve the latter, we exploit energy estimates that incorporate second order spatial derivatives. Since the existence and~$p$-integrability of the spatial derivative have already been established in Lemma~\ref{energyzweiteableitung}, we may proceed without utilizing difference quotients. Moreover, in the range~$p\geq 2$ (according to \cite[Lemma~3.1]{bogeleinduzaarmarcellini}) and in the range~$\critical<p<2$ (according to \cite[Lemma~3.1]{singer2015parabolic}), the existence of second order spatial derivatives has already been established for the case~$\mu\in(0,1]$. Therefore, we consider the full parameter range~$p>1$ in the subsequent proposition. Unlike the assertions in both~\cite[Lemma~3.1]{bogeleinduzaarmarcellini} and~\cite[Lemma~3.1]{singer2015parabolic}, we state the proposition for intrinsic parabolic cylinders instead of standard parabolic cylinders.
\begin{myproposition} \label{intrinsicgrößereexponenten}
Let~$p>1$,~$\mu \in (0,1]$, and~$u$ be a weak solution to~\eqref{pdeqgleich1} under assumptions~\eqref{voraussetzungendiffbarkeit}. Furthermore, assume that~$|Du|\in L^{p+2\alpha}_{\loc}(\Omega_T) \cap L^{2+2\alpha}_{\loc}(\Omega_T)$ for some~$\alpha\geq0$, and in the case~$1<p\leq \critical$ let~$u$ be locally bounded. 
Then, for any~$Q^{(\lambda)}_S(z_0)\Subset \Omega_T$, such that~$\frac{1}{2}S \leq R<S\leq 1$, there holds
 \begin{align} \label{intrinsicenergygrößereexponentenungleichung}
    \esssup\limits_{t\in\Lambda^{(\lambda)}_R(t_0)}& 
\frac{\lambda^{p-2}}{(1+\alpha)R^2} \displaystyle\fint_{B_R(x_0)}(\mu^2 + |Du|^2)^{1+\alpha}\,dx \\ \nonumber
    &+ \displaystyle\fiint_{Q^{(\lambda)}_R(z_0)}(\mu^2 + |Du|^2)^{\frac{p-2}{2}+\alpha}|D^2u|^2\,\dx\dt \\ \nonumber 
    & \leq \frac{C(\alpha +1)}{(S-R)^2}\displaystyle\fiint_{Q^{(\lambda)}_S(z_0)}(\mu^2 + |Du|^2 )^{\frac{p+2\alpha}{2}}\,\dx\dt  \\ \nonumber
&  \quad + \frac{C\lambda^{p-2}}{S^2-R^2}\displaystyle\fiint_{Q^{(\lambda)}_S(z_0)} (\mu^2 + |Du|^2)^{1+\alpha}\,\dx\dt \nonumber
    \end{align}
    with~$C=C(n,p,C_4,C_5,C_6)$.
\end{myproposition}
\begin{proof}
    Due to~$Q_{\rho}\Subset\Omega_T$, we find~$|h|>0$ small enough, such that~$Q^{(\lambda)}_{S+2|h|}\subset\Omega_T$ holds true. The starting point is inequality~\cite[(3.4)]{singer2015parabolic}. Importantly, note that~\cite[(3.4)]{singer2015parabolic} is valid for any~$p>1$. Hence, we are indeed able to treat the full parameter range~$p>1$. By adapting the geometry and applying~\cite[(3.4)]{singer2015parabolic}, we obtain for a.e.~$\tau\in(0,R]$ the following estimate
\begin{align} \label{singerungleichung}
  &\displaystyle\int_{B_S(x_0)} \phi^2(x)\chi(\tau) \displaystyle \int_{0}^{|\Delta^{(i)}_h u(x,\tau) |^2} \Phi(s)\,\ds\dx \\
  &+ \displaystyle\iint_{B_S \times \Lambda^{(\lambda)}_\tau(t_0)} \phi^2 \chi \mathcal{D}^{(i)}_\mu(h)^{\frac{p-2}{2}} \big|\Delta^{(i)}_h (Du) \big|^2 \Phi\big(\big|\Delta^{(i)}_h u \big|^2\big)\,\dx\dt \nonumber \\
  & \leq C \displaystyle\iint_{B_S \times \Lambda^{(\lambda)}_\tau(t_0)} \chi \Phi\big(\big|\Delta^{(i)}_h u \big|^2\big)\big[ \phi^2 \mathcal{D}^{(i)}_\mu(h)^{\frac{p}{2}} +|D\phi|^2 \mathcal{D}^{(i)}_\mu(h)^{\frac{p-2}{2}}   \big|\Delta^{(i)}_h u \big|^2 \big]\,\dx\dt \nonumber \\
&  \quad+ C \displaystyle\iint_{B_S \times \Lambda^{(\lambda)}_\tau(t_0)} \phi^2 \chi \big|\Delta^{(i)}_h u \big|^2 \mathcal{D}^{(i)}_\mu(h)^{\frac{p}{2}} \Phi'\big( \big|\Delta^{(i)}_h u \big|^2 \big) \,\dx\dt \nonumber \\
& \quad + C\displaystyle\iint_{B_S \times\Lambda^{(\lambda)}_\tau(t_0)} \phi^2 \partial_t \chi \displaystyle \int_{0}^{|\Delta^{(i)}_h u |^2} \Phi(s)\,\ds\dx\dt, \nonumber
\end{align}
with~$C=C(n,p,C_4,C_5,C_6)$, where~$\Phi\in W^{1,\infty}(\R)$ denotes a non-negative and non-decreasing Lipschitz function,~$\phi\in C^1_0(B_S(x_0),[0,1])$ is a smooth spatial cut-off function with~$\phi \equiv 1$ on~$B_R(x_0)$ and~$|D\phi|\leq \frac{2}{S-R}$, whereas~$\chi\in W^{1,\infty}\big(\Lambda^{(\lambda)}_S(t_0),[0,1]\big)$ denotes an increasing cut-off function in time satisfying~$\chi(t_0-\lambda^{2-p}S^2)=0$,~$\chi \equiv 1$ on~$\Lambda^{(\lambda)}_R(t_0)$, and~$|\partial_t \chi|\leq \frac{2 \lambda{p-2}}{S^2-R^2}$. In~\eqref{singerungleichung}, we choose~$\Phi = T_k \circ \Phi_\alpha$, where
\begin{align*}
    T_k(s) &\coloneqq \min\{k,s\},\quad s\geq 0 \\
     \Phi_\alpha(s) &\coloneqq (\mu^2 + s)^\alpha,\quad s\geq 0
\end{align*}
for~$k\in\N$, and~$\alpha\geq 0$. Together by utilizing~$\Phi\leq \Phi_\alpha$ and~$\Phi'\leq \Phi_\alpha'$, inequality~\eqref{singerungleichung} translates to
\begin{align*} 
  \displaystyle\int_{B_S(x_0)} &\phi^2(x)\chi(\tau) \displaystyle \int_{0}^{|\Delta^{(i)}_h u(x,\tau) |^2} (T_k \circ \Phi_\alpha)(s)\,\ds\dx \\
  &+ \displaystyle\iint_{B_S \times \Lambda^{(\lambda)}_\tau(t_0)} \phi^2 \chi \mathcal{D}^{(i)}_\mu(h)^{\frac{p-2}{2}} \big|\Delta^{(i)}_h (Du) \big|^2 (T_k\circ \Phi_\alpha)\big(\big|\Delta^{(i)}_h u \big|^2\big)\,\dx\dt  \\
  & \leq C \displaystyle\iint_{Q^{(\lambda)}_S(z_0)} \chi \Phi_\alpha\big(\big|\Delta^{(i)}_h u \big|^2\big)\big[ \phi^2 \mathcal{D}^{(i)}_\mu(h)^{\frac{p}{2}} +|D\phi|^2 \mathcal{D}^{(i)}_\mu(h)^{\frac{p-2}{2}}   \big|\Delta^{(i)}_h u \big|^2 \big]\,\dx\dt  \\
&  \quad+ C \displaystyle\iint_{Q^{(\lambda)}_S(z_0)} \phi^2 \chi \big|\Delta^{(i)}_h u \big|^2 \mathcal{D}^{(i)}_\mu(h)^{\frac{p}{2}} \Phi'_\alpha\big( \big|\Delta^{(i)}_h u \big|^2 \big) \,\dx\dt  \\
&  \quad+ C\displaystyle\iint_{Q^{(\lambda)}_S(z_0)} \phi^2 \partial_t \chi \displaystyle \int_{0}^{|\Delta^{(i)}_h u |^2} \Phi_\alpha(s)\,\ds\dx\dt
\end{align*}
for a.e.~$\tau\in(0,R]$. Next, we take the essential supremum with respect to~$\tau\in(0,R]$ for the first term on the left-hand side, and estimate the second quantity on the left-hand side in the inequality above by the choice~$t=R$ and a reduction of the domain of integration. Exploiting the assumptions made on~$\phi$ and~$\chi$, we return to employ~$t$ instead of~$\tau$ as the time variable and obtain
\begin{align*}
&\esssup\limits_{t\in \Lambda^{(\lambda)}_R(t_0)}\displaystyle\int_{B_R(x_0)} \displaystyle \int_{0}^{|\Delta^{(i)}_h u(x,t) |^2} (T_k \circ \Phi_\alpha)(s)\,\ds\dx \\
  &+ \displaystyle\iint_{Q^{(\lambda)}_R(z_0)} \mathcal{D}^{(i)}_\mu(h)^{\frac{p-2}{2}} \big|\Delta^{(i)}_h (Du) \big|^2 (T_k\circ \Phi_\alpha)\big(\big|\Delta^{(i)}_h u \big|^2\big)\,\dx\dt  \\
  & \leq C \displaystyle\iint_{Q^{(\lambda)}_S(z_0)} \big(\mu^2 + \big|\Delta^{(i)}_h u \big|^2\big)^\alpha\big[ \mathcal{D}^{(i)}_\mu(h)^{\frac{p}{2}} +\frac{1}{(S-R)^2} \mathcal{D}^{(i)}_\mu(h)^{\frac{p-2}{2}}   \big|\Delta^{(i)}_h u \big|^2 \big]\,\dx\dt  \\
&  \quad+ C\alpha \displaystyle\iint_{Q^{(\lambda)}_S(z_0)} \big|\Delta^{(i)}_h u \big|^2 \mathcal{D}^{(i)}_\mu(h)^{\frac{p}{2}} \big(\mu^2 + \big|\Delta^{(i)}_h u \big|^2 \big)^{\alpha -1} \,\dx\dt  \\
&  \quad+ \frac{C\lambda^{p-2}}{S^2-R^2}\displaystyle\iint_{Q^{(\lambda)}_S(z_0)} \big(\mu^2 + \big|\Delta^{(i)}_h u \big|^2\big)^{1+\alpha}\,\dx\dt. 
\end{align*}
Firstly, in the case where~$p\leq2$, we estimate~$\mathcal{D}^{(i)}_\mu(h)^{\frac{p-2}{2}} \leq (\mu^2 + |Du|^2 )^{\frac{p-2}{2}}$. In the case where~$p>2$, we exploit the assumption made on $|h|$ and apply Lemma~\ref{differenzenquotientlemmadrei} to estimate
$$\displaystyle\iint_{Q^{(\lambda)}_S(z_0)} \mathcal{D}^{(i)}_\mu(h)^{\frac{p-2}{2}}\,\dx\dt \leq C(n,p) \displaystyle\iint_{Q^{(\lambda)}_{S+2|h|}(z_0)} (\mu^2 + |Du|^2 )^{\frac{p-2}{2}}\,\dx\dt.$$
Both sets of estimates ultimately result in
\begin{align*}
     \esssup\limits_{t\in \Lambda^{(\lambda)}_R(t_0)}&\displaystyle\int_{B_R(x_0)} \displaystyle \int_{0}^{|\Delta^{(i)}_h u(x,t) |^2} (T_k \circ \Phi_\alpha)(s)\,\ds\dx \\
  &+ \displaystyle\iint_{Q^{(\lambda)}_R(z_0)}\mathcal{D}^{(i)}_\mu(h)^{\frac{p-2}{2}} \big|\Delta^{(i)}_h (Du) \big|^2 (T_k\circ \Phi_\alpha)\big(\big|\Delta^{(i)}_h u \big|^2\big)\,\dx\dt  \\
  & \leq C(\alpha+1) \displaystyle\iint_{Q^{(\lambda)}_S(z_0)} \big(\mu^2 + \big|\Delta^{(i)}_h u \big|^2\big)^\alpha \mathcal{D}^{(i)}_{\mu}(h)^{\frac{p}{2}} \,\dx\dt  \\
&  \quad+  \frac{C\lambda^{p-2}}{(S-R)^2} \displaystyle\iint_{Q^{(\lambda)}_{S+|h|}(z_0)}   ( \mu^2 + |Du|^2 )^{\frac{p-2}{2}} \big(\mu^2 + \big|\Delta^{(i)}_h u \big|^2 \big)^{\alpha}\big|\Delta^{(i)}_h u \big|^2 \,\dx\dt  \\
&  \quad+ \frac{C\lambda^{p-2}}{S^2-R^2}\displaystyle\iint_{Q^{(\lambda)}_S(z_0)} \big(\mu^2 + \big|\Delta^{(i)}_h u \big|^2\big)^{1+\alpha}\,\dx\dt. 
\end{align*}
Next, we aim to pass to the limit~$|h|\downarrow 0$. Due to Lemma~\ref{energyestimatezweiteableitung}, there holds~$|D^2u|\in L^p_{\loc}(\Omega_T)$, and according to our assumptions, we have~$|Du|\in L^{p+2\alpha}_{\loc}(\Omega_T) \cap L^{2+2\alpha}_{\loc}(\Omega_T)$ at our disposal. Therefore, after passing to a subsequence, we have the convergence of~$\Delta^{(i)}_h (Du)$ to~$D_i Du$ a.e. and the convergence of~$\Delta^{(i)}_h u$ to~$D_i u$ a.e. in~$Q_{\rho}(z_0)$ as~$|h|$ converges to zero. Now, once again, we need to distinguish between two regimes. First, we devote ourselves to the sub-quadratic case~$p<2$, which requires a more delicate treatment than the super-quadratic case~$p\geq 2$. Passing to the subsequence from above, we first apply the variant of the dominated convergence theorem from Lemma~\ref{dominatedconvergence} with the choice 
\begin{align*}
    g_h &= \mu^{p-2}\big(\mu^2 + \big|\Delta^{(i)}_h u \big|^2 \big)^{1+\alpha}, &&f_h= (\mu^2+|Du|^2)^{\frac{p-2}{2}}\big(\mu^2 + \big|\Delta^{(i)}_h u \big|^2 \big)^{\alpha} |\Delta^{(i)}_h u|^2 \\
    g &= \mu^{p-2}(\mu^2 + |D_i u|^2 )^{1+\alpha} , &&\,\,\,f=(\mu^2+|Du|^2 )^{\frac{p-2}{2}}(\mu^2 + |D_i u| ^2)^{\alpha} |D_iu|^2.
\end{align*}
Together with an application of Fatou's lemma, we obtain
\begin{align} \label{ersterfall}
     \esssup\limits_{t\in \Lambda^{(\lambda)}_R(t_0)}&\displaystyle\int_{B_R(x_0)} \phi^2(x)\chi(t) \displaystyle \int_{0}^{|D_i u(x,t)|} (T_k \circ \Phi_\alpha)(s)\,\ds\dx \\ \nonumber
  &+ \displaystyle\iint_{Q^{(\lambda)}_R(z_0)} (\mu^2 + |D_iu|^2)^{\frac{p-2}{2}} |D_i Du|^2 (T_k\circ \Phi_\alpha)(|D_iu|^2)\,\dx\dt  \\ \nonumber
  & \leq  \frac{C(\alpha +1)}{(S-R)^2} \displaystyle\iint_{Q^{(\lambda)}_S(z_0)}   ( \mu^2 + |Du|^2 )^{\frac{p+2\alpha}{2}}\,\dx\dt  \\ \nonumber
& \quad + \frac{C\lambda^{p-2}}{S^2-R^2}\displaystyle\iint_{Q^{(\lambda)}_R(z_0)}  (\mu^2 + |Du|^2)^{1+\alpha}\,\dx\dt, \nonumber
\end{align}
where we additionally exploited the assumption~$S\leq 1$. The inequality mentioned above remains valid when~$p=2$. Therefore, let us assume~$p>2$. We apply Young's inequality with exponents~$(\frac{p+2\alpha}{p-2},\frac{p+2\alpha}{2+2\alpha})$ and employ Lemma~\ref{differenzenquotientlemmadrei}, to estimate
\begin{align} \label{zweiterfall}
     & \displaystyle\iint_{Q^{(\lambda)}_{S+|h|}(z_0)}  ( \mu^2 + |Du|^2 )^{\frac{p-2}{2}} \big(\mu^2 + \big|\Delta^{(i)}_h u \big|^2 \big)^{\alpha}\big|\Delta^{(i)}_h u \big|^2 \,\dx\dt \\ \nonumber
    &\leq  \displaystyle\iint_{Q^{(\lambda)}_{S+|h|}(z_0)}  ( \mu^2 + |Du|^2 )^{\frac{p-2}{2}} \big(\mu^2 + \big|\Delta^{(i)}_h u \big|^2 \big)^{1+\alpha}\,\dx\dt \\ \nonumber
    &\leq  \displaystyle\iint_{Q^{(\lambda)}_{S+|h|}(z_0)}  \big[ ( \mu^2 + |Du|^2 )^{\frac{p+2\alpha}{2}} + \big(\mu^2 + \big|\Delta^{(i)}_h u \big|^2 \big)^{\frac{p+2\alpha}{2}} \big] \,\dx\dt \\ \nonumber
    &\leq  \displaystyle\iint_{Q^{(\lambda)}_{S+2|h|}(z_0)} ( \mu^2 + |Du|^2)^{\frac{p+2\alpha}{2}} \,\dx\dt. \nonumber
\end{align}
We insert inequality~\eqref{zweiterfall} into~\eqref{ersterfall}. Due to Lemma~\eqref{differenzenquotientlemmazwei}, we are in position to pass to the limit~$|h|\to 0$ and obtain in the full range~$p>1$ the estimate
\begin{align*}
     \esssup\limits_{t\in \Lambda^{(\lambda)}_R(t_0)}&\displaystyle\int_{B_R(x_0)} \displaystyle \int_{0}^{|D_i u(x,t)|} (T_k \circ \Phi_\alpha)(s)\,\ds\dx \\ 
  &+ \displaystyle\iint_{Q^{(\lambda)}_R(z_0)} (\mu^2 + |D_iu|^2)^{\frac{p-2}{2}} |D_i Du|^2 (T_k\circ \Phi_\alpha)(|D_iu|)\,\dx\dt  \\ 
  & \leq \frac{C(\alpha +1)}{(S-R)^2}  \displaystyle\iint_{Q^{(\lambda)}_S(z_0)}   ( \mu^2 + |Du|^2)^{\frac{p+2\alpha}{2}}\,\dx\dt  \\ 
&  \quad+ \frac{C\lambda^{p-2}}{S^2-R^2}\displaystyle\iint_{Q^{(\lambda)}_S(z_0)}  (\mu^2 + |Du|^2)^{1+\alpha}\,\dx\dt. 
\end{align*}
By passing to the limit~$k\to\infty$ in the previous inequality and utilizing the fact that $T_k \circ \Phi_\alpha \to \Phi_\alpha$ in~$Q_\rho(z_0)$ as~$k\to\infty$, along with another application of Fatou's lemma, we obtain
\begin{align*}
     \esssup\limits_{t\in \Lambda^{(\lambda)}_R(t_0)}&\displaystyle\int_{B_R(x_0)} \displaystyle \int_{0}^{|D_i u(x,t)|} \Phi_{\alpha}(s)\,\ds\dx \\ 
  &+ \displaystyle\iint_{Q^{(\lambda)}_R(z_0)} (\mu^2 + |D_iu|^2 )^{\frac{p-2}{2}} |D_i Du|^2 \Phi_{\alpha}(|D_iu|)\,\dx\dt  \\ 
  & \leq \frac{C(\alpha +1)}{(S-R)^2}\displaystyle\iint_{Q^{(\lambda)}_S(z_0)}  ( \mu^2 + |Du|^2)^{\frac{p+2\alpha}{2}}\,\dx\dt  \\ 
&  \quad+\frac{C\lambda^{p-2}}{S^2-R^2}\displaystyle\iint_{Q^{(\lambda)}_S(z_0)}  (\mu^2 + |Du|^2)^{1+\alpha}\,\dx\dt. 
\end{align*}
Exploiting the definition of~$g_\alpha$ and summing over~$i=1,..,n$, we eventually derive the claimed energy estimate~\eqref{intrinsicenergygrößereexponentenungleichung} with~$C=C(n,p,C_4,C_5,C_6)$ by taking mean values on both sides of the inequality.
\end{proof}
From our energy estimate~\eqref{intrinsicenergygrößereexponentenungleichung}, a series of results, including higher integrability of the spatial derivative of weak solutions in the sub-quadratic case and boundedness of the spatial derivative for weak solutions to~\eqref{pdeqgleich1} in the full range of parameter~$p>1$, will follow. The underlying plan is the employment of a Moser iteration procedure to demonstrate the local boundedness of spatial derivatives of weak solutions. In the case of super-quadratic equations where~$p>2$, we are in position to start with the energy estimate involving higher integrability from Proposition~\ref{intrinsicgrößereexponenten} and then proceed with Moser's iteration. However, in the sub-critical range~$1<p\leq\critical$, the spatial derivative of weak solutions are required to be integrable to any arbitrarily large power~$m$, i.e.~$|Du|\in L^{m}_{\loc}(\Omega_T)$ for any~$m>1$. The following lemma establishes that this assumption is indeed always satisfied. 
\begin{mylem}\label{higherintegrabilitylemma}
    Let~$1<p\leq 2$,~$\mu\in(0,1]$, and~$u$ be a locally bounded weak solution to~\eqref{pdeqgleich1} under assumptions~\eqref{voraussetzungendiffbarkeit}. Then, there holds~$|Du|\in L^{m}_{\loc}(\Omega_T)$ for any~$m>1$. Furthermore, for~$m>p+1$ and any cylinder~$Q^{(\lambda)}_{2\rho}(z_0)\Subset\Omega_T$, we have the quantitative estimate
    \begin{align} \label{higherintegrability}
        \displaystyle\fiint_{Q^{(\lambda)}_{\rho}(z_0)}&(\mu^2 + |Du|^2)^{\frac{m}{2}}\,\dx\dt \\
        &\leq C \bigg[\lambda \Big(\frac{\omega}{\rho \lambda}\Big)^{\frac{2}{p}}+\frac{\omega}{\rho}+\mu  \bigg]^{m-p} \displaystyle\fiint_{Q^{(\lambda)}_{2\rho}(z_0)}(\mu^2 + |Du|^2 )^{\frac{p}{2}}\,\dx\dt \nonumber
    \end{align}
    with~$C=C(m,n,p,C_4,C_5,C_6)$, and~$\omega\coloneqq \essosc\limits_{Q^{(\lambda)}_{2\rho}(z_0)}u$.
\end{mylem}
\begin{proof}
   The argument is similar as in~\cite[Lemma~9.5]{gradientholder}. As we have already established that~$|Du|\in L^2_{\loc}(\Omega_T)$ in Lemma~\ref{DuinL2}, instead of advancing in the same fashion to the initial part of~\cite[Proof of Lemma~9.5]{gradientholder}, we directly use the energy estimate~\ref{intrinsicgrößereexponenten} as ingredient. We omit the details.
\end{proof}
At this point, all the required components are available to establish the local boundedness of spatial derivatives of weak solutions to equation~\eqref{pdeqgleich1}.
\begin{proof}[\textbf{\upshape Proof of Propositions~\ref{gradientboundednesssubcritical} and~\ref{gradientboundednesssupercritical}}]
    We utilize the energy estimate stated in Proposition~\ref{intrinsicgrößereexponenten} and in the sub-quadratic case the higher integrability property mentioned in Lemma~\ref{higherintegrabilitylemma}, both of which are at our disposal. Once Proposition~\ref{intrinsicgrößereexponenten} and Lemma~\ref{higherintegrabilitylemma} are established, the arguments are the same as in~\cite[Chapter~9]{gradientholder}. Therefore, we omit the details.
\end{proof}

\begin{remark} \upshape
In conclusion, this section presents two novelties. Firstly, it demonstrates the local boundedness of the spatial derivative of weak solutions to~\eqref{pdeqgleich1} in the sub-critical range~$1<p\leq \critical$. Secondly, the quantitative gradient estimates from Proposition~\ref{gradientboundednesssubcritical} and Proposition~\ref{gradientboundednesssupercritical} presented here have only been derived for the prototype equation~\eqref{pdemitplaplace} in~\cite{gradientholder} at the current stage.
\end{remark}

\section{Schauder estimates}\label{sec:schauderestimates}
In this section, we will once again consider weak solutions to parabolic~$p$-Laplacian type equations of the form
\begin{align} \label{pdeqwieder1}
        \partial_t u - \divv A(x,t,Du)=0
    \qquad\mbox{in $\Omega_T$,}
\end{align} 
where~$\mu\in[0,1]$ and~$A$ satisfies structure conditions~\eqref{voraussetzungen}. Throughout this section, weak solutions may still exhibit sign changes. Our objective is to derive a local gradient Hölder regularity result for weak solutions to~\eqref{pdeqwieder1}. To begin with, we establish an \textit{a priori} gradient Hölder estimate for weak solutions when~$A$ is independent of the spatial variable~$x$. By employing a comparison argument, we then extend this regularity result to weak solutions of the more general equation~\eqref{pdeqwieder1} with~$A$ satisfying the given structure conditions~\eqref{voraussetzungen}. The assumed Hölder continuity of~$A$ with respect to the spatial variable~$x$ is crucial in transferring regularity from weak solutions with the vector field~$A$ independent of the spatial variable~$x$ to weak solutions of more general structure. Finally, we address the case where~$\mu=0$ in Section~\ref{subsec:approximation}, leading to the following theorem as a main result.
\begin{mytheorem} [Schauder estimate for parabolic $p$-Laplacian type equations] \label{gradientholderschaudermugleich0}
    Let~$u$ be a bounded weak solution to~\eqref{pdeqwieder1} under assumptions~\eqref{voraussetzungen} in the range~$p>1$ and~$\mu\in[0,1]$. Then, there exist a Hölder exponent~$\alpha_{0}=\alpha_0(n,p,C_1,C_2,\alpha) \in (0,1)$ and a constant~$C=C(n,p,C_1,C_2,C_3,\alpha)\geq 1$, such that there holds
    $$Du\in C^{\alpha_0,\alpha_0/2}_{\loc}(\Omega_T,\R^n).$$
    Moreover, for any $E\subset\Omega_T$ such that $r\coloneqq \frac{1}{4}\dist_p(E,\partial_p \Omega_T)>0$, and any $z_1,z_2\in E$, there hold
    the quantitative gradient estimate
    \begin{align}\label{gradientschrankeschaudermugleich0}
        \esssup\limits_{E}|Du| \leq C \bigg[\frac{\omega}{r} + \Big(\frac{\omega}{r} \Big)^{\frac{2}{p}} +\mu\bigg] \eqqcolon \lambda
    \end{align}
    and the gradient Hölder estimate
\begin{align}\label{gradientholderschauderinequalitymugleich0}
    |Du(z_1)-Du(z_2)| \leq C \lambda \Bigg[ \frac{d^{(\lambda)}_p(z_1,z_2)}{\min\big\{1,\lambda^{\frac{p-2}{2}}\big\}r} \Bigg]^{\alpha_0},
\end{align}
where~$\omega\coloneqq\essosc\limits_{\Omega_T} u$.
\end{mytheorem}

\subsection{An \textit{a priori} gradient estimate} \label{subsec:apriori}
In this paragraph, we pursue our goal of an~\textit{a priori} gradient estimate for weak solutions to parabolic~$p$-Laplacian type equations. We specifically consider equations of the following type
\begin{align}
        \partial_t u - \divv A(t,Du)=0
    \qquad\mbox{in $\Omega_T$,} \label{pdeohnex}
\end{align} 
 where the the vector field~$A$ is independent of the spatial variable~$x$, but still satisfies structure conditions~$\eqref{voraussetzungen}_1$ and~$\eqref{voraussetzungen}_2$ in the case~$\mu\in(0,1]$. This intermediate step allows us to obtain an \textit{a priori} gradient regularity result in the spirit of Campanato for weak solutions to~\eqref{pdeohnex},~which is one of the main contributions of this article. Throughout this section, we will maintain the geometry of intrinsic parabolic cylinders to derive homogeneous estimates. \,\\

 In the following, we consider~$Q^{(\lambda)}_{2\rho}(z_0)\Subset\Omega_T$ and a parameter~$\lambda\geq \frac{\mu}{L}>0$ large enough, such that the estimate
\begin{align} \label{schrankelambda}
  \esssup\limits_{Q^{(\lambda)}_{2\rho}(z_0)}(\mu^2+|Du|^{2}_\infty)\leq L^2\lambda^2
\end{align}
holds true for some fixed constant~$L\geq 1$.~During the further course of this section, the abbreviation~$\lambda_\mu \coloneqq \sqrt{L^2\lambda^2-\mu^2}$ is employed. As a result, assumption~\eqref{schrankelambda} implies the bound~$|Du|_\infty \leq \lambda_\mu$ on~$Q^{(\lambda)}_{2\rho}(z_0)$. Next, for a parameter~$\nu\in(0,1)$ and any~$i=1,\ldots,n$ we denote the super-level set of~$|D_i u|$ on~$Q^{(\lambda)}_{\rho}(z_0)$ to the level~$(1-\nu)\lambda_\mu$ by
\begin{align} \label{superniveaumenge}
    E^{(\lambda),i}_{\rho}(z_0,\nu) \coloneqq \big\{z\in Q^{(\lambda)}_\rho(z_0):\, 
|D_iu(z)|>(1-\nu)\lambda_\mu \big\}.
\end{align}
The main result in this section is the following Campanato type estimate.
\begin{mytheorem} \label{campanatoaprioritheorem}
    Let~$p>1$,~$\mu\in(0,1]$, and~$L\geq 1$. There exist a constant~$C>0$ and a Hölder exponent~$\beta\in(0,1)$, both depending on~$n,p,L,C_1,C_2$, such that: for any weak solution~$u$ to~\eqref{pdeohnex} and any~$Q^{(\lambda)}_{\rho}(z_0)\Subset \Omega_T$ under assumption~\eqref{schrankelambda},
    there holds
    \begin{align} \label{campanatoaprioriungleichung}
        \displaystyle\fiint_{Q^{(\lambda)}_{r}(z_0)} \big|Du - (Du)^{(\lambda)}_{z_0;r} \big|^p\,\dx\dt \leq C\Big(\frac{r}{\rho}\Big)^{p \beta} \lambda^p_\mu
    \end{align}
    for any~$r\in(0,\rho]$. 
\end{mytheorem}

The proof of Theorem~\ref{campanatoaprioritheorem} consists of two subsequent propositions, each addressing the non-degenerate and degenerate regimes respectively. In the non-degenerate regime, the measure condition~$\big|Q^{(\lambda)}_{\rho}(z_0) \setminus E^{(\lambda),i}_{\rho}(z_0,\nu) \big| < \nu \big|Q^{(\lambda)}_{\rho}(z_0) \big|$ holds true for at least one~$i\in\{1,\ldots,n\}$, while in the degenerate regime, this measure condition is satisfied for any~$i=1,\ldots,n$ with the reversed inequality. In the non-degenerate regime, we examine a subset of~$Q^{(\lambda)}_{\rho}(z_0)$ which is large in measure, where~$|Du|$ is close to its essential supremum and consider the~$L^2$-excess of~$Du$ on~$Q^{(\lambda)}_\rho (z_0)$, given by
\begin{align} \label{excess}
    \Phi_\lambda(z_0,\rho) \coloneqq \displaystyle\fiint_{Q^{(\lambda)}_\rho (z_0)} \big| Du-(Du)^{(\lambda)}_{z_0,\rho} \big| ^2 \,\dx\dt.
\end{align}
Conversely, in the degenerate regime, we consider a subset of points with~$|Du|$ far from its supremum, which is large in measure. 
\,\\

The first proposition deals with the non-degenerate regime.
\begin{myproposition}\label{non-degenerateproposition}
    Let~$\mu\in(0,1]$,~$L\geq 1$, and~$\lambda\geq\frac{\mu}{L}$. There exist an exponent~$\beta\in(0,1)$, a parameter~$\nu \in (0,\frac{1}{4}]$, and a constant~$C\geq 1$, all depending on~$n,p,L,C_1,C_2$, such that: for any weak solution~$u$ to~\eqref{pdeohnex} under assumption~\eqref{schrankelambda}, and any~$Q^{(\lambda)}_{\rho}(z_0)\Subset \Omega_T$ on which the measure condition
    \begin{align} \label{non-degeneratemeasurecondition}
        \big|Q^{(\lambda)}_{\rho}(z_0) \setminus E^{(\lambda),i}_{\rho}(z_0,\nu) \big| < \nu \big|Q^{(\lambda)}_{\rho}(z_0) \big|
    \end{align}
     is satisfied for at least one~$i\in\{1,\ldots,n\}$, there holds the excess-decay estimate
     \begin{equation} \label{excesseins}
         \Phi_\lambda(z_0,r) \leq C \Big(\frac{r}{\rho}\Big)^{2 \beta} \Phi_\lambda(z_0,\rho)
     \end{equation}
     for any~$r\in(0,\rho]$. Moreover, the limit
     \begin{align} \label{lebesgue-representative}
         \Gamma_{z_0} \coloneqq \lim\limits_{r\downarrow 0} (Du )^{(\lambda)}_{z_0;r}
     \end{align}
  exists and for any~$r\in(0,\rho]$ the following excess-decay estimate
    \begin{align} \label{excess-decay}
        \displaystyle\fiint_{Q^{(\lambda)}_{r}(z_0)} |Du - \Gamma_{z_0}|^2\,\dx\dt \leq  C\Big(\frac{r}{\rho}\Big)^{2 \beta} \lambda^2_\mu
    \end{align}
    holds true. Furthermore, we have the bounds
    \begin{align} \label{bounds}
       \frac{1}{4}\lambda \leq |Du|_\infty \leq \lambda_\mu \qquad \mbox{a.e. in~$Q_{\frac{\rho}{2}}^{(\lambda)}(z_0)$}.
    \end{align}
\end{myproposition}


\begin{remark} \label{remarknachnondegenerateproposition} \upshape
    We would like to emphasize that both excess-decay estimates~\eqref{excesseins} and~\eqref{excess-decay} continue to hold true without the presence of the constant~$C=C(n,p,L,C_1,C_2)\geq 1$. This matter of fact can be demonstrated by pursuing a similar approach as presented in~\cite[6. The non-degenerate regime]{degeneratesystems}. Essentially, we are able to establish a comparable energy estimate to that of~\cite[Proposition 3.2]{degeneratesystems}, and subsequently derive the existence of a parameter~$\delta=\delta(n,p,L,C_1,C_2)\in(0,\frac{1}{16}]$, such that there holds~$\Phi_\lambda(z_0,\rho) \leq \delta \lambda^2_\mu$ as long as the parameter~$\nu=\nu(n,p,L,C_1,C_2)\in(0,\frac{1}{4}]$ is chosen sufficiently small in dependence of the parameter~$\delta$. In this manner, a lower bound can be established for the mean value~$(Du)^{\lambda}_{z_0,\rho}$ and also for~$\Gamma_{z_0}$. However, for the sake of brevity, we have chosen to present both excess-decay estimates~\eqref{excesseins} and~\eqref{excess-decay} with the additional constant~$C=C(n,p,L,C_1,C_2)\geq 1$ and refrain from stating a lower bound for~$(Du)^{\lambda}_{z_0,\rho}$ and for~$\Gamma_{z_0}$.
\end{remark}

The second proposition deals with the degenerate regime.
\begin{myproposition}\label{degenerateproposition}
    Let~$\mu\in(0,1]$,~$L\geq 1$,~$\lambda\geq \frac{\mu}{L}$, and~$\nu\in(0,\frac{1}{2}]$. Then, there exists~$\kappa = \kappa(n,p,L,C_1,C_2,\nu) \in[\frac{1}{2},1)$, such that: for any weak solution~$u$ to~\eqref{pdeohnex} under assumption~\eqref{schrankelambda}, any~$Q^{(\lambda)}_{2\rho}(z_0)\Subset \Omega_T$ on which the measure condition
    \begin{align} \label{degeneratemeasurecondition}
        \big|Q^{(\lambda)}_{\rho}(z_0) \setminus E^{(\lambda),i}_{\rho}(z_0,\nu) \big| \geq \nu \big|Q^{(\lambda)}_{\rho}(z_0) \big|,
    \end{align}
     is satisfied for any~$i\in\{1,\ldots,n\}$, there holds
     \begin{align} \label{sup-decay}
         \esssup\limits_{Q^{(\lambda)}_{\Tilde{\nu}\rho}(z_0)} |Du|_\infty \leq \kappa \lambda_\mu
     \end{align}
     with~$\Tilde{\nu}\coloneqq \frac{\sqrt{\nu}}{2}$.
\end{myproposition}

\begin{proof}[\textbf{\upshape Proof of Theorem~\ref{campanatoaprioritheorem}}]
Once Propositions~\ref{non-degenerateproposition} and~\ref{degenerateproposition} have been established, Theorem~\ref{campanatoaprioritheorem} can be inferred by combining the results of the degenerate and the non-degenerate regime. Due to similarity we may refer the reader to~\cite[Proof of Theorem~1.3]{degeneratesystems} and omit the details here.
\end{proof}

\subsubsection{The non-degenerate regime} \label{subsubsec:nondegenerateregime} \,\\
The objective in this section is to establish the proof of Proposition~\ref{non-degenerateproposition}. The strategy consists of exploiting existing excess-decay results from~\cite[Proposition~3.8]{kuusi2013gradient} in the case of sub-quadratic growth~$p<2$ and from~\cite[Lemma~3.2]{kuusi2014wolff} in the case of super-quadratic growth~$p\geq 2$. Instead of considering the positive parameter~$\lambda>0$, we aim to directly obtain an excess-decay estimate stated in~\eqref{excess-decay} by utilizing the actual bound~$\lambda_\mu$ defined in~\eqref{schrankelambda}. This approach eliminates the need for an intermediate step as seen in previous works such as~\cite[Proposition~3.5]{kuusi2014wolff} in the super-quadratic case and~\cite[Proposition~3.9]{kuusi2013gradient} in the sub-quadratic case. 
 Recall that in the course of this section, the set of assumptions~\eqref{schrankelambda} and~\eqref{non-degeneratemeasurecondition} is at our disposal. \,\\

As mentioned above, the key component to establish Proposition~\ref{non-degenerateproposition} lies in the following quantitative excess-decay estimate for the excess of~$Du$ on the intrinsic cylinder~$Q^{(\lambda)}_\rho(z_0)$. Additionally, there holds a lower bound for~$|Du|_\infty$ on the smaller intrinsic cylinder~$Q^{(\lambda)}_{\frac{\rho}{2}}(z_0)$.
\begin{mylem}[Quantitative excess-decay estimate] \label{quantitativexcess}  Let~$u$ be a weak solution to~\eqref{pdeohnex} and~$\theta\in(0,\frac{1}{2}]$. There exist a constant~$\Tilde{C}\geq 1$, an exponent~$\Tilde{\beta}\in(0,1)$, and a parameter~$\Tilde{\nu}\in(0,\frac{1}{4}]$, all depending on~$n,p,L,C_1,C_2$, such that: if the measure-theoretic condition~\eqref{non-degeneratemeasurecondition} is satisfied for at least one~$i\in\{1,\ldots,n\}$ with parameter~$\Tilde{\nu}\in (0,\frac{1}{4}]$, then there holds the quantitative excess-decay estimate
    \begin{equation} \label{excessdecay}
        \Phi_\lambda(z_0,\theta\rho) \leq \Tilde{C} \theta^{2\Tilde{\beta}} \Phi_\lambda(z_0,\rho).
        \end{equation}
  Moreover, we have the lower bound
    \begin{equation} \label{lowerbound}
        |Du|_\infty \geq \frac{\lambda}{4} \qquad\mbox{a.e. in~$Q^{(\lambda)}_{\frac{\rho}{2}}(z_0)$}.    
    \end{equation}
\end{mylem}
\begin{proof}
    Firstly, we note that assumption~\eqref{schrankelambda} in particular implies the bound
\begin{equation*}
      \mu + \esssup\limits_{Q^{(\lambda)}_{2\rho}(z_0)} |Du|_\infty \leq 2L\lambda.
\end{equation*}
As a result, the excess-decay results from~\cite[Proposition 3.8]{kuusi2013gradient} for the sub-quadratic case and~\cite[Proposition 3.3]{kuusi2014wolff} for the super-quadratic case are at our disposal, as the assumption~\cite[(3.27)]{kuusi2013gradient} resp.~\cite[(3.10)]{kuusi2014wolff} is satisfied with parameter~$A\coloneqq 2L\geq 1$. Furthermore, we observe that the measure-theoretic condition~\cite[(3.28)]{kuusi2013gradient} in the sub-quadratic case and~\cite[(3.11)]{kuusi2014wolff} in the super-quadratic case holds true for a certain parameter~$\nu\in(0,\frac{1}{8}]$. This can be deduced from our measure-theoretic assumption~\eqref{non-degeneratemeasurecondition}, which ensures the existence of a parameter~$\nu\coloneqq\frac{1}{2}\Tilde{\nu}\in(0,\frac{1}{8}]$ depending on the data~$n,p,L,C_1,C_2$, such that we either have
\begin{align*}
    |\{z\in Q^{(\lambda)}_\rho(z_0):\,D_i u < \textstyle{\frac{1}{2}} \lambda_\mu\}| &\leq \mbox{$\frac{1}{2}$} |\{z\in Q^{(\lambda)}_\rho(z_0):\,|D_i u| < \textstyle{\frac{1}{2}}\lambda_\mu\}| \\
    &\leq \nu|Q^{(\lambda)}_\rho(z_0)|
\end{align*}
or the opposite measure-theoretic information
\begin{align*} 
    |\{z\in Q^{(\lambda)}_\rho(z_0):\, -D_i u < \textstyle{\frac{1}{2}} \lambda_\mu\}| \leq \nu |Q^{(\lambda)}_\rho(z_0)| 
\end{align*}
holds true for at least one~$i\in\{1,...,n\}$. By applying the excess-decay results from~\cite[Proposition 3.8]{kuusi2013gradient} in the sub-quadratic case and~\cite[Proposition 3.3]{kuusi2014wolff} in the super-quadratic case with~$q=2$, we obtain the estimate~\eqref{excessdecay}. The lower bound~\eqref{lowerbound} follows as a consequence of the results in~\cite[Proposition 3.7]{kuusi2013gradient} and \cite[Proposition 3.1 \& Proposition 3.2]{kuusi2014wolff}.
\end{proof}

At this point, we are in position to provide the proof of Proposition~\ref{non-degenerateproposition} and conclude the non-degenerate regime.

\begin{proof}[\textbf{\upshape Proof of Proposition~\ref{non-degenerateproposition}}]
Let~$\Tilde{C}\geq 1$ and~$\Tilde{\beta}\in(0,1)$, both depending on~$n,p,L,C_1,C_2$, denote the constant and the exponent from the preceding Lemma~\ref{quantitativexcess}. We observe that the first excess-decay estimate~\eqref{excesseins} readily follows from an application of Lemma~\ref{quantitativexcess}. This can be demonstrated by setting~$\theta = \frac{r}{\rho}$, where~$r\in(0,\frac{\rho}{2}]$. With this choice, there holds~$\theta\in(0,\frac{1}{2}]$, enabling us to apply Lemma~\ref{quantitativexcess} and to derive the bound~\eqref{excesseins} with constant~$\Tilde{C}$ and exponent~$\Tilde{\beta}$. Conversely, in the case where~$r\in(\frac{\rho}{2},\rho]$, we exploit the~$L^2$-minimality of the mean value and again obtain the desired excess-decay estimate
$$\Phi_\lambda(z_0,r) \leq C(n)\Phi_\lambda(z_0,\rho) \leq 2^{2\beta} C(n) \Big(\frac{r}{\rho}\Big)^{2\beta}\Phi_\lambda(z_0,\rho) \leq C(n) \Big(\frac{r}{\rho}\Big)^{2\beta}\Phi_\lambda(z_0,\rho)$$
for any $\beta\in(0,1)$. To complete the proof of Proposition~\ref{non-degenerateproposition}, let~$\beta\in(0,\Tilde{\beta})$ be arbitrary. We choose the free parameter~$\theta\in(0,\frac{1}{2}]$ from Lemma~\ref{quantitativexcess} as follows
$$\theta\coloneqq\min\big\{\Tilde{C}^{-\frac{1}{2(\Tilde{\beta}-\beta)}} ,2^{-\frac{1}{\beta}}\big\}.$$
Due to~$\beta=\beta(n,p,L,C_1,C_2)$, we therefore have the dependence~$\theta=\theta(n,p,L,C_1,C_2)$. Consider a cylinder~$Q^{(\lambda)}_\rho(z_0)\Subset \Omega_T$ and let~$\Tilde{\nu}=\Tilde{\nu}(n,p,L,C_1,C_2)\in(0,\frac{1}{4}]$ denote the parameter from Lemma~\ref{quantitativexcess} for which condition~\eqref{non-degeneratemeasurecondition} holds true. We note that in view of equivalence of norms, the general bound~\eqref{schrankelambda} yields an estimate for the excess
\begin{align} \label{schranke}
         \Phi_{\lambda}(z_0,\rho) \leq 4n \lambda^2_\mu.
    \end{align}
We will now show inductively that for any~$i\in\N$ there holds
\begin{align} \label{induktioneins}
    \Phi_\lambda(z_0,\theta^i \rho) \leq \theta^{2\beta i}\Phi_\lambda(z_0,\rho),
\end{align}
where we refer to~$\eqref{induktioneins}_i$ 
in the~$i$-th step respectively. We commence by treating the case~$i=1$. The choice of~$\theta$ and an application of Lemma~\ref{quantitativexcess} yields
\begin{equation*}
    \Phi_\lambda(z_0,\theta\rho)\leq \Tilde{C}\theta^{2\Tilde{\beta}}\Phi_\lambda(z_0,\rho)= \theta^{2\beta}\Tilde{C}\theta^{2(\Tilde{\beta}-\beta)}\Phi_\lambda(z_0,\rho)\leq \theta^{2\beta}\Phi_\lambda(z_0,\rho),
\end{equation*}
which establishes that~$\eqref{induktioneins}_1$ holds true. Next, we consider the case~$i>1$ and assume that~$\eqref{induktioneins}_{i-1}$ is satisfied. Additionally, we apply Lemma~\ref{quantitativexcess} and utilize the choice of~$\theta$ once more, to obtain
\begin{align*}
    \Phi_\lambda(z_0,\theta^i \rho) \leq \Tilde{C}\theta^{2\Tilde{\beta}}\Phi_\lambda(z_0,\theta^{i-1}\rho) \leq \Tilde{C}\theta^{2\Tilde{\beta}} \theta^{2\beta(i-1)}\Phi_\lambda(z_0,\rho) \leq \theta^{2\beta i}\Phi_\lambda(z_0,\rho).
\end{align*}
 This establishes~$\eqref{induktioneins}_{i}$. We continue by exploiting~$\eqref{induktioneins}_{i-1}$ together with~$\eqref{schranke}$ and again the choice of~$\theta$, which yields
\begin{align} \label{difference}
    \big| (Du) ^{(\lambda)}_{z_0,\theta^i \rho} - (Du)^{(\lambda)}_{z_0,\theta^{i-1}\rho} \big|^2 &\leq \displaystyle\fiint_{Q^{(\lambda)}_{ \theta^i \rho}(z_0)}\big|Du -(Du)^{(\lambda)}_{z_0,\theta^{i-1}\rho}\big|^2\,\dx\dt \\ \nonumber
    &\leq \theta^{-(n+2)}\Phi_\lambda(z_0,\theta^{i-1}\rho) \\ 
    &\leq 4n\theta^{-(n+2)} \theta^{2\beta(i-1)} \lambda^2_\mu.  \nonumber
\end{align}
Let us now consider natural numbers~$l<m$ and take roots in the preceding estimate~\eqref{difference}. Due to the fact that~$\theta^\beta \leq \frac{1}{2}$, we thus obtain
\begin{align} \label{cauchyfolge}
    \big|(Du) ^{(\lambda)}_{z_0,\theta^l \rho} - (Du)^{(\lambda)}_{z_0,\theta^{m}\rho} \big| &\leq \sum\limits_{j=l+1}^{m} \big|(Du) ^{(\lambda)}_{z_0,\theta^j \rho} - (Du)^{(\lambda)}_{z_0,\theta^{j-1}\rho} \big| \\ \nonumber
    &\leq 2\sqrt{n}\theta^{-\frac{1}{2}(n+2)}\lambda_\mu\sum\limits_{j=l+1}^{m} \theta^{\beta(j-1)} \\ \nonumber
    &\leq 2\sqrt{n}\theta^{-\frac{1}{2}(n+2)}\lambda_\mu \frac{\theta^{\beta l}}{(1-\theta^\beta)} \\ \nonumber
    &\leq 4\sqrt{n} \theta^{-\frac{1}{2}(n+2)}\theta^{\beta l}\lambda_\mu. \nonumber
\end{align}
After passing to the limit~$l\to\infty$ in the preceding estimate, we have established that~$\big((Du)^{(\lambda)}_{z_0,\theta^j \rho}\big)_{j\in\N}$ is a Cauchy sequence in~$\R^n$. Let us denote its limit by
$$\Gamma_{z_0} \coloneqq \lim\limits_{j\to\infty}(Du)^{(\lambda)}_{z_0,\theta^j \rho}.$$
Next, we pass to the limit~$m\to\infty $ in~\eqref{cauchyfolge} to derive
$$\big|(Du) ^{(\lambda)}_{z_0,\theta^l \rho} - \Gamma_{z_0} \big| \leq 4 \sqrt{n} \theta^{-\frac{1}{2}(n+2)}\theta^{\beta l}\lambda_\mu \qquad\mbox{for any~$l\in\N$}.$$
Together with estimate$~\eqref{schranke}$ and~$\eqref{induktioneins}_l$, the following excess-decay estimate for the cylinders~$Q^{(\lambda)}_{ \theta^l \rho}(z_0)$, where~$l\in\N$, can be inferred from the preceding inequality
\begin{align*}
   \displaystyle\fiint_{Q^{(\lambda)}_{ \theta^l \rho}(z_0)}|Du -\Gamma_{z_0}|^2\,\dx\dt &\leq 2 \Phi_\lambda(z_0,\theta^l \rho) + 2\big|\Gamma_{z_0} - (Du) ^{(\lambda)}_{z_0,\theta^l \rho} \big|^2 \\
   &\leq 2 \theta^{2\beta l} \lambda^2_\mu + 32n \theta^{-(n+2)}\theta^{2\beta l}\lambda^2_\mu \\
   &\leq C\theta^{2\beta l}\lambda^2_\mu 
\end{align*}
with a constant~$C=C(n,p,L,C_1,C_2)$. As a last step, this estimate is converted into the excess-decay~\eqref{excess-decay}. Let~$r\in(0,\rho]$ and choose~$l\in\N$, such that ~$\theta^{l+1}\rho<r\leq \theta^{l} \rho$ holds true. The choice of~$\theta$ and the preceding estimate yield
\begin{align*}
     \displaystyle\fiint_{Q^{(\lambda)}_{ r}(z_0)}|Du -\Gamma_{z_0}|^2\,\dx\dt &\leq \frac{1}{\theta^{n+2}} \displaystyle\fiint_{Q^{(\lambda)}_{ \theta^l \rho}(z_0)}|Du -\Gamma_{z_0}|^2\,\dx\dt \\
     &\leq C\theta^{2\beta(l+1)}\lambda^2_\mu \\
     &\leq C\Big(\frac{r}{\rho}\Big)^{2\beta}\lambda^2_\mu
\end{align*}
with~$C=C(n,p,L,C_1,C_2)$. After an application of Jensen's inequality, this implies
\begin{align*}
    \big|(Du)^{(\lambda)}_{z_0,r} - \Gamma_{z_0}\big|^2 &\leq \displaystyle\fiint_{Q^{(\lambda)}_{ r}(z_0)}|Du -\Gamma_{z_0}|^2\,\dx\dt \leq C\Big(\frac{r}{\rho}\Big)^{2\beta}\lambda^2_\mu,
\end{align*}
which further leads to
\begin{align*}
    \Gamma_{z_0} = \lim\limits_{r\downarrow 0} (Du )^{(\lambda)}_{z_0,r}.
\end{align*}
Finally, the bounds~\eqref{bounds} are an immediate consequence of Lemma~\ref{quantitativexcess} and assumption~\eqref{schrankelambda}.
\end{proof}

\subsubsection{The degenerate regime} \label{subsubsec:degenerateregime} \,\\
Our objective in this section is the proof of Proposition~\ref{degenerateproposition}, which presents the main result in the degenerate regime. The fundamental measure-theoretic information in this section is given by
\begin{align*}
        \big|Q^{(\lambda)}_{\rho}(z_0) \setminus E^{(\lambda),i}_{\rho}(z_0,\nu) \big| \geq \nu \big|Q^{(\lambda)}_{\rho}(z_0) \big|
    \end{align*}
for any~$i=1,\ldots,n$. To commence, we will establish that a certain function of~$D_i u$, where~$i=1,...,n$, is a sub-solution to a linear parabolic equation. By incorporating further modifications, we subsequently obtain a weak sub-solution to a linear parabolic equation that is uniformly elliptic. Through a standard argument, two parabolic De Giorgi class type estimates are then obtained, which serve as the principal tool for the subsequent technique of expansion of positivity. \,\\

We start with the following lemma, which asserts that for any~$i=1,...,n$, the function 
$$\Big(|D_i u|^2 - \frac{\lambda^2_\mu}{4} \Big)^2_+$$
is a weak sub-solution to a linear parabolic equation. For this matter, we recall the bilinear form~$\foo{\mathcal{B}}$ introduced in~\eqref{bilineardef}. 

\begin{mylem} \label{subsolution}
    Let~$u$ be a weak solution to~\eqref{pdeohnex} and consider the function 
    \begin{align*}
    v \coloneqq \Big(|D_i u|^2 - \frac{\lambda^2_\mu}{4}\Big)^2_+,
    \end{align*}
    where~$i\in\{1,...,n\}$ is arbitrary. Then,~$v$ is a weak sub-solution to a linear parabolic equation in the sense that~$v$ satisfies the integral inequality
    \begin{align*}
\displaystyle\iint_{Q^{(\lambda)}_{2\rho}(z_0)}\big[-v \partial_t \phi + \foo{\mathcal{B}}(Du)(Dv,D\phi )\big]\,\dx\dt \leq 0 
\end{align*}
for any non-negative test function
$$\phi\in W^{1,2}_0\big(\Lambda^{(\lambda)}_{2\rho}(t_0);L^2\big(B_{2\rho}(x_0)\big) \big) \cap L^2\big(\Lambda^{(\lambda)}_{2\rho}(t_0);W^{1,2}_0\big(B_{2\rho}(x_0)\big) \big).$$
\end{mylem}
\begin{proof}
    We differentiate the weak form of~\eqref{pdeqwieder1} by testing the latter with the function~$\varphi\coloneqq D_i\big[\phi D_i u\,\Phi(|D_iu|^2)\big]$. Here,~$\Phi\in W^{1,\infty}_{\loc}(\R_{\geq 0},\R_{\geq 0})$ denotes a non-negative and non-decreasing locally Lipschitz function, while~$\phi\in C^1_0(Q^{(\lambda)}_{2\rho}(z_0))$ is an arbitrary non-negative smooth cut-off function. By omitting a Steklov-average procedure, we obtain
    \begin{align*}
        \displaystyle\iint_{Q^{(\lambda)}_{ 2\rho}(z_0)} &\Big[u \partial_t \big[D_i\big[\phi D_i u\,\Phi(|D_i u|^2)\big]\big] \\
        &- \big\langle A(t,Du), D_i\big[D\big[\phi D_i u\,\Phi(|D_iu|^2)\big]\big] \big\rangle\Big]\dx\dt = 0.
    \end{align*}
    Integrating by parts in both, the term involving the time derivative and also the diffusion term, leads to
    \begin{align*}
        \displaystyle\iint_{Q^{(\lambda)}_{ 2\rho}(z_0)} &\Big[\partial_t [D_i u]\, \phi D_i u\,\Phi(|D_i u|^2) \\
        &+ \big\langle D_i A(x,t,Du), D\big[\phi D_i u\,\Phi(|D_iu|^2)\big] \big\rangle\Big]\dx\dt = 0.
    \end{align*}
    First, we treat the term involving the time derivative by following the approach taken in~\cite[Proposition~3.1]{degeneratesystems}. Thus, we achieve
    \begin{align*}
        \displaystyle\iint_{Q^{(\lambda)}_{ 2\rho}(z_0)} \partial_t [D_i u]\, \phi D_i u\,\Phi(|D_i u|^2) \,\dx\dt &= \frac{1}{2} \displaystyle\iint_{Q^{(\lambda)}_{ 2\rho}(z_0)} \partial_t |D_i u|^2  \phi\,\Phi(|D_i u|^2) \,\dx\dt \\
        &= \frac{1}{2} \displaystyle\iint_{Q^{(\lambda)}_{ 2\rho}(z_0)} \partial_t v\,  \phi  \,\dx\dt \\
        & = -\frac{1}{2} \displaystyle\iint_{Q^{(\lambda)}_{ 2\rho}(z_0)} v  \partial_t \phi  \,\dx\dt.
    \end{align*}
    Next, we turn our attention to the diffusion term, where we exploit the fact that
    $$D_iu D_i Du = \frac{1}{2}D|D_i u|^2\qquad\mbox{a.e. in~$Q^{(\lambda)}_{2\rho}(z_0)$}.$$
    Moreover, we utilize the bilinearity of the bilinearform~$\mathcal{B}$, to obtain
    \begin{align*}
    \displaystyle\iint_{Q^{(\lambda)}_{ 2\rho}(z_0)}& \big\langle D_i A(t,Du) , D\big[\phi D_i u\,\Phi(|D_i u|^2)\big]\big\rangle\,\dx\dt \\
    &= \displaystyle\iint_{Q^{(\lambda)}_{2\rho}(z_0)} \mathcal{B}(Du)(D_iDu,D_i Du)\rangle\Phi(|D_i u|^2)\phi\,\dx\dt\\
    & \quad+  \displaystyle\iint_{Q^{(\lambda)}_{2\rho}(z_0)} \frac{1}{2}\mathcal{B}(Du)(D|D_i u|^2,D\phi)\Phi(|D_i u|^2)\,\dx\dt \\
    & \quad+ \displaystyle\iint_{Q^{(\lambda)}_{2\rho}(z_0)}\frac{1}{2}\mathcal{B}(Du)(D|D_i u|^2,D|D_i u|^2)\Phi'(|D_i u|^2)\phi\,\dx\dt.
    \end{align*}
 The first quantity on the right-hand side of the preceding estimate may be bounded below by exploiting structure condition~$(\ref{voraussetzungen})_2$.
    Therefore, by combining our calculations for the term involving the time derivative and the diffusion term, we arrive at
\begin{align*}
    \displaystyle\iint&_{Q^{(\lambda)}_{ 2\rho}(z_0)} \Big[-\frac{1}{2} v  \partial_t \phi + \frac{1}{2}\mathcal{B}(Du)(D|D_i u|^2,D\phi)\Phi(|D_i u|^2)\Big]\,\dx\dt \\
    &+ \displaystyle\iint_{Q^{(\lambda)}_{ 2\rho}(z_0)} C_2 h_\mu(|Du|)|D_i Du|^2\Phi(|D_i u|^2)\phi\,\dx\dt \\
    &+\displaystyle\iint_{Q^{(\lambda)}_{ 2\rho}(z_0)} \frac{1}{2}\mathcal{B}(Du)(D|D_i u|^2,D|D_i u|^2)\Phi'(|D_i u|^2)\phi\,\dx\dt\\
    &\leq 0.
\end{align*}
  Finally, we choose the locally Lipschitz function
  $$\Phi(t)\coloneqq 2 \Big(t-\frac{\lambda^2_\mu}{4} \Big)_+$$
  and note that~$D v = D|D_i u|^2 \Phi(|D_i u|^2)$ a.e. in~$Q^{(\lambda)}_{ 2\rho}(z_0)$. Due to a non-negative contribution of the second and third integral term on the left-hand side of the preceding inequality, which can be inferred from Lemma~\ref{bilinearformeins}, and the fact that~$\Phi$ is non-decreasing, we discard both quantities to derive that~$v$ is a weak sub-solution to the linear parabolic equation given above. 
\end{proof}

\begin{remark} \upshape
The employment of Lemma~\ref{subsolution} allows us to directly address the quantity~$|D_i u|$, for~$i=1,...,n$, in light of the measure-theoretic information~\eqref{degeneratemeasurecondition}. In contrast to previous results in the literature, such as those presented in~\cite{kuusi2013gradient,kuusi2014wolff} for the parabolic case and~\cite{manfredi1986regularity} for the elliptic case, which required a separate analysis based on the sign of~$D_i u$, our approach offers a simplified strategy in a unified manner.
\end{remark}

Next, we re-scale~$u$ and~$\phi$ to~$\Tilde{u}$ and~$\Tilde{\phi}$, defined on~$Q_{2\rho}\coloneqq B_{2\rho}\times(-(2\rho)^2,0]$, by setting
\begin{align*}
    \Tilde{u}(x,\tau)\coloneqq\frac{u(x_0+x,t_0+\lambda^{2-p}\tau)}{\lambda_\mu}, \qquad \Tilde{\phi}(x,\tau) \coloneqq \phi(x_0+x,t_0+\lambda^{2-p}\tau).
\end{align*}
As Lemma~\ref{subsolution} holds true for any arbitrary~$i\in\{1,...,n\}$, throughout this section we fix one such~$i\in\{1,...,n\}$. Moreover, we define
\begin{align} \label{scaledsubsol}
    w(x,\tau) \coloneqq \frac{7}{16} + \frac{v(x_0+x,t_0+\lambda^{2-p}\tau )}{\lambda^4_\mu} = \frac{7}{16} + \Big(|D_i \Tilde{u}(x,\tau)|^2-\frac{1}{4} \Big)^2_+
\end{align}
for~$(x,t)\in Q_{2\rho}$. In this setting, assumption~\eqref{schrankelambda} implies the bounds
\begin{align} \label{rescaledbound}
    \esssup\limits_{Q_{2\rho}} |D_i\Tilde{u}| \leq 1 \quad \& \quad \esssup\limits_{Q_{2\rho}}w \leq 1
\end{align}
and due to Lemma~\ref{subsolution} we infer, by a change of variables together by exploiting the bilinearity of~$\foo{\mathcal{B}}$, that~$w$ is a sub-solution to a re-scaled linear parabolic equation and satisfies the integral inequality
\begin{align}\label{subsol}
\displaystyle\iint_{Q_{2\rho}}\big[-w \partial_t \Tilde{\phi} + \Tilde{\foo{\mathcal{B}}}(Dw,D\Tilde{\phi} )\big]\,\dx\dt \leq 0 
\end{align}
for any non-negative 
$$\Tilde{\phi}\in W^{1,2}_0\big((-(2\rho)^2,0]);L^2(B_{2\rho}) \big) \cap L^2\big((-(2\rho)^2,0]);W^{1,2}_0(B_{2\rho}) \big).$$
Here, we abbreviated
\begin{align*}
    \Tilde{\foo{\mathcal{B}}}(x,\tau)(\xi,\eta) \coloneqq \lambda^{2-p}\foo{\mathcal{B}}(\lambda_\mu D \Tilde{u}(x,\tau))(\xi,\eta)
\end{align*}
for~$(x,\tau)\in Q_{2\rho}$, and~$\xi,\eta \in\R^n$. We note that~$w$ remains constant on the set~$\{|D_i \Tilde{u}|\leq\frac{1}{2}\}$ and thus does not impact the diffusion part. Due to the set inclusion
$$\Big\{(x,\tau)\in Q_{2\rho}\colon~|D\Tilde{u}(x,\tau)|\leq\frac{1}{2}\Big\} \subset \Big\{(x,\tau)\in Q_{2\rho}\colon~|D_i \Tilde{u}(x,\tau)|\leq \frac{1}{2}\Big\},$$
we may redefine~$\Tilde{\foo{\mathcal{B}}}$ by incorporating
\begin{align*}
    \hat{\foo{\mathcal{B}}}(x,\tau) \coloneqq \begin{cases*}
        \hat{\foo{I}}_n & \mbox{on $\big\{(x,\tau)\in Q_{2\rho}\colon~|D_i \Tilde{u}(x,\tau)|\leq\frac{1}{2}\big\}$} \\
        \Tilde{\foo{\mathcal{B}}}(x,\tau) & \mbox{on $\big\{(x,\tau)\in Q_{2\rho}\colon~|D_i \Tilde{u}(x,\tau)|>\frac{1}{2}\big\}$},
    \end{cases*} 
\end{align*}
where~$\hat{\foo{I}}_n(\xi,\eta)\coloneqq\langle\foo{I}_n \xi,\eta\rangle = \langle \xi,\eta\rangle$ for any~$\xi,\eta\in\R^n$. This way, the preceding weak form translates to 
\begin{align*}
\displaystyle\iint_{Q_{2\rho}}\big[-w \partial_t \Tilde{\phi} + \hat{\foo{\mathcal{B}}}\big(Dw,D\Tilde{\phi} \big)\big]\,\dx\dt \leq 0 
\end{align*}
for any non-negative test function
$$\Tilde{\phi}\in W^{1,2}_0\big((-(2\rho)^2,0]);L^2(B_{\rho}) \big) \cap L^2\big((-(2\rho)^2,0]);W^{1,2}_0(B_{2\rho}) \big).$$
Furthermore, the coefficients~$\hat{\foo{\mathcal{B}}}$ are uniformly elliptic and bounded in the sense that there exists~$C=C(n,p,L,C_1,C_2)\geq 1$, such that
\begin{align} \label{elliptic}
    \frac{1}{C}|\zeta|^2 \leq \hat{\foo{\mathcal{B}}}(x,\tau)(\zeta,\zeta) \leq C |\zeta|^2
\end{align}
for any~$(x,\tau)\in Q_{2\rho}$ and any~$\zeta\in\R^n$. The preceding bound~\eqref{elliptic} clearly holds true with constant~$C=1$ in the subset~$\{|D_i \Tilde{u}|\leq \frac{1}{2}\}$, which follows from utilizing the definition of~$\hat{\foo{\mathcal{B}}}$. In the complement~$\{|D_i \Tilde{u}|> \frac{1}{2}\}$, we apply Lemma~\ref{bilinearformeins} to obtain the upper bound 
\begin{align*}
    \hat{\foo{\mathcal{B}}}(x,\tau)(\zeta,\zeta) &\leq C \frac{h_\mu(\lambda_\mu |D\Tilde{u}(x,\tau)|)}{\lambda^{p-2}}|\zeta|^2 \leq C|\zeta|^2
\end{align*}
for any~$p>1$, with~$C=C(n,p,L,C_1)$. Note that in the case~$p<2$ we exploited the fact that~$|D \Tilde{u}|\geq |D_i \Tilde{u}|> \frac{1}{2}$, whereas in the case~$p\geq 2$ we utilized the equivalence of norms~$|\cdot|$ and~$|\cdot|_\infty$ on~$\R^n$ as well as assumption~\eqref{schrankelambda}. The lower bound in~\eqref{elliptic} is established in a similar fashion, with the constant depending on~$p,L,C_2$. In turn, this yields~\eqref{elliptic} with the claimed dependence of~$C=C(p,L,C_1,C_2)$. Concluding, both inequalities~\eqref{subsol} and~\eqref{elliptic} in combination establish that~$w$ is a weak sub-solution to a linear parabolic equation with uniformly elliptic and bounded coefficients. This fact enables us to derive two energy estimates, which serve as the starting point for a measure-theoretic approach inspired by De Giorgi. To maintain simplicity, we return to employ~$t$ as the time variable of~$w$ throughout the subsequent discussion.

\begin{mylem} [A De Giorgi class type estimate for a sub-solution] \label{degiorgilem}
    Let~$w$ denote the weak sub-solution defined in~\eqref{scaledsubsol}. There exists~$C=C(n,p,L,C_1,C_2)$, such that for any~$k>0$,~$\sigma\in(0,1)$, and any cylinder~$Q_{R,S}(0,\tau_0)=B_R\times(\tau_0-S,\tau_0]\subset Q_\rho$ there hold
    \begin{align} \label{EEeins}
        \esssup\limits_{\tau\in(\tau_0-\sigma S,\tau_0]}\displaystyle & \int_{B_{\sigma R}\times\{\tau\}} (w-k)^2_+\,\dx + \displaystyle\iint_{Q_{\sigma R,\sigma S}(0,\tau_0)} |D(w-k)_+ |^2\,\dx\dt \\ \nonumber
        &\leq C \displaystyle\iint_{Q_{R,S}(0,\tau_0)}\bigg[\frac{(w-k)^2_+}{(1-\sigma)^2 R^2} + \frac{(w-k)^2_+}{(1-\sigma)S} \bigg]\,\dx\dt
    \end{align}
    and
     \begin{align} \label{EEzwei}
        \esssup\limits_{\tau\in(\tau_0- S,\tau_0]}\displaystyle & \int_{B_{\sigma R}\times\{\tau\}} (w-k)^2_+\,\dx \\ \nonumber
        &\leq \displaystyle\int_{B_R \times\{\tau_0-S\}}(w-k)^2_+\,\dx\dt +   C\displaystyle\iint_{Q_{R,S}(0,\tau_0)}\frac{(w-k)^2_+}{(1-\sigma)^2 R^2} \,\dx\dt.
    \end{align}
\end{mylem}
\begin{proof}
To simplify notation, we write~$Q_{R,S}$ instead of~$Q_{R,S}(0,\tau_0)$. We test the weak form~\eqref{subsol} with~$\phi = \varphi^2\chi(w-k)_+$, where ~$\varphi\in C^1_0(Q_R,[0,1])$ denotes a spatial cut-off function satisfying~$\varphi\equiv 1$ on~$B_{\sigma R}$ and~$|D\varphi|\leq \frac{2}{(1-\sigma)R}$. Furthermore,~$\chi\in W^{1,\infty}((\tau_0-S,\tau_0],[0,1])$ denotes a Lipschitz cut-off function in time vanishing at the end points of the interval~$(\tau_0-S,\tau_0]$. By initially considering the quantity involving the time derivative and omitting a Steklov-average procedure, we obtain
\begin{align*}
    -\displaystyle\iint_{Q_{R,S}}w \partial_t \phi \,\dx\dt &= \displaystyle\iint_{Q_{R,S}} \partial_t w \phi \,\dx\dt  \\
    &= \frac{1}{2}\displaystyle\iint_{Q_{R,S}} \partial_t (w-k)^2_+ \varphi^2\chi \,\dx\dt \\
    &= -\frac{1}{2}\displaystyle\iint_{Q_{R,S}} \partial_t \chi (w-k)^2_+ \varphi^2 \,\dx\dt.
\end{align*}
Next, the diffusion part is estimated by
\begin{align*}
   & \iint_{Q_{R,S}} \hat{\foo{\mathcal{B}}}(Dw, D\phi)\,\dx\dt \\
    &= \displaystyle\iint_{Q_{R,S}} \hat{\foo{\mathcal{B}}}(Dw,\chi(w-k)_+2\varphi D\varphi + \chi\varphi^2 D(w-k)_+)\,\dx\dt \\
    &= \displaystyle\iint_{Q_{R,S}} \big[2\chi\varphi (w-k)_+\hat{\foo{\mathcal{B}}}(D(w-k)_+,D\varphi)+ \chi\varphi^2\hat{\foo{\mathcal{B}}}(D(w-k)_+, D(w-k)_+ )\big]\,\dx\dt \\
    &\geq \displaystyle\iint_{Q_{R,S}} \bigg[\frac{1}{C}\varphi^2\chi|D(w-k)_+|^2 - C\varphi\chi(w-k)_+|D(w-k)_+||D\varphi| \bigg]\,\dx\dt \\
    &\geq \frac{1}{C} \displaystyle\iint_{Q_{R,S}} \varphi^2\chi|D(w-k)_+|^2\,\dx\dt - C\displaystyle\iint_{Q_{R,S}} \frac{(w-k)^2_+}{(1-\sigma)^2R^2} \,\dx\dt
\end{align*}
with~$C=C(p,L,C_1,C_2)$. In the penultimate step we utilized~\eqref{elliptic} and then applied Young's inequality. Combining both estimates for the term involving the time derivative and the diffusion term respectively, we end up with
\begin{align} \label{timediff}
    -\displaystyle\iint_{Q_{R,S}} \partial_t \chi (w-k)^2_+ \varphi^2 \,\dx\dt &+ \frac{1}{C} \displaystyle\iint_{Q_{R,S}} \varphi^2\chi|D(w-k)_+|^2\,\dx\dt \\ \nonumber
    &\leq C\displaystyle\iint_{Q_{R,S}} \frac{(w-k)^2_+}{(1-\sigma)^2R^2} \,\dx\dt
\end{align}
with~$C=C(p,L,C_1,C_2)$. We will first establish the energy estimate~\eqref{EEeins}. Let~$\tau\in(\tau_0-\sigma S,\tau_0]$ and ~$\epsilon>0$ small enough, such that~$0<\epsilon < \tau_0-\tau$. Subsequently, we choose
\begin{align*}
    \chi(t) = \begin{cases*}
        \frac{t-\tau_0+S}{(1-\sigma)S} & \mbox{for $\tau_0-S<t<\tau_0-\sigma S$}\\
        1 & \mbox{for $\tau_0-\sigma S\leq t<\tau$}\\
        -\frac{t-\tau-\epsilon}{\epsilon} & \mbox{for $\tau\leq t<\tau+\epsilon$}\\
        0 & \mbox{for $\tau+\epsilon\leq t\leq\tau_0$}
    \end{cases*}
\end{align*}
in a way, that~$\chi$ interpolates linearly on~$(\tau_0-S,\tau_0-\sigma S)$ and~$(\tau,\tau+\epsilon)$. This choice of~$\chi$ leads to
\begin{align*}
    -\frac{1}{2}\displaystyle & \iint_{Q_{R,S}} \partial_t \chi (w-k)^2_+ \varphi^2 \,\dx\dt \\
    &= -\frac{1}{2} \displaystyle \iint_{B_R\times(\tau_0-S,\tau_0-\sigma S)} \frac{(w-k)^2_+ \varphi^2}{(1-\sigma)S} \,\dx\dt + \frac{1}{2\epsilon}\displaystyle\iint_{B_R\times(\tau,\tau+\epsilon)}(w-k)^2_+ \varphi^2 \,\dx\dt \\
    &\geq -\frac{1}{2} \displaystyle \iint_{Q_{R,S}} \frac{(w-k)^2_+}{(1-\sigma)S} \,\dx\dt + \frac{1}{2\epsilon}\displaystyle\iint_{B_{\sigma R}\times(\tau,\tau+\epsilon)}(w-k)^2_+ \,\dx\dt,
\end{align*}
where we exploited the properties of~$\varphi$ in the last step. By plugging this estimate into~\eqref{timediff}, passing to the limit~$\epsilon\downarrow 0$, and utilizing the properties of~$\varphi$ and ~$\chi$, we obtain
\begin{align*}
     \displaystyle & \iint_{B_{\sigma R}\times\{\tau\}} (w-k)^2_+ \,\dx +  \displaystyle\iint_{B_{\sigma R}\times(\tau_0-\sigma S,\tau)}|D(w-k)_+|^2\,\dx\dt \\ \nonumber
    &\leq C\displaystyle\iint_{Q_{R,S}} \frac{(w-k)^2_+}{(1-\sigma)^2R^2} \,\dx\dt + C\displaystyle\iint_{Q_{R,S}} \frac{(w-k)^2_+}{(1-\sigma)S} \,\dx\dt
\end{align*}
with~$C=C(p,L,C_1,C_2)$. Since~$\tau\in(\tau_0-\sigma S,\tau_0]$ was arbitrary, we take the essential supremum with respect to~$\tau\in(\tau_0-\sigma S,\tau_0]$ in the first integral on the left-hand side of the preceding inequality. In the second quantity on the left-hand side, we pass to the limit~$\tau \uparrow \tau_0$. This way, the first energy estimate~\eqref{EEeins} is established. Furthermore, we prove the validity of the second energy estimate~\eqref{EEzwei} by introducing a modified version of the previous cut-off function~$\chi$, which is
\begin{align*}
    \chi(t) = \begin{cases*}
        \frac{t-\tau_0+S}{\epsilon} & mbox{for $\tau_0-S<t<\tau_0- S + \epsilon$}\\
        1 & \mbox{for $\tau_0- S + \epsilon \leq t<\tau$}\\
        -\frac{t-\tau-\epsilon}{\epsilon} & \mbox{for $\tau\leq t<\tau+\epsilon$}\\
        0 & \mbox{for $\tau+\epsilon\leq t\leq\tau_0$}
    \end{cases*}
\end{align*}
for~$\tau\in(\tau_0-S,\tau_0]$ and~$\epsilon>0$ small enough, such that~$0<\epsilon<\min\{t_0-\tau,\tau-\tau_0+S\}$ holds true. Similarly to before, by exploiting the properties of the cut-off functions~$\varphi$ and~$\chi$, we achieve
\begin{align*}
    -\frac{1}{2}\displaystyle & \iint_{Q_{R,S}} \partial_t \chi (w-k)^2_+ \varphi^2 \,\dx\dt \\
    &= -\frac{1}{2\epsilon} \displaystyle \iint_{B_R\times(\tau_0-S,\tau_0- S+\epsilon)} (w-k)^2_+ \varphi^2 \,\dx\dt + \frac{1}{2\epsilon}\displaystyle\iint_{B_R\times(\tau,\tau+\epsilon)}(w-k)^2_+ \varphi^2 \,\dx\dt \\
    &\geq -\frac{1}{2\epsilon} \displaystyle \iint_{B_R\times(\tau_0-S,\tau_0- S+\epsilon)} (w-k)^2_+ \,\dx\dt + \frac{1}{2\epsilon}\displaystyle\iint_{B_{\sigma R}\times(\tau,\tau+\epsilon)}(w-k)^2_+ \,\dx\dt.
\end{align*}
Passing to the limit~$\epsilon\downarrow 0$ in the preceding estimate yields
\begin{align*}
    -\frac{1}{2}\displaystyle & \iint_{Q_{R,S}} \partial_t \chi (w-k)^2_+ \varphi^2 \,\dx\dt \\
    &\geq - \frac{1}{2}\displaystyle \iint_{B_R\times\{\tau_0-S\}} (w-k)^2_+\,\dx\dt + \frac{1}{2}\displaystyle\iint_{B_{\sigma R}\times\{\tau\}}(w-k)^2_+ \,\dx\dt.
\end{align*}
Returning to~\eqref{timediff} and utilizing the properties of~$\varphi$ and~$\chi$ again, this implies
\begin{align*}
     \displaystyle & \iint_{B_{\sigma R}\times\{\tau\}} (w-k)^2_+ \,\dx +  \displaystyle\iint_{B_{\sigma R}\times(\tau_0- S,\tau)} |D(w-k)_+|^2\,\dx\dt \\ \nonumber
    &\leq C\displaystyle\iint_{B_R\times\{\tau_0-S\}} (w-k)^2_+ \,\dx\dt + C\displaystyle\iint_{Q_{R,S}} \frac{(w-k)^2_+}{(1-\sigma)^2 R^2} \,\dx\dt
\end{align*}
with~$C=C(p,L,C_1,C_2)$. Since~$\tau\in(\tau_0-S,\tau_0]$ was arbitrary, we take the essential supremum with respect to~$\tau\in(\tau_0-S,\tau_0]$ in the first quantity on the left-hand side above, and disregard the second term on the left-hand side due to a non-negative contribution. In turn, we obtain the energy estimate~\eqref{EEzwei}, which finishes the proof.
\end{proof}

The next ingredient required to establish Proposition~\ref{degenerateproposition} is a tool that converts measure-theoretic information of~$w$ at a fixed time into pointwise estimates for~$w$ at later times. This technique is commonly known as expansion of positivity.


\begin{myproposition} [Expansion of positivity] \label{expansion} 
    Let~$w$ denote the weak sub-solution defined in~\eqref{scaledsubsol}. Let~$t_0\in(-\rho^2,-\frac{1}{2}\rho^2)$, and~$\Gamma>0$, such that~$t_0+\Gamma\rho^2< 0$ holds true. For any~$\alpha\in(0,1)$ and $M\in(0,1)$, there exists~$\eta=\eta(n,p,L,C_1,C_2,\alpha,\Gamma)\in(0,1)$, such that whenever the assumption
    \begin{align}\label{meascond}
        |B_\rho \cap \{1-w(\cdot,t_0)>M \} |\geq \alpha|B_\rho|
    \end{align}
    is satisfied, then for all times 
    $$t_0+\Gamma\rho^2<t\leq 0$$ there holds
    $$1-w(\cdot,t)\geq \eta M\qquad\mbox{a.e. in $B_\rho$}.$$
\end{myproposition}
\begin{proof}
    Due to similarity we refer the reader to the proof of Proposition~$7.2$ in~\cite{degeneratesystems} for a detailed elaboration of the proof. Roughly speaking, by utilizing the De Giorgi class type estimate from Lemma~\ref{degiorgilem}, we are able to derive the result of Proposition~\ref{expansion} by following a similar approach to~\cite{degeneratesystems}. However, it is worth noting that the alternative conclusion~$\eta M \leq \rho$ in~\cite[Proposition~7.2]{degeneratesystems} does not apply in our setting, for the linear parabolic equation involving the sub-solution~$w$ simplifies in comparison to the corresponding version in~\cite[(7.3)]{degeneratesystems}.
\end{proof}

At this point, we are in position to prove Proposition~\ref{degenerateproposition} and conclude the degenerate regime. 
\begin{proof} [\textbf{\upshape Proof of Proposition~\ref{degenerateproposition}}]
    As before, let~$w$ denote be the weak sub-solution defined in~\eqref{scaledsubsol}, and recall property~\eqref{rescaledbound}. We begin by exploiting the measure-theoretic assumption~\eqref{degeneratemeasurecondition}, which characterizes the degenerate regime and is also alternatively indicated by
 \begin{align*} 
        \big|Q^{(\lambda)}_{\rho}(z_0) \setminus E^{(\lambda),i}_{\rho}(z_0,\nu) \big| &= \big|\big\{z\in Q^{(\lambda)}_{\rho}(z_0):\,\,|D_i u(z)| \leq (1-\nu)\lambda_\mu  \big\}\big| \\
        &\geq \nu \big|Q^{(\lambda)}_{\rho}(z_0) \big|
    \end{align*}
    for any~$i=1,\ldots,n$. Due to the definition of~$\Tilde{u}$ preliminary to~\eqref{scaledsubsol}, we conclude from the preceding measure-theoretic information that for any~$i\in\{1,...,n\}$ there holds
\begin{align*}
   |\{z\in Q_\rho\colon~|D_i\Tilde{u}|\leq 1-\nu \}| \geq \nu |Q_{\rho} |.
\end{align*}
 Since~$\nu\in(0,\frac{1}{2}]$, we have the estimate
\begin{equation*}
   {\textstyle \frac{7}{16}+ \big((1-\nu)^2-\frac{1}{4}\big)^2_+ \leq \frac{7}{16} + \big(\frac{3}{4}-\nu^2\big)^2\leq 1-\nu^4, }
\end{equation*}
which furthermore yields
\begin{align*}
   | Q_\rho\cap\{w\leq 1-\nu^4  \}| \geq \nu |Q_{\rho}|.
\end{align*}
 This measure-theoretic information implies the existence of a time~$t_0\in\big(-\rho^2,-\frac{1}{2}\nu\rho^2 \big]$, such that
\begin{align} \label{lemstart}
   | Q_\rho\cap\{w(\cdot,t_0)\leq 1-\nu^4 \}| \geq {\textstyle\frac{1}{2}}\nu |B_{\rho} |.
\end{align}
Otherwise the following reasoning gives a contradiction
\begin{align*}
    \nu |Q_\rho| &\leq \displaystyle\int_{-\rho^2}^{0} | B_\rho\cap\{w(\cdot,t)\leq 1-\nu^4 \}|\,\dt \\
    &= \displaystyle\int_{-\rho^2}^{-\frac{1}{2}\nu\rho^2} | B_\rho\cap\{w(\cdot,t)\leq 1-\nu^4 \}|\,\dt \\
    & \quad+ \displaystyle\int_{-\frac{1}{2}\nu\rho^2}^{0} | B_\rho\cap\{w(\cdot,t)\leq 1-\nu^4 \}|\,\dt \\
    &\leq {\textstyle\frac{1}{2}}\nu|Q_\rho|\big(1-{\textstyle\frac{1}{2}}\nu \big) + {\textstyle\frac{1}{2}}\nu|Q_\rho| \\
    &= \nu{\textstyle \big(1-\frac{1}{4}\nu\big)} |Q_\rho|.
\end{align*}
Due to~\eqref{lemstart}, we are in position to apply Proposition~\ref{expansion} with the choice~$\alpha \coloneqq \frac{1}{2}\nu$,~$\Gamma\coloneqq \frac{1}{4}\nu$, and~$M\coloneqq \nu^4$. In turn, this yields the existence of a parameter~$\eta=\eta(n,p,L,C_1,C_2,\nu)\in(0,1)$, such that
$$w \leq 1-\eta\nu^4 \qquad\mbox{a.e. in $B_\rho\times \big(-\frac{1}{4}\nu\rho^2,0 \big]$}$$
holds true. By utilizing the definition of~$w$, the preceding estimate translates to
$${\textstyle \frac{7}{16}+\big(|D_i \Tilde{u}|^2-\frac{1}{4}\big)^2_+\leq 1-\eta\nu^4\qquad\mbox{a.e. in $B_\rho\times \big(-\frac{1}{4}\nu\rho^2,0 \big]$} }$$
or expressed equivalently
$${\textstyle \big(|D_i \Tilde{u}|^2-\frac{1}{4}\big)_+\leq \big(\frac{9}{16}-\eta\nu^4\big)^{\frac{1}{2}}\qquad\mbox{a.e. in $B_\rho\times \big(-\frac{1}{4}\nu\rho^2,0 \big]$}. }$$
Next, we exploit the elementary estimate~$\sqrt{a^2-\epsilon}\leq a - \frac{1}{2a}\epsilon$ that holds true for positive quantities~$0\leq \epsilon<a^2$, to obtain
$${\textstyle \big(|D_i \Tilde{u}|^2-\frac{1}{4}\big)_+\leq \frac{3}{4}-\frac{2}{3}\eta\nu^4\qquad\mbox{a.e. in $B_\rho\times \big(-\frac{1}{4}\nu\rho^2,0 \big]$}. }$$
In turn, by utilizing the same estimate again, there holds
$${\textstyle |D_i \Tilde{u}|\leq \big(1-\frac{2}{3}\eta\nu^4\big)^{\frac{1}{2}} \leq 1-\frac{1}{3}\eta\nu^4\qquad\mbox{a.e. in $B_\rho\times \big(-\frac{1}{4}\nu\rho^2,0 \big]$}. }$$
We define~$\Tilde{\eta}= \Tilde{\eta}(n,p,L,C_1,C_2,\nu)\coloneqq \frac{1}{3}\eta\nu^4\in\big(0,\frac{1}{2}\big]$ and transform back to the original solution~$u$. This way, we obtain the bound
\begin{align*}
   |D_i u| \leq (1-\Tilde{\eta})\lambda_\mu \qquad\mbox{a.e. in $B_\rho(x_0)\times \Lambda^{(\lambda)}_{\Tilde{\nu}\rho}(t_0)$},
\end{align*}
where~$\Tilde{\nu}=\frac{\sqrt{\nu}}{2}$. Moreover, as the same procedure of this section may be performed for any arbitrary~$i\in\{1,...,n\}$, there holds
\begin{align*}
   |D u|_\infty \leq (1-\Tilde{\eta})\lambda_\mu \qquad\mbox{a.e. in $B_\rho(x_0)\times \Lambda^{(\lambda)}_{\Tilde{\nu}\rho}(t_0)$}.
\end{align*}
This finishes the proof of the proposition with the parameter~$\kappa\coloneqq (1-\Tilde{\eta})\in\big[\frac{1}{2},1 \big)$. 
\end{proof}

\subsection{\textit{A priori} comparison, oscillation, and energy estimates} \label{subsec:aprioriestimates} 
The objective in this section is to derive suitable~\textit{a priori} estimates. Specifically, we will obtain a comparison estimate, an oscillation estimate, and subsequently also an energy estimate.

\subsubsection{Comparison estimate and comparison principle} \label{subsubsec:comparisonestimate} \,\\
In order to utilize the~\textit{a priori} gradient estimate in form of Theorem~\ref{campanatoaprioritheorem} of the previous section, it is crucial to establish a comparison estimate that allows a transfer of the regularity of weak solutions to a Cauchy-Dirichlet problem involving equation~\eqref{pdeqwieder1}. For this purpose, we examine the following: let~$p>1$ be arbitrary and denote
\begin{equation} \label{vergleichsfunktionzwei}
    v\in C\big([0,T];L^2(\Omega)\big)\cap L^p\big(0,T;W^{1,p}(\Omega)\big)
\end{equation}
as the unique weak solution to the Cauchy-Dirichlet problem
\begin{align}\label{cauchy-dirichlet-comparison}
        \begin{cases}
        \partial_t v - \divv V(x,t,Dv) = 0 &\,\,\mbox{in $\Omega_T$,} \\[3pt]
        v=u & \,\,\mbox{on $\partial_p \Omega_T$.} \\[3pt]
        \end{cases}
\end{align}
Here~$V\colon\Omega_T \times\R^n$ is assumed to be a Carath\'eodory function satisfying the following set of structure conditions
\begin{align}
    \left\{
    \begin{array}{l}
    | V(x,t,\xi)| \leq C_7 (\mu^2 + |\xi|^{2} )^{\frac{p-1}{2}} \\[3pt]
    \langle V(x,t,\xi) - V(x,t,\eta),\xi-\eta \rangle \geq C_8 ( \mu^2 + |\xi|^2 + |\eta|^2)^{\frac{p-2}{2}}|\xi-\eta|^2 
    \end{array}
    \right. \label{voraussetzungen-comparison}
\end{align}
for a.e.~$(x,t)\in\Omega_T$, any~$\xi,\eta\in\R^n$, and~$\mu\in[0,1]$, where~$0<C_7\leq C_8$ denote positive constants. Note that, unlike the conditions specified in~\eqref{voraussetzungen}, no additional assumptions regarding regularity, such as differentiability, are imposed on the vector field~$V$.

\begin{remark} \upshape
    For the existence of a weak solution~$v$ to the preceding Cauchy-Dirichlet problem~\eqref{cauchy-dirichlet-comparison}, we refer to~\cite[Chapter~III, Proposition 4.1 and Example~4.A]{existence}.
\end{remark}

We obtain the following comparison estimate involving the solution~$u$ and the comparison function~$v$ given above, which can be inferred from~\cite[Lemma~10.3]{gradientholder}.
\begin{mylem} \label{comparison-lemma}
    Let~$p>1$. Further, let~$u$ be a weak solution of~\eqref{pdeqgleich1} and let~$v$ denote the unique weak solution of~\eqref{cauchy-dirichlet-comparison}. There exists~$C=C(p,C_8)>0$ with
    \begin{align} \label{comparison-estimate-eins}
        \displaystyle\iint_{\Omega_T}&|Du-Dv|^p\,\dx\dt \\ \nonumber
        &\leq C \left\|A(\cdot,Du)-V(\cdot,Du) \right\|^{p'}_{L^{p'}(\Omega_T)} \\ \nonumber
        & \quad+ \bigchi_{(1,2)}(p)\,C\left\|A(\cdot,Du)-V(\cdot,Du) \right\|^{p}_{L^{p'}(\Omega_T)}\bigg[\displaystyle\iint_{\Omega_T}(\mu^2 + |Du|^2 )^{\frac{p}{2}}\,\dx\dt \bigg]^{2-p}, \nonumber
    \end{align}
where~$p' \coloneqq \frac{p}{p-1}$ denotes the Hölder conjugate of~$p$.
    Furthermore, there holds
    \begin{align} \label{comparison-estimate-zwei}
        \displaystyle\iint_{\Omega_T}(\mu^2 + |Dv|^2 )^{\frac{p}{2}}\,\dx\dt &\leq C \displaystyle\iint_{\Omega_T}(\mu^2
        + |Du|^2 )^{\frac{p}{2}}\,\dx\dt \\
        &  \quad+ C \left\|A(\cdot,Du)-V(\cdot,Du) \right\|^{p'}_{L^{p'}(\Omega_T)}. \nonumber
    \end{align}
\end{mylem}

\begin{mylem} \label{comparisonprinciple}
    Let~$p>1$ and~$u$,~$v$ be weak solutions to~\eqref{pdeqwieder1} under structure conditions~\eqref{voraussetzungen}. If 
$$
    u\leq v \qquad\mbox{on $\partial_p \Omega_T$,}
$$
then there holds 
$$
    u\leq v\qquad\mbox{a.e. in $\Omega_T$}.
$$
\end{mylem}
\begin{proof}
    We apply~\cite[Corollary~4.8]{comparison} with the choice~$q=1$ and note that structure conditions~\cite[$(4.5)_2$ - $(4.5)_3$]{comparison} are satisfied. The Lipschitz condition~\cite[$(4.5)_4$]{comparison} is satisfied anyways since the vector field A is independent of the function variable~$u$. Furthermore, due to the specific choice~$q=1$,~condition~\cite[$(4.5)_1$]{comparison} and subsequently~\cite[Lemma~4.6]{comparison} are redundant for the application of~\cite[Corollary~4.8]{comparison}. Additionally, the lower bound assumption in~\cite[(4.7)]{comparison} is not required, due to the lack of nonlinearity in the evolutionary term.    
\end{proof}
\begin{remark} \upshape
    The assumption~$u\leq v$ on $\partial_p \Omega_T$ is understood in the sense that~$(u-v)_+\in L^p(0,T;W^{1,p}_0(\Omega))$ and~$u(\cdot,0)\le v(\cdot,0)$ a.e. in~$\Omega$.
\end{remark}

\subsubsection{Oscillation estimates} \label{subsubsec:oscillationestimate} \,\\
We obtain two \textit{a priori} estimates for the oscillation of weak solutions to~\eqref{pdeqwieder1}. Due to structure condition~$\eqref{voraussetzungen}_1$, the subsequent lemmas follow in the very same manner as in~\cite[Lemma~10.4 \& Lemma~10.5]{gradientholder}

\begin{mylem} \label{oscillation-lemma-eins}
    Let~$p>1$ and~$u$ be a weak solution to~\eqref{pdeqwieder1} under assumptions~\eqref{voraussetzungen} with~$|Du|\in L^\infty_{\loc}(\Omega_T)$. For any~$Q=B_\rho(x_0)\times(t_1,t_2] \subset \Omega_T$ there exists~$C=C(n,C_1)>0$, such that the following oscillation estimate holds true
    \begin{align} \label{oscillationestimate-eins}
        \essosc\limits_{Q} u \leq 4 \sqrt{n}\rho \|Du\|_{L^\infty(Q)} + C \frac{t_2-t_1}{\rho}\big(\mu^2 + \|Du\|^2_{L^\infty(Q)} \big)^{\frac{p-1}{2}}.
    \end{align}
\end{mylem}
\begin{mylem} \label{oscillation-lemma-zwei}
    Let~$1<p<2$ and~$u$ be a weak solution to~\eqref{pdeqwieder1} under assumptions~\eqref{voraussetzungen} with~$|Du|\in L^{\infty}_{\loc}(\Omega_T)$. For any~$Q^{(\lambda)}_\rho(z_0) \Subset \Omega_T$ there exists~$C=C(n,C_1)>0$, such that the following oscillation estimate holds true
    \begin{align} \label{oscillationestimate-zwei}
        \essosc\limits_{Q^{(\lambda)}_\rho(z_0)} u \leq C\rho \big(\|Du\|_{L^\infty(Q^{(\lambda)}_\rho(z_0))} +\mu+\lambda\big).
    \end{align}
\end{mylem}

\subsubsection{Energy estimate}\label{subsubsec:energyestimate}\,\\
The following energy estimate will turn out to be beneficial in the determination of an upper bound for the gradient later on in Section~\ref{subsec:schauderplaplacetype}.
\begin{mylem} \label{energyestimatelemma}
    Let~$\mu\in[0,1]$,~$p>1$, and~$u$ be a weak solution to~\eqref{pdeqwieder1} under assumptions~\eqref{voraussetzungen}. For any~$\xi\in\R$ and any~$Q^{(\lambda)}_S(z_0) \Subset \Omega_T$, such that~$\frac{1}{2}S\leq R<S$, there holds 
    \begin{align} \label{energyestimateinequality}
        \esssup\limits_{t\in(t_0-\lambda^{2-p}R^2,t_0]}& \frac{\lambda^{p-2}}{R^2}\displaystyle\fint_{B_R(x_0)}|u-\xi|^2\,\dx + \displaystyle\fiint_{Q^{(\lambda)}_R(z_0)} (\mu^2 + |Du|^2 )^{\frac{p}{2}}\,\dx\dt \\
        &\leq C \displaystyle\fiint_{Q^{(\lambda)}_S(z_0)} \bigg[\frac{|u-\xi|^p}{(S-R)^p} + \lambda^{p-2} \frac{|u-\xi|^2}{S^2-R^2} + \mu^p \bigg]\,\dx\dt \nonumber
    \end{align}
    with~$C = C(n,p,C_1,C_2)$.
\end{mylem}
\begin{proof}
Instead of the weak formulation of Definition~\ref{definitionglobal} for~$u$, we choose the version in terms of Steklov-means pointed out in Section~\ref{subsec:mollificationintime}. Integrating with respect to~$t\in(0,T)$, this leads to
\begin{align*}
    \displaystyle\iint_{\Omega_T}\partial_t[u]_h \phi + \langle [A(x,t,Du)]_h, D\phi\rangle \,\dx\dt = 0
\end{align*}
for any test function~$\phi\in W^{1,p}_0$. Let~$\zeta\in C^1\big(z_0+Q^{(\lambda)}_S,[0,1]\big)$ be a cut-off function vanishing on the parabolic boundary of the sub-cylinder~$B_S(x_0)\times(t_0-S\lambda^{2-p},t_0)$, such that~$\zeta\equiv 1$ on~$z_0 +Q^{(\lambda)}_R$,~$|D\zeta|\leq \frac{2}{S-R}$ and~$|\partial_t\zeta| \leq \frac{2 \lambda^{p-2}}{S^2-R^2}$. We choose~$\epsilon>0$ small enough, such that~$t_0-\lambda^{2-p}R^2 <t_1 
<t_1+\epsilon< t_0$ holds true, and further define a Lipschitz function in time~$\psi_\epsilon~\in W^{1,\infty}\big(t_0+\Lambda^{(\lambda)}_R,[0,1]\big)$ by
\begin{align*}
    \psi_\epsilon(t) \coloneqq \begin{cases}
        1, & \mbox{for $t_0-\lambda^{2-p}S^2 \leq t \leq t_1$}\\
        1-\frac{t-t_1}{\epsilon}, & \mbox{for $t_1<t\leq t_1 +\epsilon$}\\
        0, & \mbox{for $t_1+\epsilon < t\leq t_0$}.
    \end{cases}
\end{align*}
 We test the weak form in terms of Steklov-means, given above, with~$\phi = \zeta^p\psi_\epsilon [u-\xi]_h$ and first treat the quantity involving the time derivative
\begin{align*}
    \displaystyle\iint_{\Omega_T}\partial_t[u]_h \phi\,\dx\dt &=  \displaystyle\iint_{\Omega_T}\partial_t[u-\xi]_h \phi\,\dx\dt \\
    &= \displaystyle\iint_{\Omega_T}\zeta^p\psi_\epsilon\partial_t[u-\xi]_h [u-\xi]_h\,\dx\dt \\
&= \displaystyle\iint_{\Omega_T}\frac{1}{2}\zeta^p\psi_\epsilon\partial_t[u-\xi]^2_h \,\dx\dt \\
&= - \displaystyle\iint_{\Omega_T} \frac{1}{2} (\zeta^p\psi'_\epsilon + \partial_t \zeta^p \psi_\epsilon )[u-\xi]^2_h \,\dx\dt.
\end{align*}
Passing to the limit~$h\downarrow 0$ in the preceding estimate yields
\begin{align*}
    \lim\limits_{h\downarrow 0}  \displaystyle\iint_{\Omega_T}\partial_t[u]_h \phi\,\dx\dt &= - \displaystyle\iint_{\Omega_T}\frac{1}{2} (\zeta^p\psi'_\epsilon + \partial_t \zeta^p \psi_\epsilon ) |u-\xi|^2 \,\dx\dt\\
    &=  -[ \foo{I} +  \foo{II}].
\end{align*}
Passing to the limit~$\epsilon \downarrow 0$ next, we obtain for the first term~$\foo{I}$
\begin{align*}
    \lim\limits_{\epsilon\downarrow 0} \foo{I} =  - \displaystyle\int_{B_S(x_0)\times\{t_1\}}\zeta^p |u-\xi|^2\,\dx,
\end{align*}
while for the second quantity~$\foo{II}$, we have
\begin{align*}
    \lim\limits_{\epsilon\downarrow 0} \foo{II} = \displaystyle\iint_{B_S(x_0)\times(t_0-\lambda^{2-p}S^2,t_1)}\partial_t\zeta^p |u-\xi|^2\,\dx\dt.
\end{align*}
Next, we will treat the diffusion term. By exploiting structure condition~$\eqref{voraussetzungen}_1$, Lemma~\ref{AbschätzungenfürAeins} and Lemma~\ref{lemmafurenergyestimate}, we obtain
\begin{align*}
    &\lim\limits_{h\downarrow 0}\displaystyle\iint_{\Omega_T}\langle [ A(x,t,Du) ]_h, D\phi\rangle \,\dx\dt \\
    &= \displaystyle\iint_{\Omega_T}\psi_\epsilon [\langle A(x,t,Du)-A(x,t,0), Du  \rangle \zeta^p + \langle A(x,t,0),Du\rangle\zeta^p ]\,\dx\dt \\
    & \quad+ \displaystyle\iint_{\Omega_T}\psi_\epsilon \langle A(x,t,Du), D\zeta \rangle p \zeta^{p-1}(u-\xi)\,\dx\dt
    \\
    &\geq C\displaystyle\iint_{\Omega_T}\psi_\epsilon\Big[ (\mu^2 + |Du|^2 )^{\frac{p-2}{2}}|Du|^2 \zeta^p -\mu^{p-1}(\mu^2+|Du|^2)^{\frac{1}{2}}\zeta^p\Big] \,\dx\dt \\
    & \quad - C\displaystyle\iint_{\Omega_T}p\zeta^{p-1} \psi_\epsilon (\mu^2 + |Du|^2 )^{\frac{p-1}{2}}|D\zeta| |u-\xi|\,\dx\dt \\
    &\geq C \displaystyle\iint_{\Omega_T}\psi_\epsilon \zeta^p\big[ (\mu^2 + |Du|^2 )^{\frac{p}{2}} -\mu^p\big]\,\dx\dt \\
    & \quad - C\displaystyle\iint_{\Omega_T}\psi_\epsilon |D\zeta|^p |u-\xi|^p\,\dx\dt
\end{align*}
with~$C=C(n,p,C_1,C_2)$. In the last step, we first applied Lemma~\ref{lemmafurenergyestimate}, and subsequently Young's inequality twice with exponents~$(p,\frac{p}{p-1})$ and a suitable choice of constants in order to reabsorb the term involving the spatial derivative~$Du$ twice. Passing to the limit~$\epsilon\downarrow 0$ and combining both estimates for the quantity involving the time derivative and for the diffusion term respectively, we obtain 
 \begin{align*} 
       &\displaystyle\int_{B_S(x_0)\times\{t_1\}}\zeta^p|u-\xi|^2\,\dx + \displaystyle\iint_{B_S(x_0)\times(t_0-\lambda^{2-p}S^2,t_1)} \zeta^p (\mu^2 + |Du|^2 )^{\frac{p}{2}}\,\dx\dt \\
        &\leq C \displaystyle\iint_{B_S(x_0)\times(t_0-\lambda^{2-p}S^2,t_1)}[ |D\zeta|^p |u-\xi|^p\,\dx\dt + \partial_t \zeta^p |u-\xi|^2 + \mu^p]\,\dx\dt 
    \end{align*}
with~$C=C(n,p,C_1,C_2)$. At this point, a standard argument finishes the proof. On the left-hand side we first discard the second integral and afterwards take the essential supremum with respect to~$t_1\in t_0 + \Lambda^{(\lambda)}_R$. Due to the assumptions made on the cut-off function~$\zeta$, we end up with
\begin{align*} 
        \esssup\limits_{t\in(t_0-\lambda^{2-p}R^2,t_0]}& \displaystyle\int_{B_R(x_0)}|u-\xi|^2\,\dx + \displaystyle\iint_{z_0+Q^{(\lambda)}_R} (\mu^2 + |Du|^2)^{\frac{p}{2}}\,\dx\dt \\
        &\leq C \displaystyle\iint_{z_0+Q^{(\lambda)}_S} \bigg[\frac{|u-\xi|^p}{(S-R)^p} + \lambda^{p-2} \frac{|u-\xi|^2}{S^2-R^2} + \mu^p \bigg]\,\dx\dt 
    \end{align*}
    with~$C = C(n,p,C_1,C_2)$. Finally, taking mean values on both sides, together with exploiting the assumption~$\frac{1}{2}S\leq R < S$, finishes the proof.
\end{proof}

\subsection{Schauder estimates for parabolic \texorpdfstring{$p$}{}-Laplacian type equations} \label{subsec:schauderplaplacetype} 
In this section, we will show that the gradient of any bounded weak solution~$u$ to~\eqref{pdeqwieder1} in the case~$\mu\in(0,1]$ and under the given structure conditions~\eqref{voraussetzungen}, exhibits local Hölder continuity within~$\Omega_T$. The strategy is to start by freezing the spatial variable~$x$, allowing us to utilize the \textit{a priori} gradient regularity result from~\ref{campanatoaprioritheorem} for a weak solution to a Cauchy-Dirichlet problem involving a more regular equation. This unique weak solution serves as a comparison function for the solution~$u$ itself. Exploiting the comparison estimates in~\ref{subsubsec:comparisonestimate}, the regularity of the comparison function is then transferred to the weak solution~$u$ itself. We assume~\textit{a priori} throughout this section that~$|Du|\in L^\infty_{\loc}(\Omega_T)$ holds true. The main result of this section is summarized in form of the following theorem.
\begin{mytheorem}\label{gradientholderschauder}
    Let~$u$ be a bounded weak solution to~\eqref{pdeqwieder1} under assumptions~\eqref{voraussetzungen} in the range~$p>1$ and~$\mu\in(0,1]$. Furthermore, assume that there holds~$|Du|\in L^{\infty}_{\loc}(\Omega_T)$. Then, there exist a Hölder exponent~$\alpha_{0}=\alpha_0(n,p,C_1,C_2,\alpha) \in (0,1)$ and a constant~$C=C(n,p,C_1,C_2,C_3,\alpha)\geq 1$, such that there holds
    $$Du\in C^{\alpha_0,\alpha_0/2}_{\loc}(\Omega_T,\R^n).$$
    Moreover, for any~$E\subset\Omega_T$, such that~$r\coloneqq \frac{1}{4}\dist_p(E,\partial_p \Omega_T)>0$ and any~$z_1,z_2\in E$ there hold
    the quantitative gradient estimate
    \begin{align}\label{gradientschrankeschauder}
        \esssup\limits_{E}|Du| \leq C \bigg[\frac{\omega}{r} + \Big(\frac{\omega}{r} \Big)^{\frac{2}{p}} +\mu\bigg] \eqqcolon \lambda
    \end{align}
    and the gradient Hölder estimate
\begin{align}\label{gradientholderschauderinequality}
    |Du(z_1)-Du(z_2)| \leq C \lambda \Bigg[ \frac{d^{(\lambda)}_p(z_1,z_2)}{\min\big\{1,\lambda^{\frac{p-2}{2}}\big\}r} \Bigg]^{\alpha_0},
\end{align}
where~$\omega \coloneqq \essosc\limits_{\Omega_T} u$.
\end{mytheorem}
\begin{proof} 
Instead of providing a detailed proof of Theorem~\ref{gradientholderschauder}, we will only outline the major steps since our approach closely resembles the approach taken for the prototype equation, which is performed in~\cite[10.5 Schauder estimates for Lipschitz solutions]{gradientholder}. We start by freezing the spatial variable. Let us consider nested cylinders~$Q_{\rho_1}(\hat{z}) \subset Q_{\rho_2}(\hat{z})\Subset \Omega_T$ for arbitrary radii~$0<\rho_1 < \rho_2\leq 1$. Further, let~$\lambda\geq \mu >0$, and assume that the following upper bound for the gradient is satisfied
\begin{align} \label{gradientschranke}
    \esssup\limits_{Q_{\rho_2}(\hat{z})}(\mu^2 + |Du |^2)^{\frac{1}{2}}\leq \lambda.
\end{align}
We consider a radius~$\rho\in(0,\frac{\rho_0}{8})$,
where
$$\rho_0 \coloneqq \frac{1}{2}\min\big\{1,\lambda^{\frac{p-2}{2}}\big\}(\rho_2-\rho_1).$$
This choice implies that for any~$z_0 \in Q_{\rho_1}(\hat{z})$ there holds~$Q^{(\lambda)}_{R}(z_0) \subset Q_{\rho_2}(\hat{z})$, as long as~$\frac{R}{2}\leq \rho_0$. For some~$\kappa \in (0,1)$ fixed later, we set
\begin{align} \label{radiuswahl}
    R \coloneqq \Big(\frac{8\rho}{\rho_0}\Big)^\kappa \rho_0\quad \Longleftrightarrow\quad \rho = \frac{1}{8}\Big(\frac{R}{\rho_0} \Big)^{\frac{1}{\kappa}}\rho_0,
\end{align}
which implies~$\rho<\frac{R}{8}$ and also~$R<\rho_0$. Since~$z_0$ is now fixed, we will omit the vertex~$z_0$ from all cylinders in the following. As before, let
$$v\in C\big(\Lambda^{(\lambda)}_R;L^2(B_R)\big)\cap L^p\big(\Lambda^{(\lambda)}_R;W^{1,p}(B_R)\big)$$
denote the unique weak solution to the Cauchy-Dirichlet problem
\begin{align}\label{cauchy-dirichlet-comparison-freeze}
        \begin{cases}
        \partial_t v - \divv A(x_0,t,Dv) = 0 &\,\,\mbox{in $Q^{(\lambda)}_R$,} \\[3pt]
        v=u & \,\,\mbox{on $\partial_p Q^{(\lambda)}_R$}, \\[3pt]
        \end{cases}
\end{align}
where the boundary values~$u$ are taken in the sense of Definition~\ref{definitioncauchydirichletproblem}. \,\\

Our next aim is to utilize the~\textit{a priori} estimate~\eqref{campanatoaprioriungleichung} from Theorem~\ref{campanatoaprioritheorem} and derive a Campanato type estimate for the comparison function~$v$ first. Since~$u \in L^\infty(Q^{(\lambda)}_R)$ and due to~$v=u$ on~$\partial_p Q^{(\lambda)}_R$, the comparison principle from Lemma~\ref{comparisonprinciple} implies that~$v \in L^\infty(Q^{(\lambda)}_R)$ and, moreover, yields the estimate
$$\essosc \limits_{Q^{(\lambda)}_R}v \leq \essosc \limits_{Q^{(\lambda)}_R}u \eqqcolon \omega^{(\lambda)}_R.$$
We note that structure conditions~\eqref{voraussetzungen-comparison} are satisfied for~$A(x,t,\xi)$ and~$A(x_0,t,\xi)$ due to Lemma~\ref{AbschätzungenfürAeins}. Therefore, we are in position to apply Lemma~\ref{comparison-lemma} and the mean value theorem. This yields 
\begin{align} \label{comparisonabscheins}
    \displaystyle&\iint_{Q^{(\lambda)}_R} |Du-Dv|^p\,\dx\dt \\
    &\leq 
    C \left\|A(\cdot,Du)-A(x_0,\cdot,Du) \nonumber \right\|^{p'}_{L^{p'}(Q^{(\lambda)}_R)} \\ \nonumber
        & \quad+ \bigchi_{(1,2)}(p)\, C\left\|A(\cdot,Du)-A(x_0,\cdot,Du) \right\|^{p}_{L^{p'}(Q^{(\lambda)}_R)}\bigg[\displaystyle\iint_{Q^{(\lambda)}_R}(\mu^2 + |Du|^2 )^{\frac{p}{2}}\,\dx\dt \bigg]^{2-p} \nonumber \\
        &\leq C R^{\alpha_* p} \displaystyle\iint_{Q^{(\lambda)}_R} (\mu^2 + |Du|^2)^{\frac{p}{2}}\,\dx\dt \nonumber
\end{align}
as well as the estimate
\begin{align} \label{comparisonabschzwei}
        \displaystyle\iint_{Q^{(\lambda)}_R}(\mu^2 + |Dv|^2 )^{\frac{p}{2}}\,\dx\dt &\leq C \displaystyle\iint_{Q^{(\lambda)}_R}(\mu^2
        + |Du|^2 )^{\frac{p}{2}}\,\dx\dt 
    \end{align}
with~$C=C(n,p,C_2,C_3)$, and 
\begin{align}\label{alpha*}
    \alpha_* = \begin{cases}
        \alpha, & \mbox{if $1<p< 2$} \\
        \frac{\alpha}{p-1}, & \mbox{if $p\geq 2$}.
    \end{cases}
\end{align}
Note that by defining~$\alpha_*$, we exploited the assumption~$R\leq 1$. Next, we aim to verify that condition~\eqref{schrankelambda} is satisfied for the comparison solution~$v$. Proposition~\ref{gradientboundednesssubcritical} and Proposition~\ref{gradientboundednesssupercritical} imply~$|Dv|\in L^\infty(Q^{(\lambda)}_R)$ for any~$p>1$. Furthermore, there exist quantitative estimates. In the range~$1<p<2$, Proposition~\ref{gradientboundednesssubcritical} yields for any~$\epsilon\in(0,1]$ the estimate
\begin{align}\label{gradientboundednesssubcriticalestimatecomparisonfunction}
         \esssup\limits_{Q^{(\lambda)}_{\frac{R}{4}}}|Dv| \leq C\epsilon\lambda &+ \frac{C\lambda^{\frac{1}{2}}}{\epsilon^{\theta}}\bigg[\Big(\frac{\omega^{(\lambda)}_R}{R \lambda} \Big)^{\frac{2}{p}} + \frac{\omega^{(\lambda)}_R}{R \lambda} + \frac{\mu}{\lambda} \bigg]^{\frac{n(2-p)+2p}{4p}} \\
         &\cdot \bigg[\displaystyle\fiint_{Q^{(\lambda)}_{R}}(\mu^2 + |Dv|^2)^{\frac{p}{2}}\,\dx\dt \bigg]^{\frac{1}{2p}} \nonumber
     \end{align}
     with~$\theta=\theta(n,p)>0$, and~$C=C(n,p,C_1,C_2)$. Note that~$C$ indeed is independent of structure constant~$C_6$ due to the frozen spatial variable~$x_0$ in the vector field~$A(x_0,t,\xi)$, such that assumption~$\eqref{voraussetzungendiffbarkeit}_3$ holds true with the choice~$C_6=0$. By employing the oscillation estimate from Lemma~\ref{oscillation-lemma-zwei}, together with inequality~\eqref{gradientschranke}, we have the bounds
     $$\frac{\omega^{(\lambda)}_R}{R} \leq C(n,C_1)\lambda,\qquad \frac{\mu}{\lambda}\leq 1.$$
 Thus, the reasoning above and inequality~\eqref{comparisonabschzwei} imply for any~$\epsilon\in(0,1]$ the following
\begin{align}\label{supercriticalestimatecomparisonfunction}
         \esssup\limits_{Q^{(\lambda)}_{\frac{R}{4}}}|Dv| \leq C\epsilon\lambda &+ \frac{C\lambda^{\frac{1}{2}}}{\epsilon^{\theta}}\bigg[\displaystyle\fiint_{Q^{(\lambda)}_{R}}(\mu^2 + |Dv|^2)^{\frac{p}{2}}\,\dx\dt \bigg]^{\frac{1}{2p}},
     \end{align}
where~$\theta=\theta(n,p)>0$, and~$C=C(n,p,C_1,C_2)$. In the case~$p\geq 2$, Proposition~\ref{gradientboundednesssupercritical} together with~\eqref{gradientschranke} and~\eqref{comparisonabschzwei},
 yield for any~$\epsilon\in(0,1]$ the very same estimate~\eqref{supercriticalestimatecomparisonfunction}. In what follows, we will consistently refer to equation~\eqref{supercriticalestimatecomparisonfunction} in a comprehensive manner encompassing both scenarios, where~$1<p<2$ and~$p\geq 2$. An application of~\eqref{supercriticalestimatecomparisonfunction} with the choice~$\epsilon=1$ and exploiting the bound~\eqref{gradientschranke} yields the following estimate
\begin{equation} \label{schrankecomparisonfunction}
    \esssup\limits_{Q^{(\lambda)}_{\frac{R}{4}}}(\mu^2 + |Dv|^2)^{\frac{1}{2}} \leq \mu+2C\lambda\leq C\lambda
\end{equation}
with~$C=C(n,p,C_1,C_2)\geq 1$. Due to~\eqref{schrankecomparisonfunction}, we are in position to apply the \textit{a priori} gradient estimate from Theorem~\ref{campanatoaprioritheorem} with~$L\coloneqq C$. The latter yields the bound
$$\displaystyle\fiint_{Q^{(\lambda)}_\rho}\big|Dv-(Dv)^{(\lambda)}_\rho\big|^p\,\dx\dt \leq C \Big(\frac{\rho}{R}\Big)^{p\beta}\lambda^p$$
for any $0<\rho<\frac{R}{8}$, with $\beta=\beta(n,p,C_1,C_2)\in(0,1)$, and $C=C(n,p,C_1,C_2)>0$. Combining inequality~\eqref{supercriticalestimatecomparisonfunction} and the preceding \textit{a priori} gradient estimate, this yields for any~$\epsilon\in(0,1]$ 
\begin{align} \label{campanatoestimatecomparisonfunctionextended}
    \displaystyle\fiint_{Q^{(\lambda)}_\rho}& \big|Dv-(Dv)^{(\lambda)}_\rho\big|^p\,\dx\dt \\ \nonumber
 &\leq C \Big(\frac{\rho}{R}\Big)^\beta \bigg[\epsilon^{p-1}\lambda^p + \epsilon^{-\hat{\theta}}  
 \displaystyle\fiint_{Q^{(\lambda)}_{R}}(\mu^2 + |Du|^2)^{\frac{p}{2}}\,\dx\dt \bigg]  \nonumber
\end{align}
with~$\hat{\theta}=\hat{\theta}(n,p)$, and~$C=C(n,p,C_1,C_2)$. \,\\

Subsequently, as already mentioned above, we derive a campanato type estimate for the spatial derivative of the weak solution~$u$ itself. Due to the Campanato type estimate~\eqref{campanatoestimatecomparisonfunctionextended}, we are in position to exploit the comparison inequality~\eqref{comparisonabscheins} and the comparison estimates from Lemma~\ref{comparison-lemma}. This necessitates the selection of the free parameter~$\kappa$ as
$$\kappa \coloneqq \frac{n+2+\beta}{n+2+\beta+\alpha_* p}$$
and the Hölder exponent~$\alpha_0$ as
\begin{align} \label{alpha0}
    \alpha_0 \coloneqq \frac{\alpha_* \beta}{2(n+2+\beta + \alpha_* p)}\in(0,1).
\end{align} 
Accordingly, the Hölder exponent~$\alpha_0$ only depends on the data~$n,p,C_1,C_2,\alpha$, but is independent of structure constant~$C_3$. Furthermore, we obtain for any~$\rho\in(0,\rho_0)$ and any~$\epsilon\in(0,1]$ the estimate 
\begin{align}\label{campanatoestimateuextended}
    \displaystyle\fiint_{Q^{(\lambda)}_\rho} \big|Du - (Du)^{(\lambda)}_{\rho} \big|\,\dx\dt \leq \Big(\frac{\rho}{\rho_0}\Big)^{\alpha_0 p} \bigg[\epsilon^{p-1}\lambda^p + \epsilon^{-\Tilde{\theta}p} \displaystyle\fiint_{Q^{(\lambda)}_{\rho_0}} (\mu^2 + |Du|^2 )^{\frac{p}{2}}\,\dx\dt \bigg]
\end{align}
 with~$\Tilde{\theta}>0$, and~$C=C(n,p,C_1,C_2,C_3)$.
In particular, the choice~$\epsilon=1$ and the upper bound~\eqref{gradientschranke}, together with the preceding estimate, imply
\begin{align} \label{campanatoestimateu}
    \displaystyle\fiint_{Q^{(\lambda)}_\rho} \big|Du - (Du)^{(\lambda)}_{\rho} \big|\,\dx\dt \leq C \Big(\frac{\rho}{\rho_0}\Big)^{\alpha_0 p}\lambda^p
\end{align}
with~$C=C(n,p,C_1,C_2,C_3)$. \,\\

Next, a local gradient bound for~$u$ can be obtained. First, we achieve that the limit
$$\Gamma_{z_0} \coloneqq \lim\limits_{i\to\infty}(Du)^{(\lambda)}_{r_i},$$
where~$r_i \coloneqq \frac{\rho_0}{2^i}$ and~$\rho_0$ denotes the radius from the Campanato type estimate~\eqref{campanatoestimateuextended}, exists and that for any~$\rho\in(0,\rho_0]$ and any~$\epsilon\in(0,1]$, there holds
\begin{align}\label{lebesguerepresentant}
    \big|\Gamma_{z_0} - (Du)^{(\lambda)}_r\big| \leq C \Big(\frac{\rho}{\rho_0} \Big)^{\alpha_0} \Bigg[\epsilon^{\frac{1}{p'}}\lambda + \epsilon^{-\Tilde{\theta}}\bigg[\displaystyle\fiint_{Q^{(\lambda)}_{\rho_0}}(\mu^2 + |Du|^2 )^{\frac{p}{2}}\,\dx\dt \bigg]^{\frac{1}{p}} \Bigg]
\end{align}
 with~$C=C(n,p,C_1,C_2,C_3,\alpha)$, where~$\Tilde{\theta}>0$ still denotes the parameter from estimate~\eqref{campanatoestimateuextended}. We note that the constant~$C$ additionally exhibits a dependence on the Hölder exponent~$\alpha\in(0,1)$. This reasoning shows that~$\Gamma_{z_0}$ is the Lebesgue representative of~$Du$ in~$z_0$. To obtain an upper bound for the gradient, we define 
\begin{align} \label{abstand}
r \coloneqq \frac{1}{4} \dist_{\p}(E,\partial_p\Omega_T),    
\end{align}
where~$E\subset \Omega_T$. With this choice, there holds~$Q_{2r}(\hat{z}) \Subset\Omega_T$ for any~$\hat{z}\in E$. Setting 
$$E_{\p} \coloneqq \displaystyle\bigcup\limits_{\hat{z}\in E}Q_r(\hat{z}),$$
the energy estimate from Lemma~\ref{energyestimatelemma} in turn yields the bound
\begin{equation} \label{upperboundforgradient}
\esssup\limits_{E_{\p}}|Du| < \esssup\limits_{E_{\p}}(\mu^2 + |Du|^2 )^{\frac{1}{2}} \leq C\bigg[\frac{\omega}{r} + \Big(\frac{\omega}{r} \Big)^{\frac{2}{p}} +\mu\bigg]    
\end{equation}
with~$C=C(n,p,C_1,C_2,C_3,\alpha)$, where~$\omega\coloneqq \essosc\limits_{\Omega_T} u$. Since~$E\subset E_p$, this establishes the local gradient bound~\eqref{gradientschrankeschauder}.\,\\

Finally, the local Hölder continuity of the gradient~$Du$ can be obtained, due the fact that inequalities~\eqref{gradientschranke},~\eqref{campanatoestimateu},~\eqref{lebesguerepresentant},~\eqref{upperboundforgradient}, the definition of~$r$, and the assumption~$|Du|\in L^\infty_{\loc}(\Omega_T)$ all are at our disposal. In turn, this leads to the gradient Hölder estimate~\eqref{gradientholderschauderinequality} and finishes the proof idea of Theorem~\ref{gradientholderschauder}.
\end{proof}

\subsection{Approximation:~The case \texorpdfstring{$\mu=0$}{}} \label{subsec:approximation} 
In this section, we will relinquish the assumption~$|Du|\in L^{\infty}_{\loc}(\Omega_T)$ from the previous section and treat the missing case~$\mu=0$ by an approximation argument. As a result, we derive Theorem~\ref{gradientholderschaudermugleich0}.
\begin{proof}[\textbf{\upshape Proof of Theorem~\ref{gradientholderschaudermugleich0}}]
Let the vector field~$A$ satisfy structure conditions~\eqref{voraussetzungen} in the case where~$\mu=0$. Further, let~$E\subset\Omega_T$, such that~$r\coloneqq \frac{1}{4}\dist_p(E,\partial_p \Omega_T)>0$ holds true. The idea is to approximate the vector field~$A$ by a sequence of vector fields~$A_{\epsilon_j}$ satisfying the set of assumptions~\eqref{voraussetzungen} and additionally~$\eqref{voraussetzungendiffbarkeit}_3$ for some parameter~$\epsilon_j>0$. We consider the subset~$\Tilde{E}\coloneqq\{x\in E:\,\text{dist}(x,\partial E)>\epsilon\}$ for a parameter~$0<\epsilon<r$ and take some sequence~$(\epsilon_j)_{j\in\N}\in(0,\epsilon)$ with~$\epsilon_j \downarrow 0$ as~$j\to\infty$. Let~$\phi\in C^\infty_0(B_1)$ denote a standard mollifier in space with~$\int_{B_1}\phi(x)\,\dx = 1$ and consider its scaled version~$\phi_{\epsilon_j}(\cdot)\coloneqq \epsilon^{-n}_j\phi(\frac{\cdot}{\epsilon_j})$ that is compactly supported in~$B_{\epsilon_j}$. We define the~$i$-th component of the regularized vector field~$A_{\epsilon_j}$ by
 $$A_{\epsilon_j,i}(x,t,\xi)\coloneqq \int_{B_{\epsilon_j}}A_i(y,t,\xi)\phi_{\epsilon_j}(x-y)\,\dy$$
 for a.e.~$(x,t)\in \Tilde{E}_T\coloneqq\Tilde{E}\times(0,T)$,~$i=1,...,n$,~$\xi\in\R^n$, which is finite due to the~$p$-growth assumption~$\eqref{voraussetzungen}_1$ on~$A$. The Hölder regularity with respect to the spatial variable~$x$ of~$A$ according to assumption~$\eqref{voraussetzungen}_3$ yields the estimate
 \begin{equation} \label{diffhölder}
     |A_i(x,t,\xi)-A_{\epsilon_j,i}(x,t,\xi)|\leq C_3 \epsilon^\alpha_j |\xi|^{\frac{p-1}{2}}
 \end{equation}
 for a.e.~$(x,t)\in \Tilde{E}_T$,~$\xi\in\R^n$. In order to perform the approximation, we consider unique weak solutions 
$$u_j \in C\big([0,T];L^2(\Tilde{E})\big)\cap L^p\big(0,T;W^{1,p}(\Tilde{E})\big)$$
to the Cauchy-Dirichlet problem
\begin{align*}
        \begin{cases}
        \partial_t u_j - \divv \Tilde{A}_{\epsilon_j}(x,t,Du_j) = 0 &\,\,\mbox{in $\Tilde{E}$,} \\[3pt]
        u_j=u & \,\,\mbox{on $\partial_p \Tilde{E}_T$,} \\[3pt]
        \end{cases}
\end{align*}
where the approximating vector fields are given by
\begin{align}\label{approximatingvectorfields}
\Tilde{A}_{\epsilon_j}(x,t,\xi)\coloneqq A_{\epsilon_j}(x,t,\xi) + \epsilon_j(1 + |\xi|^2 )^{\frac{p-2}{2}}\xi    
\end{align} 
for~$(x,t)\in\Tilde{E}_T$, and~$\xi\in\R^n$. The solutions~$u_j$ will serve as comparison functions for~$u$. Note that by construction,~$\Tilde{A}_{\epsilon_j}$ additionally admits a bounded derivative with respect to the spatial variable~$x$, i.e. condition~$\eqref{voraussetzungendiffbarkeit}_3$ is satisfied due to the bound
 $$|\partial_x \Tilde{A}_{\epsilon_j,i}(x,t,\xi)| \leq \frac{C_1 \|D\phi\|_{L^1}}{\epsilon_j}|\xi|^{p-1}$$
 for any~$i=1,...,n$,~$(x,t)\in\Tilde{E}_T$, and~$\xi\in\R^n$. Further, a direct calculation verifies
\begin{align*}
    \partial_\xi \Big[\epsilon_j(1 + |\xi|^2 )^{\frac{p-2}{2}}\xi \Big] =  \epsilon_j\Big[(1 + |\xi|^2 )^{\frac{p-2}{2}}\foo{I}_n + (p-2)(1 + |\xi|^2 )^{\frac{p-4}{2}}(\xi\otimes\xi)\Big].
\end{align*}
Due to our construction, the approximating vector fields~$\Tilde{A}_{\epsilon_j}$ satisfy the~$p$-growth structure conditions~\eqref{voraussetzungen} with the parameter~$\mu=0$ replaced by~$\epsilon_j$, and with positive structure constants~$\Tilde{C}_1 = \Tilde{C}_1(p,C_1),\Tilde{C}_2 = \Tilde{C}_2(p,C_2),C_3$, and a Hölder exponent~$\alpha\in(0,1)$. In particular, the ellipticity condition~$\eqref{voraussetzungen}_2$ holds true with a constant~$\Tilde{C}_2$ that is independent of the approximating parameter~$\epsilon_j$, i.e. we have
$$\langle \partial_{\xi} \Tilde{A}_{\epsilon_j}(x,t,\xi) \eta,\eta\rangle \geq \Tilde{C}(p,C_2)(\epsilon^2_j+|\xi|^2)^{\frac{p-2}{2}}|\eta|^2$$
for a.e.~$(x,t)\in\Tilde{E}_T$, and any~$\xi,\eta\in\R^n$. In fact, it is worth noting that for the treatment of the super-quadratic case where~$p> 2$, we need to replace the additional term~$\epsilon_j(1 + |\xi|^2 )^{\frac{p-2}{2}}\xi$ in the approximating vector fields~$\Tilde{A}_{\epsilon_j}$ with the quantity~$\epsilon^{p-2}_j(1 + |\xi|^2 )^{\frac{p-2}{2}}\xi$. This adjustment allows us to obtain homogeneous estimates, and through a convexity argument, it can be verified that the vector fields~$\Tilde{A}_{\epsilon_j}$ indeed satisfy the~$p$-growth structure conditions~\eqref{voraussetzungen}. However, in order to avoid a separate treatment of the cases where~$p\leq 2$ and~$p> 2$, we adopt a unified approach by only considering the vector fields~$\Tilde{A}_{\epsilon_j}$ with the quantity~$\epsilon_j(1 + |\xi|^2 )^{\frac{p-2}{2}}\xi$ as described above. Consequently, we apply the comparison estimate~\eqref{comparison-estimate-eins} from Lemma~\ref{comparison-lemma} with the vector field~$V(x,t,\xi) = \Tilde{A}_{\epsilon_j}(x,t,\xi)$ and with structure constants~$C_7=\Tilde{C}_1$,~$C_8=\Tilde{C}_2$. Together with growth condition~$\eqref{voraussetzungen}_1$, this yields for any~$j\in\N$ the following estimate
\begin{align*}
   \displaystyle\iint_{\Tilde{E}_T}|Du-Du_j|^p\,\dx\dt
        &\leq C \big(\epsilon_j^{\alpha p}+ \epsilon_j ^{\alpha p'}+\epsilon^{p}_j +\epsilon^{p'}_j \big) \displaystyle\iint_{\Tilde{E}_T}(1+|Du|^2)^{\frac{p}{2}}
\end{align*}
with~$C=C(p,C_2)$. Passing to the limit~$j\to\infty$ in the preceding estimate, we obtain
\begin{align*}
    Du_{j} \to Du\qquad\mbox{in $L^p(\Tilde{E}_T,\R^n)$\,\,as $j\to\infty$}.
\end{align*}
Furthermore, Lemma~\ref{AbschätzungenfürAeins} ensures that the assumptions of the comparison principle from Lemma~\ref{comparisonprinciple} are satisfied, and we obtain the bound
\begin{align}\label{oscillationapproximation}
    \essosc \limits_{\Tilde{E}_T}u_j \leq \essosc \limits_{\Tilde{E}_T}u<\infty\qquad\mbox{for any $j\in\N$}.
\end{align}
The preceding estimate implies~$u_j\in L^{\infty}(\Tilde{E}_T)$ for any~$j\in\N$. Thus, the gradient bounds from Proposition~\ref{gradientboundednesssubcritical} and Proposition~\ref{gradientboundednesssupercritical} are at our disposal, and we infer~$|Du_j|\in L^{\infty}(\Tilde{E}_T)$ for any~$j\in\N$. At this point, we are in position to apply Theorem~\ref{gradientholderschauder}. In particular, we obtain
\begin{align}\label{gradientschrankeapproximation}
        \esssup\limits_{E}|Du_j| \leq C \bigg[\frac{\omega_j}{r}+ \Big(\frac{\omega_j}{r} \Big)^{\frac{2}{p}}+\epsilon_j\bigg] \eqqcolon \lambda_j 
    \end{align}
    and
\begin{align}\label{hölderstetigkeitapproximation}
    |Du_j(z_1)-Du_j(z_2)| \leq C \lambda_j \Bigg[ \frac{d^{(\lambda_j)}_p(z_1,z_2)}{\min\big\{1,\lambda_j ^{\frac{p-2}{2}}\big\}r} \Bigg]^{\alpha_0}
\end{align}
for any~$z_1,z_2\in E$, with~$C=C(n,p,C_1,C_2,C_3,\alpha)\geq 1$ and~$\alpha_0=\alpha_0(n,p,C_1,C_2,\alpha)\in(0,1)$, where~$\omega_j\coloneqq  \essosc\limits_{\Omega_T} u_j$. Due to~\eqref{oscillationapproximation}, this implies
\begin{align}\label{lambdaischranke}
    \limsup\limits_{j\to\infty}\lambda_j \leq  C \bigg[\frac{\omega}{r} + \Big(\frac{\omega}{r} \Big)^{\frac{2}{p}} 
    \bigg] \eqqcolon \lambda. 
\end{align}
A similar reasoning to~\cite[10.6.~Approximation]{gradientholder} establishes that the mapping
\begin{align*}
    \lambda_j \mapsto C \lambda \Bigg[ \frac{d^{(\lambda_j)}_p(z_1,z_2)}{\min\big\{1,\lambda_j ^{\frac{p-2}{2}}\big\}r} \Bigg]^{\alpha_0}
\end{align*}
is monotonically increasing. Passing to the limit~$j\to\infty$ in~\eqref{hölderstetigkeitapproximation} and exploiting~\eqref{lambdaischranke} thus leads to
\begin{align} \label{hölderabschätzunggleichmäßig}
    \limsup\limits_{j\to\infty}|Du_j(z_1)-Du_j(z_2)| \leq C \lambda \Bigg[ \frac{d^{(\lambda)}_p(z_1,z_2)}{\min\big\{1,\lambda ^{\frac{p-2}{2}}\big\}r} \Bigg]^{\alpha_0}
\end{align}
for any~$z_1,z_2\in E$. The inequalities~\eqref{gradientschrankeapproximation},~\eqref{lambdaischranke}, and~\eqref{hölderabschätzunggleichmäßig} together show that the sequence~$(Du_j)_{j\in\N}\subset C^{\alpha_0,\alpha_0/2}(E,\R^n)$ is uniformly bounded and equicontinuous. Due to the fact that~$Du_j \to Du$ in $L^p(E,\R^n)$ as~$j\to\infty$, the theorem of Arzel\`{a}-Ascoli yields the estimates~\eqref{gradientschrankeschaudermugleich0} and~\eqref{gradientholderschauderinequalitymugleich0}. Overall, we have proven the theorem. 
\end{proof}

\section{Hölder continuity of the gradient}\label{sec:höldercontinuityofthegradient}
In this section, we finally show that the gradient of any non-negative weak solution to equations of the type
\begin{align*}
    \partial_t u^q - \divv A(x,t,Du)=0\qquad \mbox{in $\Omega_T$},
\end{align*}
assuming the structure conditions~\eqref{voraussetzungen}, is locally Hölder continuous in~$\Omega_T$, and give the proof our main regularity result that is Theorem~\ref{hölderstetigkeitohnegrößer0}. This is accomplished by utilizing the Schauder estimates for equations of parabolic~$p$-Laplacian type from Section~\ref{sec:schauderestimates} and the time-insensitive Harnack inequality from~\cite[Theorem~1.10]{gradientholder}, which is valid in the super-critical fast diffusion regime~$0<p-1<q<\frac{n(p-1)}{(n-p)_+}$.
To state the Harnack inequality, we require the following $p$-growth and coercivity assumptions
\begin{align}\label{voraussetzungenhauptresultat}
    \begin{cases}
   |A(x,t,\xi)| \leq \Tilde{C}_1|\xi|^{p-1} & \,\\
   \langle A(x,t,\xi),\xi\rangle \geq \Tilde{C}_2|\xi|^{p} & \,
    \end{cases}
\end{align}
for a.e.~$(x,t)\in\Omega_T$ and any~$\xi\in\R^n$, where~$\Tilde{C}_1, \Tilde{C}_2$ denote positive constants. 

\begin{mytheorem} [Harnack inequality] \label{Harnackinequality}
   Let~$0<p-1<q<\frac{n(p-1)}{(n-p)_+}$ and~$u$ be a non-negative and continuous weak solution to~\eqref{pde} under assumptions~\eqref{voraussetzungenhauptresultat}. Then, there exist~$\gamma>1$ and~$\sigma\in(0,1)$, both depending on~$n,p,q,\Tilde{C}_1,\Tilde{C}_2$, such that:~if $u(x_0,t_0)>0$ and the set inclusion
\begin{align}\label{harnacksetinclusion}
    B_{8\rho}(x_0) \times \big (t_0 - \mathcal{L}^{q-p+1}(8\rho)^p,t_0 + \mathcal{L}^{q-p+1}(8\rho)^p \big)\subset \Omega_T,
\end{align}
where 
$$\mathcal{L}\coloneqq \esssup\limits_{B_\rho(x_0)}u(\cdot,t_0),$$
holds true, then for any
$$(x,t)\in B_\rho(x_0)\times \big( t_0 -\sigma[u(x_0,t_0)]^{q-p+1}\rho^p,t_0 +\sigma[u(x_0,t_0)]^{q-p+1}\rho^p \big)$$
there holds
   \begin{align}\label{harnackestimate}
      \gamma^{-1}u(x_0,t_0) \leq u(x,t) \leq \gamma u(x_0,t_0). 
   \end{align}
\end{mytheorem}
\begin{remark} \upshape
    In the case~$\mu=0$, assumptions~\eqref{voraussetzungen} together with Lemma~\ref{AbschätzungenfürAeins} ensure that both conditions~\eqref{voraussetzungenhauptresultat} are always fulfilled. Hence, the Harnack inequality stated in Theorem~\ref{Harnackinequality} is at our disposal. Moreover, it is worth mentioning that the Harnack inequality in~\cite[Theorem~1.10]{gradientholder} is presented for parabolic cylinders defined with spatial cubes instead of spatial balls. However, this modification is negligible and the Harnack inequality still remains valid in the form of Theorem~\ref{Harnackinequality} for parabolic cylinders composed of spatial balls.
\end{remark}

The following proposition serves as an intermediate result on our path to Theorem~\ref{hölderstetigkeitohnegrößer0}.

\begin{myproposition} \label{hauptresultat}
     Let~$0<p-1<q<\frac{n(p-1)}{(n-p)_+}$ and~$u$ be a non-negative, continuous weak solution to~\eqref{pde} under assumptions~\eqref{voraussetzungen} with~$\mu=0$. Then, there exist~$\gamma=\gamma(n,p,q,C_1,C_2)>1$,~$C=C(n,p,q,C_1,C_2,C_3,\alpha)>1$, and~$\alpha_0=\alpha_0(n,p,q,C_1,C_2,\alpha)\in(0,1)$, such that: if~$u(x_0,t_0)>0$ and the set inclusion
\begin{align}\label{hauptresultatsetinclusion}
     B_{\gamma\rho}(x_0) \times \big (t_0 - \mathcal{L}^{q-p+1}(\gamma\rho)^p,t_0 + \mathcal{L}^{q-p+1}(\gamma\rho)^p \big)\subset \Omega_T,
\end{align}
where 
$$\mathcal{L}\coloneqq \esssup\limits_{B_\rho(x_0)}u(\cdot,t_0),$$
holds true, then we have the gradient bound
\begin{align}\label{hauptresultatgradientbound}
    \esssup\limits_{Q_{z_0}}|Du| \leq C\frac{u(x_0,t_0)}{\rho},
\end{align}
with
$$Q_{z_0}\coloneqq B_{\rho}(x_0) \times \big(t_0 - [u(x_0,t_0)]^{q-p+1}\rho^p,t_0 + [u(x_0,t_0)]^{q-p+1}\rho^p \big).$$
Furthermore, there hold the Lipschitz estimate
\begin{align} \label{hauptresultatlipschitz}
   |u(z_1)-u(z_2)| \leq C u(x_0,t_0)\Bigg[\frac{|x_1-x_2|}{\rho}+\displaystyle\sqrt{\frac{|t_1-t_2|}{[u(x_0,t_0)]^{q-p+1}\rho^p}} \,\Bigg]
\end{align}
as well as the gradient Hölder estimate
\begin{align} \label{hauptresultathölder}
    |Du(z_1)-Du(z_2)| \leq C\frac{ u(x_0,t_0)}{\rho}\Bigg[\frac{|x_1-x_2|}{\rho}+\displaystyle\sqrt{\frac{|t_1-t_2|}{[u(x_0,t_0)]^{q-p+1}\rho^p}} \,\Bigg]^{\alpha_0},
\end{align}
for any~$z_1=(x_1,t_1)$,~$z_2 = (x_2,t_2)\in Q_{z_0}$.
\end{myproposition}


\begin{remark}\upshape
The continuity assumption on~$u$ in Proposition~\ref{hauptresultat} is only required to ensure the well-definedness of the pointwise evaluation~$u(x_0,t_0)$. However, this requirement is not restrictive, as any locally bounded weak solution to~\eqref{pde} under assumptions~\eqref{voraussetzungen} admits an upper semi-continuous representative~$u^*$ with~$u=u^*$ a.e. in~$\Omega_T$. A reference for this result can be found in~\cite[Theorem~2.3]{liao2021regularity}.
\end{remark}
\begin{remark}\upshape \label{remnachhr}
    The bounds provided in~\eqref{hauptresultatgradientbound},~\eqref{hauptresultatlipschitz}, and~\eqref{hauptresultathölder} are analogous to those stated in~\cite[Theorem~1.1]{gradientholder}. However, the set inclusion condition~\eqref{hauptresultatsetinclusion} assumed in Proposition~\ref{hauptresultat} is more stringent than its counterpart for the prototype equation~\eqref{pdemitplaplace}. This difference arises from the more general structure of~\eqref{pde}, and despite the greater constraint, it is a natural requirement. 
    In fact, it is noteworthy that the assumption~\eqref{hauptresultatsetinclusion} is a consequence of the time-insensitive Harnack inequality stated in form of Theorem~\ref{Harnackinequality}, while the same Harnack inequality for the prototype equation~\eqref{pdemitplaplace} in~\cite[Theorem~1.11]{gradientholder} applies under a weaker condition. 
\end{remark}


\begin{remark} \upshape
    The optimality of the range of exponents in Proposition~\ref{hauptresultat} can be inferred from~\cite[Section~11.1]{gradientholder}. In particular, the statement of the proposition breaks down in the borderline cases~$q=p-1$ and~$q=\frac{n(p-1)}{(n-p)_+}$, which essentially arises due to the fact that the Harnack inequality from Theorem~\ref{Harnackinequality} holds true solely within the specified range of parameters.
\end{remark}

\begin{proof}[\textbf{\upshape Proof of Proposition~\ref{hauptresultat}}]
To begin with, we refer to Corollary 1.7 in~\cite{gradientholder}, which implies that~$u$ is locally bounded in~$\Omega_T$. Consider~$z_0=(x_0,t_0)\in\Omega_T$ with~$u(x_0,t_0)>0$. Let~$\sigma\in(0,1)$ and~$\gamma>1$ both depending on~$n,p,q,C_1,C_2$, denote the constants from Theorem~\ref{Harnackinequality} and set~$\Tilde{\gamma}\coloneqq \frac{8}{\sigma}>1$. Further, consider~$\rho>0$, such that the set inclusion
$$B_{\Tilde{\gamma}\rho}(x_0) \times(t_0 - \mathcal{L}^{q-p+1}(\Tilde{\gamma}\rho)^p,t_0 + \mathcal{L}^{q-p+1}(\Tilde{\gamma}\rho)^p )\subset \Omega_T$$
holds true. In fact, due to~$u$ being locally bounded and therefore~$\mathcal{L}$ being finite, this assumption is justifiable for~$\rho>0$ small enough. The choices above imply the set inclusion
$$B_{8\rho/\sigma^{\frac{1}{p}}}(x_0) \times \big (t_0 - \mathcal{L}^{q-p+1}\big[8\rho/\sigma^{\frac{1}{p}}\big]^p,t_0 + \mathcal{L}^{q-p+1}\big[8\rho/\sigma^{\frac{1}{p}}\big]^p \big)\subset \Omega_T.$$
Moreover, due to the Harnack inequality from Theorem~\ref{Harnackinequality}, the bound~\eqref{harnackestimate} holds true a.e. in
$$B_{\rho/\sigma^{\frac{1}{p}}}(x_0) \times \big (t_0 - \sigma[u(x_0,t_0) ]^{q-p+1}\big[\rho/\sigma^{\frac{1}{p}}\big]^p,t_0 + \sigma[u(x_0,t_0)]^{q-p+1}\big[\rho/\sigma^{\frac{1}{p}}\big]^p \big).$$
Furthermore, we have
\begin{align*}
    &\quad B_{\rho}(x_0) \times \big (t_0 - [u(x_0,t_0)]^{q-p+1}\rho^p,t_0 + [u(x_0,t_0)]^{q-p+1}\rho^p \big) \\
    &\subset B_{\rho/\sigma^{\frac{1}{p}}}(x_0) \times \big (t_0 - [u(x_0,t_0)]^{q-p+1}\rho^p,t_0 + [u(x_0,t_0)]^{q-p+1}\rho^p \big) \\
    &= B_{\rho/\sigma^{\frac{1}{p}}}(x_0) \times \big (t_0 - \sigma[u(x_0,t_0)]^{q-p+1}\big[\rho/\sigma^{\frac{1}{p}}\big]^p,t_0 + \sigma[u(x_0,t_0)]^{q-p+1}\big[\rho/\sigma^{\frac{1}{p}}\big]^p \big),
\end{align*}
which allows an application of the Harnack inequality on $Q_{z_0}$. Therefore, we achieve
\begin{align}\label{schrankeaufQz0}
    \gamma^{-1}u(x_0,t_0) \leq u(x,t) \leq \gamma u(x_0,t_0)
\end{align}
for any~$(x,t)\in Q_{z_0}$. To relinquish the quantity~$u(x_0,t_0)$ in~\eqref{schrankeaufQz0} and obtain bounds independent of the pointwise value of~$u$, we re-scale to the unit cylinder~$\mathcal{Q}_1 \coloneqq B_1(0)\times(-1,1)\subset \R^{n+1}$ and consider
\begin{align*}
  \hat{u}(y,s)\coloneqq & {\textstyle \frac{1}{u(x_0,t_0)}}u\big(x_0+\rho y,t_0+[u(x_0,t_0)]^{q-p+1}\rho^p s \big),\quad
 (y,s)\in \mathcal{Q}_1.
\end{align*}
Then,~$\hat{u}$ is a weak solution to the doubly nonlinear equation
\begin{align}\label{transformedpde}
    \partial_t\hat{u}^q - \divv \hat{A}(y,s,D\hat{u})=0 \quad\mbox{in $\mathcal{Q}_1$},
\end{align}
where we defined 
\begin{align}\label{transformedvectorfield}
    \hat{A}(y,s,\xi) \coloneqq A(x_0+\rho y,t_0+[u(x_0,t_0)]^{q-p+1}\rho^p s,\xi)
\end{align}
for~$(y,s)\in \mathcal{Q}_1$ and~$\xi\in\R^n$. Hence, there holds
\begin{equation} \label{schrankeuhut}
    \gamma^{-1} \leq \hat{u}(y,s)\leq \gamma
\end{equation}
for any~$(y,s)\in \mathcal{Q}_1$. With the Schauder estimates considered in the preceding section in mind, we proceed to replace~$\hat{u}^q$ with~$\Tilde{u}$. Therefore, equation~\eqref{transformedpde} transforms to
\begin{align}\label{transformedpdeqgleicheins}
    \partial_t \Tilde{u} - \divv \Tilde{A}(y,s,D\Tilde{u})=0 \quad\mbox{in $\mathcal{Q}_1$}
\end{align}
with vector field
\begin{align}\label{transformedvectorfieldqgleicheins}
    \Tilde{A}(y,s,\xi) \coloneqq \hat{A}\big(y,s,q^{-1}\Tilde{u}(y,s)^{\frac{1-q}{q}}\xi\big).
\end{align}
By employing the abbreviation
$$a(y,s)\coloneqq q^{-1}\Tilde{u}(y,s)^{\frac{1-q}{q}}$$
and utilizing the bound~\eqref{schrankeuhut}, we have the following lower and upper bound for the coefficients
\begin{equation*}
   q^{-1}\gamma^{-|1-q|} \leq a(y,s) \leq q^{-1} \gamma^{|1-q|}\qquad\mbox{for any~$(y,s)\in \mathcal{Q}_1$.}
\end{equation*}
Consequently, we are in position to apply a classical result found in~\cite[Chapter~III,~§1.,~Theorem~1.1]{dibenedetto1993degenerate}, which asserts that~$v$ is locally Hölder continuous in~$\mathcal{Q}_1$ with some Hölder exponent~$\beta\in(0,1)$ that depends on~$n,p,q,C_1,C_2$. Due to the upper and lower bounds for the vector field~$A$, it readily follows that the structure conditions~$\eqref{voraussetzungen}_1$ --~$\eqref{voraussetzungen}_2$ are once again satisfied, with certain positive constants~$\Tilde{C}_1,\Tilde{C}_2$. Note that~$\Tilde{C}_1$ depends only on~$\gamma,q$, while~$\Tilde{C}_2$ depends on~$\gamma, q$. Furthermore, the local Hölder continuity of~$v$ in~$\mathcal{Q}_1$, together with the~$p$-growth condition~$\eqref{voraussetzungen}_1$, implies that the vector field~$A$ also satisfies~$\eqref{voraussetzungen}_3$ in~$\mathcal{Q}_{\frac{1}{2}}\coloneqq B_{\frac{1}{2}}(0)\times(-\frac{1}{4},\frac{1}{4})$ with a constant~$\Tilde{C}_3$ that depends on~$\gamma,\Tilde{C}_1,\Tilde{C}_2,C_3$, and with a Hölder exponent~$\Tilde{\alpha}\in(0,1)$ depending on~$n,p,q,\Tilde{C}_1,\Tilde{C}_2,\alpha$. As a result, we are now in position to apply Theorem~\ref{gradientholderschaudermugleich0}, which yields
$$D\Tilde{u}\in C^{\alpha_0,\alpha_0/ 2}\big(\mathcal{Q}_\frac{1}{2},\R^n\big)$$
for some~$\alpha_0=\alpha_0(n,p,\Tilde{C}_1,\Tilde{C}_2,\alpha)\in(0,1)$. Due to the dependence of the constants~$\Tilde{C}_1,\Tilde{C}_2$ there holds~$\alpha_0 =\alpha_0(n,p,q,C_1,C_2,\alpha)$. Furthermore, Theorem~\ref{gradientholderschaudermugleich0} yields quantitative estimates: there exists a constant~$C=C(n,p,q,C_1,C_2,C_3,\alpha)\geq 1$ and a Hölder exponent~$\alpha_0=\alpha_0(n,p,q,C_1,C_2,\alpha)\in(0,1)$, such that for any~$E\subset\mathcal{Q}_{\frac{1}{2}}$ with~$r\coloneqq \frac{1}{4}\dist_{\p}\big(E,\mathcal{Q}_{\frac{1}{2}} \big)>0$, the following quantitative estimates
  \begin{align*}
        \esssup\limits_{E}|D\Tilde{u}| \leq C \Big[\Big(\frac{\omega}{r}\Big) + \Big(\frac{\omega}{r} \Big)^{\frac{2}{p}}\Big] \eqqcolon \lambda 
    \end{align*}
    and
\begin{align*}
    |D\Tilde{u}(\Tilde{z}_1)-D\Tilde{u}(\Tilde{z}_2)| \leq C \lambda \Bigg[ \frac{d^{(\lambda)}_p(z_1,z_2)}{\min\big\{1,\lambda^{\frac{p-2}{2}}\big\}r} \Bigg]^{\alpha_0}
\end{align*}
hold true, where~$\omega \coloneqq \essosc\limits_{\mathcal{Q}_{\frac{1}{2}}}\Tilde{u}$, $\Tilde{z}_1=(\Tilde{x}_1,\Tilde{t}_1)$,~$\Tilde{z}_2=(\Tilde{x}_2,\Tilde{t}_2)\in E$. We now choose~$E=\mathcal{Q}_{\frac{1}{4}}=B_{\frac{1}{4}}(0)\times(-\frac{1}{16},\frac{1}{16})$. In turn, we have~$r=\frac{\sqrt{3}}{16}$ and there holds~$\omega\leq \gamma^q$. This implies
\begin{align*}
    \lambda = C\Big[\Big(\frac{\omega}{r}\Big) + \Big(\frac{\omega}{r} \Big)^{\frac{2}{p}}\Big] \leq C
\end{align*}
 and moreover
\begin{align} \label{schranketildeu}
    \esssup\limits_{\mathcal{Q}_{\frac{1}{4}}}|D\Tilde{u}| \leq C
\end{align}
with~$C=C(n,p,q,C_1,C_2,C_3,\alpha)$. Furthermore, there holds
\begin{align} \label{hölderstetigkeittildeu}
     |D\Tilde{u}(\Tilde{z}_1)-D\Tilde{u}(\Tilde{z}_2)| \leq C\d_{\p}(\Tilde{z}_1,\Tilde{z}_2)^{\alpha_0}
\end{align}
for any~$\Tilde{z}_1$,~$\Tilde{z}_2\in\mathcal{Q}_{\frac{1}{4}}$, with~$C=C(n,p,q,C_1,C_2,C_3,\alpha)$. The previous estimate is obvious in the case~$p\leq 2$, whereas in the case~$p>2$ we utilized~$\alpha_0\leq\frac{2}{p-2}$, which is an immediate consequence of the definition of~$\alpha_0$ according to~\eqref{alpha0}. To proceed, we consider two points~$\Tilde{z}_1$,~$\Tilde{z}_2 \in \mathcal{Q}_{\frac{1}{8}} = B_{\frac{1}{8}}(0)\times(-\frac{1}{64},\frac{1}{64})$ with~$\Tilde{t}_1\leq\Tilde{t}_2$, and a cylinder~$Q= Q_{\rho}(\overline{x})\times(\Tilde{t}_1,\Tilde{t}_2]$, such that~$\overline{x} = \frac{1}{2}(\Tilde{x}_1 + \Tilde{x}_2)$ and~$\frac{1}{2}|\Tilde{x}_1-\Tilde{x}_2|\leq \rho \leq \frac{1}{8}$. Now, an application of Lemma~\ref{oscillation-lemma-eins} with the choice~$\mu=0$ on the cylinder~$Q$, together with~\eqref{schranketildeu}, yield the following
\begin{align*} 
       |\Tilde{u}(\Tilde{z}_1) - \Tilde{u}(\Tilde{z}_2)| &\leq \essosc\limits_{Q} \Tilde{u} \\
       &\leq C \bigg[\rho\|D\Tilde{u}\|_{L^\infty(Q)} + C \frac{\Tilde{t}_2 - \Tilde{t}_1}{\rho}\|D\Tilde{u}\|^{p-1}_{L^\infty(Q)}\bigg] \\
       &\leq C \bigg[\rho + \frac{\Tilde{t}_2 - \Tilde{t}_1}{\rho} \bigg]
\end{align*}
with~$C=C(n,p,q,C_1,C_2,C_3,\alpha)$. Distinguishing between both cases~$\sqrt{\Tilde{t}_2 - \Tilde{t}_1} \leq |\Tilde{x}_1 - \Tilde{x}_2|$ and~$\sqrt{\Tilde{t}_2 - \Tilde{t}_1} > |\Tilde{x}_1 - \Tilde{x}_2|$, this overall leads to the bound
$$|\Tilde{u}(\Tilde{z}_1)-\Tilde{u}(\Tilde{z}_2)| \leq C \Big[|\Tilde{x}_1 - \Tilde{x}_2| + \sqrt{|\Tilde{t}_1-\Tilde{t}_2|} \Big]$$
for any~$\Tilde{z}_1$,~$\Tilde{z}_2\in\mathcal{Q}_{\frac{1}{8}}$, with~$C=C(n,p,q,C_1,C_2,C_3,\alpha)$. Exploiting~\eqref{transformedvectorfieldqgleicheins} and Lemma~\ref{lem:a-b}, we establish
\begin{align} \label{hochbeta}
    |\Tilde{u}^\beta (\Tilde{z}_1) - \Tilde{u}^\beta (\Tilde{z}_2)| \leq C \gamma^{q|\beta-1|}\Big[|\Tilde{x}_1 - \Tilde{x}_2| + \sqrt{|\Tilde{t}_1-\Tilde{t}_2|} \Big]
\end{align}
for any~$\beta\in\R$, and any~$\Tilde{z}_1$,~$\Tilde{z}_2\in\mathcal{Q}_{\frac{1}{8}}$, with~$C=C(n,p,q,C_1,C_2,C_3,\alpha,\beta)$. By following the final steps performed in~\cite[Proof of Theorem 1.1]{gradientholder}, we transform back to the original solution~$u$, which yields the claimed estimates~$\eqref{hauptresultatgradientbound}-\eqref{hauptresultathölder}$. It is worth mentioning that the dependence on~$\beta$ of~$C$ is eliminated at the end by the specific choices~$\beta = \frac{1}{q}$ and~$\beta = \frac{1-q}{q}$.
\end{proof}

Finally, we are in position to provide the proof of Theorem~\ref{hölderstetigkeitohnegrößer0}.

\begin{proof}[\textbf{\upshape Proof of Theorem~\ref{hölderstetigkeitohnegrößer0}}]
Let~$\Tilde{\gamma}>1$ denote the constant from Proposition~\ref{hauptresultat}. As~$u$ is locally bounded in~$\Omega_T$, there holds~$\esssup\limits_{\mathcal{K}_2}u<\infty$. We now choose the radius~$\rho \coloneqq \frac{\Tilde{K}^{-\frac{q-p+1}{p}}\rho_1}{4\Tilde{\gamma}}$, which implies
\begin{align*}
    \esssup\limits_{B_\rho(x)}u(\cdot,t) \leq \esssup\limits_{B_\frac{\rho_1}{4}(x)}u(\cdot,t) \leq \esssup\limits_{\mathcal{K}_2}u \leq K_2
\end{align*}
for any~$(x,t)\in\mathcal{K}_1$. Additionally, we have the set inclusion
\begin{align*}
    &B_{\Tilde{\gamma}\rho}(x) \times (t-\Tilde{K}^{q-p+1}(\Tilde{\gamma}\rho)^p,t+ \Tilde{K}^{q-p+1}(\Tilde{\gamma}\rho)^p) \\
    &\subset B_{\frac{\rho_1}{4}(x)} \times \Big(t-\Big(\frac{\rho_1}{4}\Big)^p,t+ \Big(\frac{\rho_1}{4}\Big)^p\Big)\subset \mathcal{K}_2 \Subset \Omega_T
\end{align*}
for any~$(x,t)\in\mathcal{K}_2$. Let us now fix a point~$z_0=(x_0,t_0)\in\mathcal{K}_1$. If~$u(x_0,t_0)>0$, the preceding inequalities enable us to apply the gradient bound~\eqref{hauptresultatgradientbound} from Proposition~\ref{hauptresultat} on the cylinder
$$Q_{z_0} \coloneqq B_{\rho}(x_0) \times \big(t_0-[u(x_0,t_0)]^{q-p+1}\rho^p,t_0+[u(x_0,t_0)]^{q-p+1}\rho^p \big)$$
to obtain
\begin{align}\label{korollargradientboundgrößer0}
    |Du(x_0,t_0)| \leq C \frac{u(x_0,t_0)}{\rho} \leq C \frac{K_1}{\rho} = C \frac{K_1 K_2^{\frac{q-p+1}{p}}}{\rho_1}
\end{align}
with~$C=C(n,p,q,C_1,C_2,C_3,\alpha)$. If~$u(x_0,t_0)=0$, then the Harnack inequality from Theorem~\ref{Harnackinequality} implies~$u(\cdot,t_0)\equiv 0$ in $\Omega$, and in particular~$Du(x_0,t_0)=0$. Since~$z_0$ was chosen arbitrarily, by combining the calculations from both cases, we obtain the stated gradient bound~\eqref{korollargradientbound}. Following the strategy in~\cite[Proof of Corollary~1.3]{gradientholder}, we also obtain the local gradient Hölder estimate~\eqref{korollarhölder}, which finishes the proof.
\end{proof}

 \nocite{*}
\bibliographystyle{plain}
\bibliography{Literatur.bib}

\end{document}